\scrollmode
\documentclass[12pt]{article}

\usepackage{amsmath}				%
\usepackage{mathtools}			%
\usepackage{amsthm,amssymb,wasysym}		%
\usepackage{amsfonts}			%
\usepackage[numbers,square]{natbib}	%
\usepackage[babel]{csquotes}
\usepackage[ngerman, english]{babel}
\usepackage{caption}
\usepackage{lmodern}
\usepackage{changepage}
\usepackage{hyperref}
\usepackage{url}
\usepackage{tocloft}
\usepackage{wrapfig}

\usepackage{amssymb}
\usepackage{amsthm}
\usepackage{amsmath}
\usepackage{extarrows}

\usepackage{epsfig}
\usepackage{float,floatflt,afterpage}
\usepackage{color}
\usepackage{csquotes}
\renewcommand{\thefootnote}{\fnsymbol{footnote}}

\usepackage{tikz}
\usepackage{pgfplots}
\usepackage{wrapfig}

\newcommand{\imagepath}{./}

\newcounter{Breferencecounter}
\newcounter{tempcounter}

\newcommand{\WBecker}{W_{\mathrm{Becker}}^{\nu=0}}
\newcommand{\WBeckerhat}{{\widehat W}_{\mathrm{Becker}}^{\nu=0}}
\newcommand{\tauHencky}{\tau_{\mathrm{H}}}
\newcommand{\tauBecker}{\tau_{\mathrm{B}}}
\newcommand{\BiotBecker}{\Biot_{\mathrm{B}}}
\newcommand{\Biothat}{\widehat{T}^{\mathrm{Biot}}}
\newcommand{\Uhat}{\widehat{U}}
\newcommand{\alphahat}{\hat{\alpha}}
\newcommand{\hhat}{\hat{h}}
\newcommand{\iprod}[1]{<#1>}
\newcommand{\stress}{T}
\renewcommand{\stretch}{{\mathcal{E}}}

\newcommand{\vect}[3]{\matr{#1\\#2\\#3}}
\newcommand{\vects}[3]{\matrs{#1\\#2\\#3}}
\newcommand{\iplane}{\mathcal{E}}
\newcommand{\fplane}{\mathcal{F}}

\newcommand{\dmatr}[3]{\matr{#1&0&0\\ 0&#2&0\\ 0&0&#3}}	%
\newcommand{\dmatrs}[3]{\matrs{#1&0&0\\ 0&#2&0\\ 0&0&#3}}	%
\newcommand{\signdmatr}[3]{\signmatr{#1&0&0\\ 0&#2&0\\ 0&0&#3}}	%

\newcommand{\axiombox}[2]{%
\begin{center}%
\fbox{\parbox{380\unitlength}{%
\textbf{#1}\\[4mm]
#2%
}}
\end{center}
}

\newcommand{\upperPhant}{\vphantom{\int^1}}

\usepackage{titling}
\newcommand{\Boriginal}[1]{#1}
\newcommand{\Bedited}[1]{}

\newcommand{\Bfootnote}[1]{}		%

\newcommand{\Blabel}[1]{%
\setcounter{tempcounter}{\number\value{footnote}}%
\setcounter{footnote}{\number\value{Breferencecounter}}%
\renewcommand{\thefootnote}{\arabic{footnote}}%
\refstepcounter{Breferencecounter}%
\footnotemark[\number\value{Breferencecounter}]%
\label{#1}%
\renewcommand{\thefootnote}{\fnsymbol{footnote}}%
\setcounter{footnote}{\number\value{tempcounter}}%
}
\newcommand{\Bref}[2][]{(#1\ref{#2}, p. \pageref{#2})}

\newcommand{\Bquote}[1]{{\em\enquote{#1}}}
\newcommand{\setcolumnlength}[1]{
	\ifdefined\columnlength
	\else
		\newlength{\columnlength}
	\fi
	\setlength{\columnlength}{#1}
}
\newcommand{\co}[1]{\ \hbox to \columnlength{#1 \hfill}}
\newcommand{\twoco}[2]{\ \hbox to \columnlength{#1 \hfill} \hbox to \secondcolumnlength{#2 \hfill}}
\newcommand{\setsecondcolumnlength}[1]{
	\ifdefined\secondcolumnlength
	\else
		\newlength{\secondcolumnlength}
	\fi
	\setlength{\secondcolumnlength}{#1}
}

\DeclareMathOperator{\SL}{SL}
\DeclareMathOperator{\opsl}{\mathfrak{sl}}
\renewcommand{\sl}{\opsl}
\DeclareMathOperator{\GL}{GL}
\DeclareMathOperator{\GLp}{GL^+}
\DeclareMathOperator{\SO}{SO}
\DeclareMathOperator{\OO}{O}
\DeclareMathOperator{\PSym}{PSym}
\DeclareMathOperator{\Sym}{Sym}
\DeclareMathOperator{\so}{\mathfrak{so}}
\DeclareMathOperator{\gl}{\mathfrak{gl}}
\newcommand{\R}{\mathbb{R}}

\newcommand{\N}{\mathbb{N}}
\newcommand{\Q}{\mathbb{Q}}
\newcommand{\Z}{\mathbb{Z}}

\newcommand{\Rnn}{\Rmat{n}{n}}

\newcommand{\GLpn}{\GLp(n)}

\newcommand{\SOn}{\SO(n)}
\newcommand{\On}{\OO(n)}

\newcommand{\Symn}{\Sym(n)}
\newcommand{\PSymn}{\PSym(n)}
\newcommand{\range}{\operatorname{range}}
\renewcommand{\div}{\operatorname{div}}

\DeclareMathOperator{\dev}{dev}
\DeclareMathOperator{\tr}{tr}
\newcommand{\id}{ {1\!\!\!\:1 } }
\newcommand{\matr}[1]{\begin{pmatrix} #1 \end{pmatrix}}
\newcommand{\signmatr}[1]{\begin{pmatrix*}[r] #1 \end{pmatrix*}}
\newcommand{\matrs}[1]{\left(\begin{smallmatrix} #1 \end{smallmatrix}\right)}
\newcommand{\inv}{^{-1}}
\newcommand{\norm}[1]{\Vert #1 \Vert}

\newcommand{\Bignorm}[1]{\Big\Vert #1 \Big\Vert}
\newcommand{\innerproduct}[1]{\langle #1 \rangle}
\DeclareMathOperator{\Cof}{Cof}
\DeclareMathOperator{\diag}{diag}
\DeclareMathOperator{\sym}{sym}

\newcommand{\afrac}[2]{#1\! /\! #2}
\newcommand{\tel}[1]{\frac{1}{#1}}
\newcommand{\half}{\tel{2}}
\newcommand{\eps}{\varepsilon}
\newcommand{\grad}{\nabla}
\newcommand{\nnl}{\nonumber\\}
\newcommand{\eqq}{\;=\;}

\renewcommand{\GLpn}{\GLp(3)}
\renewcommand{\Symn}{\Sym(3)}
\renewcommand{\PSymn}{\PSym(3)}
\renewcommand{\On}{\OO(3)}
\renewcommand{\SOn}{\SO(3)}
\renewcommand{\Rnn}{\R^{3\times3}}

\newcommand{\xhat}{\hat{x}}

\newcommand{\What}{\widehat W}
\newcommand{\Wtilde}{\widetilde W}

\newcommand{\lambdahat}{\hat \lambda}

\renewcommand{\GLpn}{\GLp(3)}
\renewcommand{\Symn}{\Sym(3)}
\renewcommand{\PSymn}{\PSym(3)}
\renewcommand{\On}{\OO(3)}
\renewcommand{\SOn}{\SO(3)}
\renewcommand{\Rnn}{\R^{3\times3}}

\theoremstyle{plain}
\newcounter{theoremCounter}
\numberwithin{theoremCounter}{section}
\newtheorem{lemma}[theoremCounter]{Lemma}
\newtheorem{proposition}[theoremCounter]{Proposition}
\newtheorem{corollary}[theoremCounter]{Corollary}

\theoremstyle{definition}

\newtheorem{remark}[theoremCounter]{Remark}

\renewcommand{\cos}{\,{\rm{cos}\,} }
\renewcommand{\tan}{\,{\rm{tan}\,} }
\renewcommand{\cot}{\,{\rm{cot}\,} }
\newcommand{\arccot}{\,{\rm{arccot}} }
\renewcommand{\arccos}{\,{\rm{arccos}} }

\makeatletter\@addtoreset{footnote}{page}
\makeatother

\def\Young{E}
\def\Poisson{\nu}
\def\bulk{K}
\def\shear{G}
\def\PKone{S_1}
\def\PKtwo{S_2}
\def\Cauchy{\sigma}
\def\Biot{T^{\mathrm{Biot}}}
\def\varYoung{{E^*}}
\def\e{e}
\def\Nu-xi{\xi}
\def\Kappa-eta{\eta}
\def\Bresultant{\mathcal{R}}
\def\Bnormal{\mathcal{N}}
\def\Btangent{\mathcal{T}}

\begin{document}

\changepage{-15mm}{25mm}{-10mm}{-5mm}{0mm}{-10mm}{0mm}{0mm}{0mm}
\thanksmarkseries{arabic}

\title{Rediscovering G.F. Becker's early axiomatic deduction of a multiaxial nonlinear stress-strain relation based on logarithmic strain}
\author{Patrizio Neff\thanks{Corresponding author.\: Head of Chair for Nonlinear Analysis and Modelling, University of Duisburg-Essen, Thea-Leymann-Str. 9, 45127 Essen, Germany, email: patrizio.neff@uni-due.de}, \;Ingo M\"unch\thanks{Institute for Structural Analysis, Karlsruhe Institute of Technology, Kaiserstr. 12, 76131 Karlsruhe, Germany, email: ingo.muench@kit.edu} \;\,and\; Robert Martin\thanks{Chair for Nonlinear Analysis and Modelling, University of Duisburg-Essen, Thea-Leymann-Str. 9, 45127 Essen, Germany, email: robert.martin@stud.uni-due.de}%
\vspace*{10mm}\\%
{\normalsize Dedicated to Ray Ogden on the occasion of his 70\textsuperscript{th} birthday.\vspace*{5mm}}%
}

\maketitle
\thispagestyle{empty}
\begin{abstract}
We discuss a completely forgotten work of the geologist G.F. Becker on the ideal isotropic nonlinear stress-strain function \cite{becker1893}. In doing this we provide the original paper from 1893 newly typeset in {\LaTeX} and with corrections of typographical errors as well as an updated notation.
Due to the fact that the mathematical modelling of elastic deformations has evolved greatly since the original publication we give a modern reinterpretation of Becker's work, combining his approach with the current framework of the theory of nonlinear elasticity.

Interestingly, Becker introduces a multiaxial constitutive law incorporating the logarithmic strain tensor, more than 35 years before the quadratic Hencky strain energy was introduced by Heinrich Hencky in 1929. Becker's deduction is purely axiomatic in nature. He considers the finite strain response to applied shear stresses and spherical stresses, formulated in terms of the principal strains and stresses, and postulates a principle of superposition for principal forces which leads, in a straightforward way, to a unique invertible constitutive relation, which in today's notation can be written as
\begin{equation*}
\begin{aligned}
	\Biot &= 2\,G\cdot\dev_3\log(U) + \bulk\cdot\tr[\log(U)]\cdot\id\\
	&= 2\,G\cdot\log(U) + \Lambda\cdot\tr[\log(U)]\cdot\id\,,
\end{aligned}
\end{equation*}
where $\Biot$ is the Biot stress tensor, $\log(U)$ is the principal matrix logarithm of the right Biot stretch tensor $U=\sqrt{F^TF}$, $\tr X = \sum_{i=1}^3 X_{i,i}$ denotes the trace and $\dev_3 X = X - \frac13 \tr(X) \cdot \id$ denotes the deviatoric part of a matrix $X\in\R^{3\times3}$.

Here, $G$ is the shear modulus and $K$ is the bulk modulus. For Poisson's number $\nu=0$ the formulation is hyperelastic and the corresponding strain energy
\[
	\WBecker(U) = 2\,G\; [\,\iprod{U,\,\log(U) - \id} + 3\,]
\]
has the form of the maximum entropy function.
\end{abstract}

\newpage
\tableofcontents
\thispagestyle{empty}

\newpage
\changepage{15mm}{-25mm}{10mm}{5mm}{0mm}{10mm}{0mm}{0mm}{0mm}

\renewcommand{\ln}{\log}
\newpage
\markboth{}{}
\changetext{-25mm}{35mm}{-15mm}{-10mm}{0mm}
\newif\ifshowimages\showimagesfalse
\newif\ifshowimages\showimagestrue
\newpage
\captionsetup{justification=raggedright}
\captionsetup{format=hang}
\section{Introduction}
\subsection{Some reflections on constitutive assumptions in nonlinear elasticity}
\label{section:reflectionsOnConstitutiveAssumptions}

The question of proper constitutive assumptions in nonlinear elasticity has puzzled many generations of researchers.
The problem of finding simple enough constitutive assumptions which are sufficient to characterize a physically plausible behaviour of \emph{\enquote{completely elastic}} materials was even called \emph{\enquote{das ungel\"oste Hauptproblem der endlichen Elastizit\"atstheorie}} (the unsolved main problem of finite elasticity theory) by C. Truesdell \cite{truesdell1956}. While such assumptions can only lead to an \emph{idealized} material behaviour, the merits of such an ideal model were already described by H. Hencky in his 1928 article \emph{On the form of the law of elasticity for ideally elastic materials} %
\cite{hencky1928, henckyTranslation}:

\begin{quote}
{\em Like so many mathematical and geometric concepts, it is a useful ideal, because once its deducible properties are known it can be used as a comparative rule for assessing the actual elastic behaviour of physical bodies. {\rm [\dots]} While it is certainly a matter of empirical observation to determine how actual materials compare to the ideally elastic body, the law itself acts as a measuring instrument which is extended into the realm of the intellect, making it possible for the experimental researcher to make systematic observations.}%
\end{quote}%

The \emph{range of applicability} of such an idealized response, however, must necessarily be restricted to minute strains, perhaps in the order of $1\%$, for otherwise we are to expect interference with non-elastic effects like plastic deformations, microstructural instabilities or bifurcations. Nonetheless, it should be formulated tensorially correct for arbitrarily large strains. It is also clear that the restriction to small elastic strains does not imply that one can use linear elasticity theory, nor that the ideal elastic response for larger stresses or strains is arbitrary. On the contrary, our idealization should work, as an ideal model, for arbitrarily large strains.

In the past, a large number of possible basic assumptions for elastic materials have been suggested in order to respond to the idealization described above. Among the most commonly accepted are:
\begin{itemize}
	\item hyperelasticity: the existence of a strain energy function $W$,
	\item homogeneity: the strain energy $W$ does not depend on the position in the body,
	\item simple material: the strain energy $W=W(F)$ depends only on the first deformation gradient $F$,
	\item objectivity: $W(Q\cdot F) = W(F)$ for all $Q\in\SO(3)$,
	\item isotropy: $W(F\cdot Q) = W(F)$ for all $Q\in\SO(3)$,
	\item unique (up to rotations) stress-free reference state $U=\sqrt{F^TF}=\id$,
	\item linearization consistent with linear elasticity theory at the reference state,
	\item well-posedness of the corresponding linear elasticity model in statics and dynamics,
	\item correct stress response for extreme strains: $\sigma\to\infty$ as $V=\sqrt{FF^T}\to\infty$ as well as $\sigma\to-\infty$ as $\det(V)\to0$, where $\sigma$ denotes the Cauchy stress tensor,
	\item second-order behaviour in agreement with Bell's experimental observations \cite{bell1973}, i.e. the instantaneous elastic modulus $E$ decreases for tension and increases in the case of compression (c.f. Fig. \ref{figure:thirdOrderCurve}, which shows the unsuitability of the Saint-Venant-Kirchhoff model),
	\item correct energetic behaviour for extreme strains in order to ensure invertibility of the deformation gradient $F$: $W\to\infty$ for $\norm{F}\to\infty$ as well as $W\to\infty$ for $\det(F)\to0$,
	\item polyconvexity \cite{ball1977, Schroeder_Neff01, schroeder2008, schroeder_neff2010, ebbing2009construction, ebbing2009approximation}, quasiconvexity \cite{morrey1952, schroeder_neff2010},
	\item Legendre-Hadamard-ellipticity \cite{NeffGhibaLankeit},
	\item Baker-Ericksen inequalities \cite{bakerEri54}.
\end{itemize}
\begin{figure}[t]
	\centering
	\begin{tikzpicture}[scale=1]
		\ifshowimages
			\input{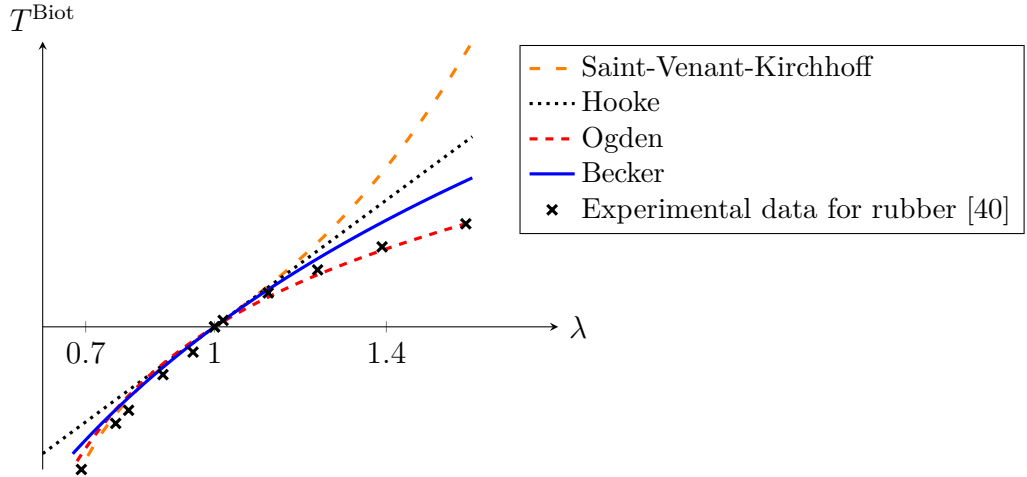}
		\fi
	\end{tikzpicture}
	\caption{Nonlinear behaviour of incompressible elastic materials in a simple tensile stress test, drawing the Biot stress $\Biot$ versus the principal stretch $\lambda$.}
	\label{figure:thirdOrderCurve}
\end{figure}
Apart from these conditions there are several properties which, while not generally viewed as necessary for an elasticity model, may be considered as constitutive assumptions for an \emph{idealized} material as well:
\begin{itemize}
\item superposition principle: $T(V_1\cdot V_2) = T(V_1) + T(V_2)$ for all \emph{coaxial} stretches $V_1$ and $V_2$ and some corresponding stress tensor $T$;
\item invertible stress-strain relation: the mapping $E\mapsto T(E)$ is invertible for some stress tensor $T$ and a corresponding work conjugate strain tensor $E$ (if $T$ is the Cauchy stress tensor, then this invertibitliy condition is satisfied e.g. for a variant of the compressible Neo-Hooke energy \cite{ghiba2014}; if T is the Biot stress tensor, then this condition is Truesdell's invertible force stretch (IFS) relation \cite[p. 156]{truesdell65});
\item tension-compression symmetry: $T(V\inv)=-T(V)$ for some stress tensor $T$ (note that the classical hyperelastic tension-compression symmetry $W(F)=W(F\inv)$ is equivalent to $\tau(V\inv)=-\tau(V)$ for the Kirchhoff stress $\tau$ and the left stretch tensor $V$);
\item plausible behaviour under simple homogeneous finite stresses (similar to linear elasticity):
\begin{itemize}
\item pure shear stresses of the form $T=\matrs{0&s&0\\s&0&0\\0&0&0}$ should induce stretches of the form $V=\matrs{B_{11}&B_{12}&0\\B_{12}&B_{22}&0\\0&0&1}$ with $\det\matrs{B_{11}&B_{12}\\B_{12}&B_{22}} = 1$,
\item spherical stresses of the form $T=\matrs{a&0&0\\0&a&0\\0&0&a}$ should induce volumetric stretches of the form $V=\matrs{\lambda&0&0\\0&\lambda&0\\0&0&\lambda}$;
\end{itemize}
\unitlength1.0mm
\begin{minipage}{.45\textwidth}
\begin{figure}[H]\centering
 \begin{picture}(90,25)
  \put(0,0){\epsfig{file=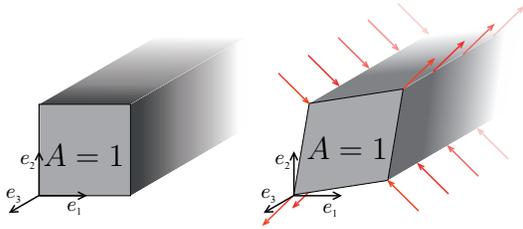, height=30mm, angle=0}}
  \put(5,8.1){$A=1$}
  \put(40,9){$A=1$}
 \end{picture}
 \caption{Pure shear stress should induce pure shear stretch, preserving the area $A$.}
 \label{figure:pureshear3D}
\end{figure}
\end{minipage}%
\hspace*{4mm}%
\begin{minipage}{.45\textwidth}
\begin{figure}[H]\centering
 \begin{picture}(90,25)
  \put(7,0){\epsfig{file=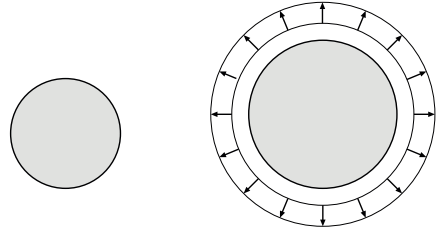, height=30mm, angle=0}}
 \end{picture}
 \caption{Spherical stress should induce purely volumetric stretch.}
 \label{figure:sphericalStress}
\end{figure}
\end{minipage}
\unitlength0.357mm
\item ordered stresses (\enquote{greater stress corresponds to greater stretch}): $(T_i-T_j)\cdot(\lambda_i-\lambda_j)>0$ for all $\lambda_i\neq\lambda_j$ and some stress tensor $T$, where $T_k$ are the principal stresses, i.e. the principal values of $T$, and $\lambda_k$ are the principal stretches (note that the Baker-Ericksen inequality can be stated as $(\sigma_i-\sigma_j)\cdot(\lambda_i-\lambda_j)>0$ for all $\lambda_i\neq\lambda_j$ where $\sigma$ is the Cauchy stress tensor);
\item simple volumetric-isochoric decoupling to ensure a suitable formulation of the incompressibility constraint,
\item minimal number of physically motivated and experimentally identifiable constitutive coefficients, e.g. only the two isotropic Lam\'e constants,
\item clear physical interpretation of Poisson's number $\nu$ for finite deformations: $\nu=\frac12$ enforces exact incompressibility ($\det F = 1$) and $\nu=0$ implies no lateral contraction under uniaxial tension (as in linear elasticity),
\item greatest possible extent of elastic determinacy \cite[p. 19]{henckyTranslation}: the stress response should not depend on a specific reference state or previously applied deformations; a similar condition was proposed by Murnaghan \cite{murnaghan1941,murnaghan1944}, who argued that the dependence of the stress response on a specific \enquote{position of zero strain} was tantamount to an \enquote{action at a distance} and should therefore be avoided.
\end{itemize}
Several attempts to propose such an idealized model of elasticity can be found in the literature. Becker's deduction can be seen as an early example of such an attempt.

\subsection{A modern interpretation of Becker's development}
Becker, in his development of a nonlinear law of elasticity, rejects many of his contemporaries' approaches to the problem of finite elasticity. He starts his introduction with a description of Hooke's law, stating that apart from its original formulation (\Bquote{Strain is proportionate to the load, or the stress initially applied to an unstrained mass} \Bref{becker:originalHooke}) it is often interpreted in a different way (\Bquote{Strain is proportional to the final stress required to hold a strained mass in equilibrium} \Bref{becker:alternativeHooke}). These two different interpretations of Hooke's law for finite deformations\footnote{The different possibilities of a Hookean law for finite deformations have been discussed in \cite{xiao2011} and \cite{batra1998}.} can be expressed as
\begin{equation}\label{eq:hookeOriginal}
	\begin{aligned}
		\Biot &= 2\,G\cdot\dev_3(U-\id) + \bulk\cdot\tr[U-\id]\cdot\id\\
		&= 2\,G\cdot(U-\id) + \Lambda\cdot\tr[U-\id]\cdot\id
	\end{aligned}
\end{equation}
and
\begin{equation}\label{eq:hookeAlternative}
	\begin{aligned}
		\sigma &= 2\,G\cdot\dev_3(V-\id) + \bulk\cdot\tr[V-\id]\cdot\id\\
		&= 2\,G\cdot(V-\id) + \Lambda\cdot\tr[V-\id]\cdot\id
	\end{aligned}
\end{equation}
respectively, where $U = \sqrt{F^T F}$ is the right Biot stretch tensor, $V = \sqrt{FF^T}$ is the left Biot stretch tensor, $\sigma$ is the Cauchy stress tensor, $\Biot$ is the Biot stress tensor, $\dev_3 X = X-\frac13\tr(X)\cdot \id$ denotes the deviatoric part of $X\in\R^{3\times3}$, $\bulk$ is the bulk modulus and $G,\Lambda$ are the Lam\'e constants.
According to Becker, it was already \Bquote{universally acknowledged that either law {\rm [\eqref{eq:hookeOriginal}, \eqref{eq:hookeAlternative}]} is applicable  only to strains so small that their squares are negligible} \Bref{becker:hookeRejection}. He gives a number of reasons for this rejection of Hooke's law as a model for finite deformations, including the fact that it allows for infinite distortions ($\det F = 0$) under finite stresses \Bref{becker:hookeSingularity}. In addition, Becker states that Hooke's law \Bquote{rests entirely upon experiment}, i.e. that there is no underlying framework necessitating the linearity of the stress-strain relation. However, Becker also rejects the idea that the stress response could be discovered by \Bquote{any process of pure reason} alone\footnote{As a proponent of the works of Immanuel Kant \cite{becker1898}, it is consequential that Becker rejects the purely empiricist approach as well the rationalist one.}, an approach which he attributes to Barr\'e de Saint-Venant. In fact, in Becker's time the elasticity models that had been developed through purely geometrical considerations generally implied the so-called \emph{Cauchy relations} \cite{hehl2002,Neff_Jeong_IJSS09, grioli2013}, i.e. they determined the lateral contraction independent of the specific material, corresponding to a fixed value $\nu=\frac14$ for Poisson's number $\nu$. As Becker points out in a footnote later on \Bref{becker:cauchyHypothesis}, this value for $\nu$ should only be regarded as a special case and not as a general law\footnote{This model, with a stress-strain law of the form\[\sigma \;=\; G\,(V-\id)+\tfrac{G}{2}\tr(V-\id)\cdot\id \;=\; G\,[(V-\id)+\tfrac12\tr(V-\id)\cdot\id]\,,\] is also called the \emph{rari-constant} theory of isotropic elasticity. The elastic behaviour of many materials, including metal, can not be described accurately by this one-parameter model.}.

Instead, his approach to describe the deformation of an ideally elastic body can be summarized as follows: motivated by geometric considerations he postulates a connection between \emph{shear stresses} and \emph{shear strains} as well as between \emph{volumetric stresses} and \emph{dilational strains}. He then shows that every homogeneous finite deformation can be decomposed into two shear stretches and a purely dilational deformation. Finally Becker assumes that a \emph{law of superposition} holds for all coaxial finite strains, allowing him to reduce the problem of a general stress-stretch relation to shears and dilations only.

Thus Becker makes a number of \emph{assumptions} about the stress-stretch relation from which he then \emph{deduces} a law of elasticity. His final result is a stress-stretch relation which in today's notation can be written as
\begin{equation*}
\boxed{\begin{aligned}
	\Biot &= 2\,G\cdot\dev_3\log(U) + \bulk\cdot\tr[\log(U)]\cdot\id\\
	&= 2\,G\cdot\log(U) + \Lambda\cdot\tr[\log(U)]\cdot\id\,,
\end{aligned}}
\end{equation*}
where $\log(U)$ is the principal matrix logarithm of the right Biot stretch tensor $U=\sqrt{F^TF}$.

In order to reproduce Becker's approach in a more modern framework of elasticity theory we will therefore interpret Becker's implicit assumptions as \emph{axioms} for a law of ideal elasticity. While sections \ref{section:assumptions} and \ref{section:geometryOfShear} summarize Becker's motivation for these axioms as well as some of his computations, a generalized deduction of Becker's law of elasticity will be given in section \ref{section:deduction}. Finally, we will investigate some basic properties of the resulting stress-stretch relation in section \ref{section:properties}.

While the axioms are in fact sufficient to completely characterize an isotropic stress-stretch relation uniquely up to two material parameters, it turns out that some of them may be weakened considerably without changing the result.

\section{Becker's assumptions}\label{section:assumptions}
In order to understand Becker's approach it is important to distinguish his (often implicitly stated) assumptions from his deduced results. Like Becker, we consider a unit cube in the reference configuration $\Omega_0$ the edges of which are aligned with an orthogonal coordinate system $e_1,e_2,e_3$. Unless indicated otherwise, all matrix representations of linear mappings are given with respect to this coordinate system.

\subsection{Basic assumptions}
Becker's most basic assumptions are that the stress-stretch relation is an analytic function \Bref{becker:analyticity} as well as isotropic\footnote{By isotropy, Becker means the absence of any directional information. He does not have a representation theorem for isotropic tensor functions at his disposal; note that J. Finger's influential monograph on isotropic nonlinear elasticity in terms of the three principal invariants \cite{finger1894} was not published until 1894.} \Bref{becker:isotropy} and that it is possible to \Bquote{regard strains as functions of load} \Bref{becker:invertibility}, i.e. that the strain-load mapping is invertible.
Note that here and throughout we will interpret Becker's \enquote{initial stress} (which, for a unit cube, is equal to the load) as the \emph{Biot stress tensor} \cite{curnier1991}
\[
	\Biot \;=\; U\cdot\PKtwo \;=\; J\cdot R^T \cdot \sigma \cdot F^{-T}\,,
\]
where $F$ denotes the deformation gradient, $F = R\,U$ is the polar decomposition \cite{Lankeit2014} of $F$ with $R\in\SOn$ and $U=\sqrt{F^TF}\in\PSymn$,\, $J=\det F = \det U$ is the Jacobian determinant, $\sigma$ is the Cauchy stress tensor and $\PKtwo$ is the symmetric second Piola-Kirchhoff stress tensor. A justification of this interpretation can be found in Appendix \ref{appendix:stressTensors}. Note that in the isotropic case, $\Biot$ is a symmetric tensor as well \cite{Neff_Biot07}.

\subsection{Pure finite shear}\label{section:pureFiniteShearAssumption}
One of Becker's main assumptions is that \Bquote{a simple finite shearing strain must result from the action of two equal loads or initial stresses of opposite signs at right angles to one another} \Bref{becker:shearAxiom}. This assumption is mostly motivated by geometric considerations: if the deformation is a homogeneous pure shear of the form
\begin{equation}\label{eq:shearDeformationDefinition}
	F = \dmatr{\alpha}{\frac1\alpha}{1}\,, \quad \alpha>1\,,
\end{equation}
then the so-called \emph{planes of no distortion} have a number of properties connected to the quantity $\alpha$. Becker then relates these properties of strain to certain properties of stress, thereby establishing that the stress corresponding to the above shear deformation must be a pure shear stress of the form
\begin{equation}\label{eq:shearStressDefinition}
	\Biot = \dmatr{s}{-s}{0}\,, \quad s\in\R\,.
\end{equation}
His arguments are largely based on the assumption that \Bquote{[in] the particular case of a shear (or a $\it{pure}$ shear) there are two sets of planes on which the stresses are purely tangential, for otherwise there could be no planes of zero distortion} \Bref{becker:planesOfNoDistortionAndTangentialStrain}.

\unitlength1.0mm
\begin{figure}[H]\centering
 \begin{picture}(100,37)
  \put(0,0){\epsfig{file=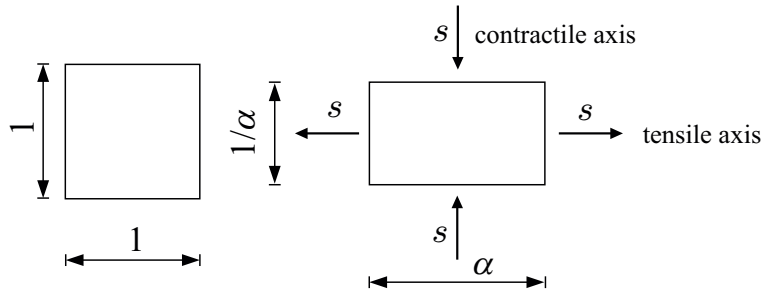, width=100mm, angle=0}}
 \end{picture}
 \caption{Pure shear load and the corresponding shear deformation.}
 \label{figure:pureShear}
\end{figure}
\unitlength0.357mm
If a a Biot shear stress of the form \eqref{eq:shearStressDefinition} corresponds to a pure shear deformation of the form \eqref{eq:shearDeformationDefinition}, then the corresponding Cauchy stress tensor $\sigma$ computes to
\begin{equation}\label{eq:shearCauchyStressGeneral}
	\sigma \;=\; \frac{1}{\det U}\cdot U\inv\cdot F\cdot \Biot\cdot F^T \;=\; 1\cdot F\inv\cdot F \cdot \Biot \cdot F \;=\; \dmatr{\alpha\,s}{-\frac{s}{\alpha}}{0}\,.
\end{equation}
The principal Cauchy stresses are therefore $\sigma_1 = \alpha\,s$, $\sigma_2 = -\frac{s}{\alpha}$ and $\sigma_3 = 0$. Note carefully that Becker's shear deformation is oriented differently: in his considerations, the contractile axis (along the eigenvector to the smaller eigenvalue $\frac1\alpha$) is aligned with the $e_1$-axis. This difference is reflected in Becker's formula $-\,\Cauchy_{1}\, \alpha = \Cauchy_{2}/\alpha$ \; \Bref{becker:cauchyStressShearFormula}. For our choice of axes, the corresponding equality reads
\[
	-\,\Cauchy_{2}\, \alpha = \Cauchy_{1}/\alpha\,.
\]

A more detailed geometric description of this relation will be given in section \ref{section:geometryOfShear}. More recent discussions of shear stresses and shear strains can be found in \cite{destrade2012} and \cite{norris2006}. The so-called \emph{planes of no distortion} also play an important role in Becker's treatment of the rupture of rocks \cite{becker1892} as well as his later works on schistosity and slaty cleavage \cite{becker1904, becker1907}, where properties of the planes of no distortion are linked to failure criteria for deformations beyond the range of elastic deformations. A summary of Becker's work on yield criteria for rocks can be found in \cite{turner1942}, while the concept of planes of no distortion and its relation to the tangential shear strain is described in a more detailed form in \cite{griggs1935}.

\subsubsection{The planes of no distortion}\label{section:planesOfNoDistortion}
Let $F\in\GLpn$ denote an invertible linear mapping. We call a plane $\iplane\subset\R^3$ through the origin an \emph{initial plane of no distortion} if the restriction of $F$ to $\iplane$ is a rotation, which is the case if and only if $\innerproduct{Fx,Fy} = \innerproduct{x,y}$ for all $x,y\in\iplane$ or, equivalently, if $\norm{Fx} = \norm{x}$ for all $x\in\iplane$, where $\innerproduct{\cdot,\cdot}$ denotes the Euclidean inner product on $\R^3$. Furthermore, we call $\fplane\subset\R^3$ a \emph{final plane of no distortion} if $\fplane$ is the image under $F$ of an initial plane of no distortion, which is the case if and only if $\norm{F\inv x} = \norm{x}$ for all $x\in\fplane$. Since Becker's considerations of such planes are confined to the deformed (or final) configuration of a homogeneous deformation, we will often refer to $\fplane$ simply as a \emph{plane of no distortion}.\\
\unitlength1.0mm
\begin{figure}[H]\centering
 \begin{picture}(130,45)
  \put(0,0){\epsfig{file=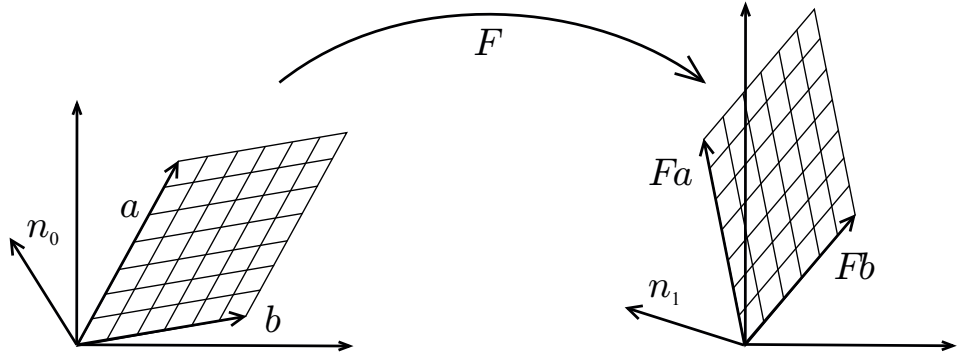, height=47mm, angle=0}}
 \end{picture}
 \caption{An initial plane of no distortion is only rotated by the linear mapping $F$, preserving the angles between two vectors as well as their lengths.}
 \label{figure:pondrotation}
\end{figure}
\unitlength0.357mm
The following basic existence property can be found in \cite{griggs1935}.

\begin{proposition}\label{prop:planesOfNoDistortionExistence}
Let $F\in\GLpn$. Then there exists a plane of no distortion for $F$ if and only if $\lambda_2 = 1$, where $\lambda_1 > \lambda_2 > \lambda_3$ denote the \emph{singular values} of $F$.
\end{proposition}
\begin{remark}
If $\lambda_2 \neq 1$ for the second singular value $\lambda_2$ of a linear mapping $F\in\GLpn$, then instead of a plane there exists a \emph{surface of no distortion} in the form of an elliptical cone \cite[p. 133]{griggs1935} instead of a plane of no distortion.
If, however, we generalize the term to mean any plane $\fplane$ such that the restriction of $F$ to $\fplane$ is a \emph{dilated rotation} (also called a \emph{conformal mapping}) of the form $\lambda\cdot Q$ with $\lambda\in\R^+$ and $Q\in\SOn$, then Proposition \ref{prop:planesOfNoDistortionExistence} shows that such a plane exists for $F\in\GLpn$ if and only if the dilational factor $\lambda$ is the second singular value of $F$.
\end{remark}
\begin{remark}
It was shown by J. Ball and R.D. James \cite[Proposition 4]{ball1987} that the equality $\lambda_2 = 1$ holds if and only if the right Cauchy-Green tensor $C = F^TF$ corresponding to $F$ is expressible in the form
\[
	C \;=\; (\id + \xi\otimes\eta)\cdot(\id + \eta\otimes\xi) \;=\; (\id + \eta\otimes\xi)^T\cdot(\id + \eta\otimes\xi)
\]
with $\xi,\eta\in\R^3$, where $\xi\otimes\eta \in \Rnn$ denotes the tensor product of $\xi$ and $\eta$. Thus there exists a plane of no distortion for $F\in\GLpn$ if and only if there exists a rank-one tensor $H\in\Rnn$ with
\[
	F^TF \;=\; C \;=\; (\id+H)^T\cdot(\id+H)\,.
\]
\end{remark}
Since Becker only considers planes of no distortion for pure shear deformations, we will assume from now on that $F$ has the form
\begin{equation}\label{eq:shearDeformation}
	F \;=\; \dmatr{\alpha}{\frac1\alpha}{1}
\end{equation}
with $\alpha > 1$. Then clearly $\lambda_2=1$, hence there are two distinct planes of no distortion which can be determined by direct computation: for $x = (x_1,x_2,x_3)^T \in\R^3$ with $\norm{x}=1$ we can use the equalities
\[
	x_1^2+x_2^2+x_3^2 = \norm{x}^2 \qquad \Longrightarrow \qquad x_3^2 = \norm{x}^2-x_1^2-x_2^2
\]
to find
\[
	\quad\norm{Fx}^2 \;=\; \alpha^2 x_1^2 \,+\, \frac{1}{\alpha^2}\,x_2^2 \,+\, x_3^2 \;=\; (\alpha^2-1)x_1^2 \,+\, \left(\frac{1}{\alpha^2} - 1\right)x_2^2 \,+\, \norm{x}^2\,.
\]
Thus the equality $\norm{Fx}=\norm{x}$ is equivalent to
\begin{align}
	&\norm{x}^2 \;=\; (\alpha^2-1)x_1^2 \,+\, \left(\frac{1}{\alpha^2} - 1\right)x_2^2 \,+\, \norm{x}^2\nnl
	\Longleftrightarrow\quad &\left(1 - \frac{1}{\alpha^2}\right)x_2^2 \;=\; (\alpha^2-1)\,x_1^2 \quad\Longleftrightarrow\quad x_2^2 \;=\; \alpha^2\,x_1^2\,,
\end{align}
hence every $x\in\R^3$ with $\norm{Fx}=\norm{x}$ is of the form $x=(s,\, \pm \alpha\,s,\, t)^T$ with $s,t\in\R$. Since the final directions of no distortion are the images of those vectors, they in turn have the form $y = Ax = (\pm \alpha\,s,\, s,\, t)^T$. Therefore, for every shear deformation of the form \eqref{eq:shearDeformation} with $\alpha>1$, there are two initial planes of no distortion,
\begin{equation}
	\iplane^+ \;=\; \Big\{\vect{x_1}{x_2}{x_3} \,:\, x_2 = \alpha\cdot x_1 \Big\} \qquad \text{and} \qquad \iplane^- \;=\; \Big\{\vect{x_1}{x_2}{x_3} \,:\, x_2 = -\alpha\cdot x_1 \Big\}\,,
\end{equation}
as well as two final planes of no distortion
\begin{equation}
	\fplane^+ \;=\; \Big\{\vect{x_1}{x_2}{x_3} \,:\, x_1 = \alpha\cdot x_2 \Big\} \qquad \text{and} \qquad \fplane^- \;=\; \Big\{\vect{x_1}{x_2}{x_3} \,:\, x_1 = -\alpha\cdot x_2 \Big\}\,.
\end{equation}
To verify Becker's claim that \Bref{becker:planesOfNoDistortion}
\begin{quote}
\Bquote{[in] a finite shearing strain of ratio $\alpha$, it is easy to see that the normal to the planes of no distortion makes an angle with the contractile axis of shear the cotangent of which is $\alpha$},
\end{quote}
we consider the direction of the contractile axis along the eigenvector $(0,1,0)^T=e_2$ corresponding to the smallest eigenvalue $\frac1\alpha$. Then the cotangent of the angle between this axis and a vector $\hat{n}=(y_1,y_2,y_3)^T$ with $y_1\neq0$ is given by $\frac{y_2}{y_1}$. %
Since
\begin{equation}\label{eq:normalsToPOND}
	n^+ = (1,-\alpha,0)^T \quad\text{ and }\quad n^- = (1,\alpha,0)^T
\end{equation}
are normal vectors to the planes of no distortion $\fplane^+$ and $\fplane^-$, respectively, the cotangent of the angle between these normals and the contractile axis is indeed given by $\pm\alpha$. Similarly, it is easy to see that the normals to the initial planes of no distortion $\iplane^+$ and $\iplane^-$ form angles of cotangent $\pm\frac1\alpha$ with the contractile axis.

This definition of the planes of no distortion is also consistent with the definition by means of the shear ellipsoid\footnote{The shear ellipsoid is also discussed by Becker in the context of hyperbolic functions \cite[p. xxxii]{becker1909hyperbolic}.}
given by C.K. Leith\footnote{Equivalent, but more implicit definitions are given by Becker \cite{becker1892} and Griggs \cite{griggs1935}} in \emph{Structural geology} \cite{leith1913}:
\begin{quote}
	\emph{In a strain ellipsoid with three unequal principal axes there are only two cross-sections which are circular in outline [\dots] These planes [\dots] are called \enquote{planes of no distortion} because they preserve a circular cross-section similar to a section of the original sphere\dots}
\end{quote}
The strain ellipsoid of the deformation with principal stretches $\alpha$, \,$\alpha\inv$ and $1$ is defined by the equation
\[
	\frac{x_1^2}{\alpha^2} + \frac{x_2^2}{\alpha^{-2}} + x_3^2 = 1\,,
\]
where $x_1$ and $x_2$ denote coordinates with respect to the tensile axis and the contractile axis in pure shear respectively. The planes $\fplane^+$ and $\fplane^-$ are characterized by the equation $x_1 = \pm\alpha\,x_2$, thus we can verify that their intersections with the ellipsoid are indeed circles of radius 1 centred at the origin $(0,0,0)^T$:
\begin{align*}
	\frac{x_1^2}{\alpha^2} + \frac{x_2^2}{\alpha^{-2}} + x_3^2 &= 1 \;\wedge\; x_1 = \pm\alpha\,x_2\\
	\quad\Longrightarrow \quad \norm{\matrs{x_1\\x_2\\x_3}-\matrs{0\\0\\0}} &= x_1^2 + x_2^2 + x_3^2 = x_1^2 + x_2^2 + 1 - \frac{x_1^2}{\alpha^2} - \alpha^2 x_2^2\\ &= x_1^2 + x_2^2 + 1 - x_2^2 - x_1^2 \;=\; 1\,.
\end{align*}

\subsubsection{Different characterizations of the planes of no distortion}
The initial planes of no distortion can also be characterized in a number of different ways, for example
by means of the cofactor matrix $\Cof F$: If the restriction of $F$ to a plane is a rotation, then
\[
	\innerproduct{Fx,\, Fy} = \innerproduct{x,\, y}
\]
for all $x,y$ in the plane. Then for all $x,y$ in the plane with $\innerproduct{x,y}=0$ we find
\[
	\norm{Fx \times Fy}^2 = \norm{Fx}^2\cdot\norm{Fy}^2 - 2\cdot\innerproduct{Fx,\, Fy} = \norm{x}^2\cdot\norm{y}^2 - 2\cdot\innerproduct{x,\, y} = \norm{x}^2\cdot\norm{y}^2\,.
\]
Now let $n$ denote a unit normal vector to the plane of no distortion. Then $n$ can be represented as $n=x\times y$ with unit vectors $x,y$ in the plane and $\innerproduct{x,\,y}=0$. Since, in general,
\[
	(\Cof F)\,(x\times y) = Fx \,\times Fy\,,
\]
we find
\[
	\norm{(\Cof F)\,n}^2 = \norm{(\Cof F)\,(x\times y)}^2 = \norm{Fx \,\times Fy}^2 = \norm{x}^2\cdot\norm{y}^2 = 1\,.
\]
Therefore the equality $\norm{(\Cof F)\,n} = 1$ is a necessary condition for a unit normal vector $n$ to be the normal to an initial plane of no distortion. We compute
\[
	\Cof F \;=\; \det(F) \cdot F^{-T} \;=\; \alpha\cdot\frac1\alpha \cdot \dmatr{\alpha}{\frac1\alpha}{1}\inv \;=\; \dmatr{\frac1\alpha}{\alpha}{1}
\]
as well as
\begin{align*}
	\norm{(\Cof F)\,n}^2 \;&=\; \frac{n_1^2}{\alpha^2} + \alpha^2n_2^2 + n_3^2 \;=\; \frac{n_1^2}{\alpha^2} + \alpha^2n_2^2 + 1 - n_1^2 - n_2^2\\
	&=\; 1 \,+\, (\frac{1}{\alpha^2}-1)\,n_1^2 + (\alpha^2 - 1)\,n_2^2 \;=\; 1 \,+\, (\frac{1}{\alpha^2}-1)\,(n_1^2 - \alpha^2n_2^2)\,,
\end{align*}
thus
\[
	\norm{(\Cof F)\,n} \;=\; 1 \quad\Longleftrightarrow\quad 0 \;=\; (\frac{1}{\alpha^2}-1)\,(n_1^2 - \alpha^2n_2^2) \quad\Longleftrightarrow\quad n_1^2 \;=\; \alpha^2n_2^2
\]
for $\alpha>1$, showing that the cotangent of the angle between $n$ and the contractile axis is $\pm\frac1\alpha$. Therefore the equality $\norm{(\Cof F)\,n} = \norm{n} = 1$ holds \emph{only} if $n$ is a unit normal vector to a plane of no distortion. Note that this equality also implies
\[
	0 = \innerproduct{(\Cof F)\,n, (\Cof F)\,n} - \innerproduct{n,n} = \innerproduct{(\Cof F)^T(\Cof F)\,n-n,\,n} = \innerproduct{(\Cof B-\id)\,n,\,n}\,,
\]
where $B=FF^T$ denotes the left Cauchy-Green deformation tensor.

Becker gives another important characterization of the planes of no distortion: \Bquote{the planes of no distortion {\rm[\dots]} are also the planes of maximum tangential strain} \Bref{becker:maximumTangentialStrain}. In Becker's \emph{Finite Homogeneous Strain, Flow and Rupture of Rocks} \cite{becker1892}, the \emph{tangential strain} of $x\in\R^3$ is defined as the tangent of the angle between $x$ and $Fx$. Accordingly, the \emph{plane of maximum tangential strain} of a pure shear deformation $F$ is defined as the plane containing the undistorted axis as well as the line for which the tangential strain is maximal.

\unitlength1.0mm
\begin{figure}[H]\centering
\begin{picture}(160,42)
	\put(0,0){\epsfig{file=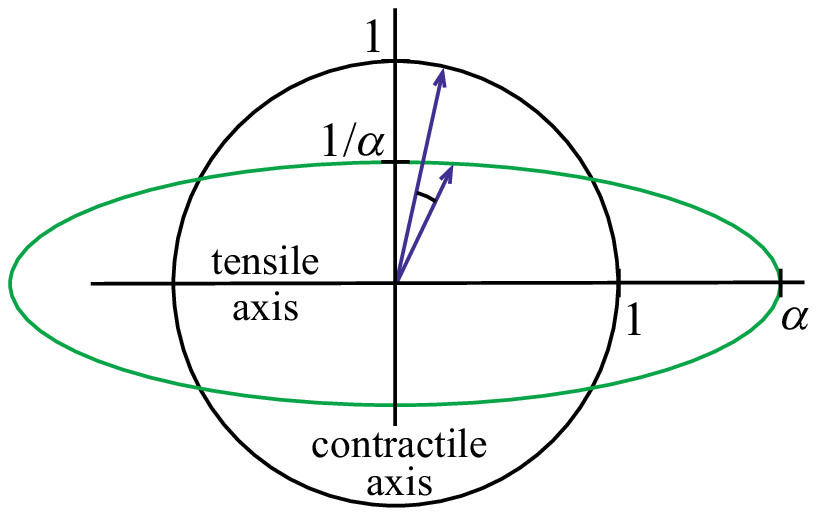, width=70mm, angle=0}}
	\put(39,23){$\widetilde{\vartheta}$}
	\put(37.5,39.5){\textcolor{blue}{$\tilde{x}$}}
	\put(39,31){\textcolor{blue}{$F\tilde{x}$}}
	\put(84,0){\epsfig{file=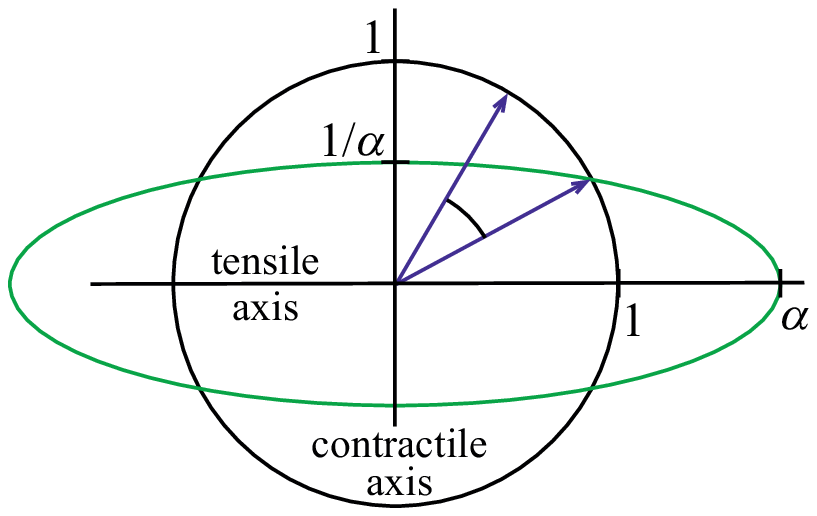, width=70mm, angle=0}}
	\put(125,26){$\vartheta$}
	\put(128,37){\textcolor{blue}{$x$}}
	\put(136,29){\textcolor{blue}{$Fx$}}
\end{picture}
\caption{The shear ellipsoid of a deformation $F$, showing an arbitrary tangential strain $\tan \widetilde{\vartheta}$ (left) and the maximum tangential strain $\tan \vartheta$ (right), which is realized for the plane of no distortion.}
\label{figure:tangentialStrain}
\end{figure}
\unitlength0.357mm

We will show that the planes of maximum tangential strain and the initial planes of no distortion are, in fact, identical. Let $x$ be of the form $x=(x_1,x_2,0)^T$ with $x_1^2+x_2^2=1$, i.e. we assume that $x$ is a unit vector orthogonal to the undistorted $e_3$-axis. Then the angle $\vartheta$ between $x$ and $Fx$ is given by
\[
	\cos\vartheta = \frac{\innerproduct{x,\,Fx}}{\norm{x}\cdot\norm{Fx}} = \frac{\innerproduct{x,\,Fx}}{\norm{Fx}}\,.
\]
The direction $x$ of maximum tangential strain is characterized by the angle $\vartheta$ for which $\tan \vartheta$ is maximal. In order to find $x$ it is therefore sufficient to maximize%
\[
	\vartheta = \arccos\,\frac{\innerproduct{x,Fx}}{\norm{Fx}} = \arccos\,\frac{\alpha\,x_1^2 + \frac{1}{\alpha}\,x_2^2}{\sqrt{\alpha^2x_1^2 + \frac{1}{\alpha^2}\,x_2^2}}\,,
\]
since $\vartheta\mapsto\tan\vartheta$ is monotone. Using the equality $x_1^2 = 1-x_2^2$, we obtain
\[
	\arccos\,\frac{\innerproduct{x,Fx}}{\norm{Fx}} = \arccos\,\frac{\alpha\,(1-x_2^2) + \frac{1}{\alpha}\,x_2^2}{\sqrt{\alpha^2(1-x_2^2) + \frac{1}{\alpha^2}\,x_2^2}}\,.
\]
\newcommand{\tempvar}{t}%
In order to find $\tempvar\in[-1,1]$ for which the function
\[
	f:[-1,1]\to\R\,, \quad f(\tempvar) = \arccos \frac{\alpha(1-\tempvar^2) + \frac{1}{\alpha}\,\tempvar^2}{\sqrt{\alpha^2(1-\tempvar^2) +\frac{1}{\alpha^2}\,\tempvar^2}}
\]
attains its maximum, we compute the first derivative of $f$ to be
\[
	\frac{{\rm d}}{{\rm d}\,t}\,f(t)=\frac{(\alpha^2 \tempvar^2-1)+\tempvar^2}{\alpha\,\tempvar (\tempvar^2-1)} \cdot \frac{ \sqrt{\frac{(\alpha^2-1)^2 \tempvar^2 (\tempvar^2-1)}{\alpha^4 (\tempvar^2-1)-\tempvar^2}}}{ \sqrt{\alpha^2-\frac{(\alpha^4-1) \tempvar^2}{\alpha^2}}}\,.
\]
Thus the possible extremal points of $f$ are $0$, $\pm1$ and $\pm\frac{\alpha}{\sqrt{1+\alpha^2}}$. Since $f(\frac{\alpha}{\sqrt{1+\alpha^2}}) = f(-\frac{\alpha}{\sqrt{1+\alpha^2}})$ as well as $f(-1)=f(0)=f(1)=0$ and $f(t)\geq0$ for all $t\in[-1,1]$, the global maxima of $f$ are given by $t=\pm \frac{\alpha}{\sqrt{1+\alpha^2}}$. Applying this result to our original problem, we find that the maximum tangential strain is attained for the directions $\xhat_\pm = \Big(\sqrt{1-\frac{\alpha^2}{1+\alpha^2}},\, \pm\frac{\alpha}{\sqrt{1+\alpha^2}},\, 0\Big)^T$.

\unitlength1.0mm
\begin{figure}[h]\centering
\begin{picture}(90,60)
  \put(4,0){\epsfig{file=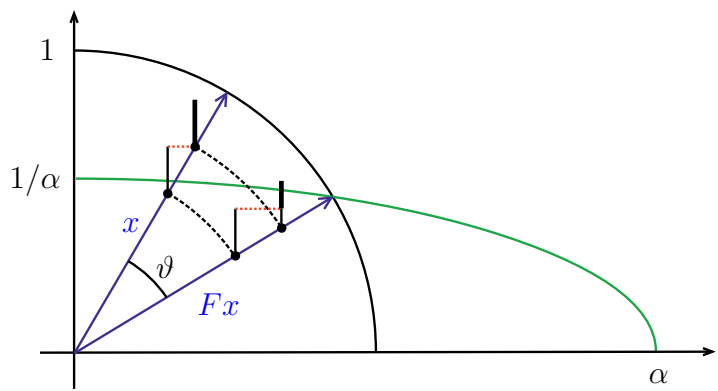, width=90mm, angle=0}}
  \put(4,44){$1$}
  \put(0,27){$1/\alpha$}
  \put(85,1){$\alpha$}
  \put(19.5,15){$\vartheta$}
  \put(15,21){\textcolor{blue}{$x$}}
  \put(25,10){\textcolor{blue}{$Fx$}}
\end{picture}
\caption{Another visualization of the tangential strain: the shift between two parallel
 vertical lines under the deformation is maximal along the plane of no distortion.}
\label{figure:visualisationTangentialStrain}
\end{figure}
\unitlength0.357mm

Finally, in order to verify Becker's claim that \Bquote{the planes of no distortion {\rm[\dots]} are also the planes of maximum tangential strain} \Bref{becker:maximumTangentialStrain}, we observe that
\begin{align*}
	\norm{F\xhat_\pm}^2 \,=\, \Bignorm{\vect{\alpha\cdot\sqrt{1-\frac{\alpha^2}{1+\alpha^2}}\,}{\frac1\alpha\cdot(\pm\frac{\alpha}{\sqrt{1+\alpha^2}})}{0}}^2
	\,&=\, \alpha^2\cdot\left(1-\frac{\alpha^2}{1+\alpha^2}\right) \,+\, \frac{1}{1+\alpha^2}\\
	\,&=\, \alpha^2 + \frac{1-\alpha^4}{1+\alpha^2}
	\,=\, \alpha^2 + 1 - \alpha^2 = 1 = \norm{\xhat_\pm}\,.
\end{align*}
Thus the plane of maximum tangential strain is indeed the initial plane of no distortion.

\subsubsection{Simple glide deformation}\label{section:simpleGlide}
A homogeneous deformation $F$ is called \emph{simple glide} if it has the form \cite{Neff_Muench_simple_shear09}
\[
	F = \matr{1&\gamma&0\\ 0&1&0\\ 0&0&1}\,,
\]
with $\gamma\in\R$, $\gamma>0$. Since $F\cdot(t,0,s)^T = (t,0,s)^T$ for all $t,s\in\R$, the restriction of $F$ to the $e_1$-$e_3$-plane is the identity function, showing that it is the initial plane of no distortion as well as the final plane of no distortion.
\unitlength1.0mm
\begin{figure}[h]\centering
\begin{picture}(90,32)
  \put(0,0){\epsfig{file=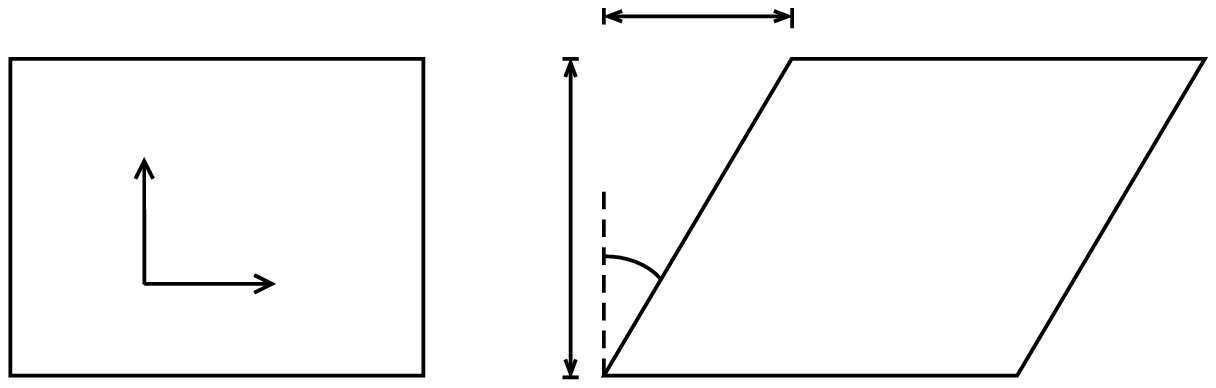, width=90mm, angle=0}}
  \put(22,6){$e_1$}
  \put(9,19){$e_2$}
  \put(39,11){1}
  \put(47,11){$\varphi$}
  \put(50,29){$\gamma$}
\end{picture}
\caption{Simple glide deformation with shear angle $\tan{\varphi}=\gamma/1$; horizontal lines slide relative to each other, vertical lines tilt to accomodate; originally right angles are distorted, the shear strain measures the change of angles.}
\label{figure:simpleGlide}
\end{figure}
\unitlength0.357mm
\noindent We examine the right Cauchy-Green deformation tensor of a simple glide, which is given by
\[
	C = F^TF = \matr{1&\gamma&0\\ \gamma&1+\gamma^2&0\\ 0&0&1}\,.
\]
The principal stretches $\lambda_i$ are the singular values of $F$ which, in turn, are the square roots of the eigenvalues of $C$,
\[
	\lambda_1 = \frac12(\gamma + \sqrt{\gamma^2+4})\,, \qquad \lambda_2 = \frac12(\gamma + \sqrt{\gamma^2-4})\,, \qquad \lambda_3 = 1\,,
\]
and the principal axes are given by the corresponding eigenvectors of $C$:
\[
	v_1 = \matr{2\\ \gamma + \sqrt{\gamma^2+4}\\ 0}\,, \qquad v_2 = \matr{-2\\ \gamma + \sqrt{\gamma^2-4}\\ 0}\,, \qquad v_1 = \matr{0\\ 0\\ 1}\,.
\]
Since $1 < \lambda_1 = \frac1\lambda_2$ as well as $\lambda_3=1$, we can interpret the simple glide as a rotated pure shear with shear ratio $\alpha=\lambda_1$, where the direction of the tensile axis is $v_1$ and the contractile axis is given by $v_2$. The cotangent of the angle $\vartheta$ between $v_2$ and the $e_1$-axis is
\[
	\cot(\vartheta) = \frac{-2}{\gamma + \sqrt{\gamma^2-4}} = \frac{-2}{2\cdot\lambda_2} = -\frac{1}{\lambda_2} = -\alpha\,.
\]
Hence the angle between $(1,0,0)^T=e_1$ and the contractile axis is
\[
	\cot(-\vartheta) = -\cot(\vartheta) = \alpha\,,
\]
yielding another angular characterization of the shear ratio $\alpha$.

\subsection{Dilation}
Similarly to the connection between shear stress and shear stretches, Becker further assumes that \Bquote{dilational forces acting positively and equally in all directions} \Bref{becker:volumetricStress} are \Bquote{[the loads] effecting dilation} \Bref{becker:dilationalStrain}, i.e. that purely volumetric initial stresses correspond to purely dilational stretches. More precisely, this assumption can be stated in the following way: if the principal axes remain fixed, every Biot stress of the form $\Biot = a\cdot\id$ with $a\in\R$ corresponds to a deformation of the form $F = \lambda\cdot\id$ with $\lambda>0$. Again, this assumption is stated only implicitly.

\subsection{Superposition}
Becker's most important assumption is that of a \emph{law of superposition} for coaxial deformations: \Bquote{the load sums correspond to the products of the strain ratios} \Bref{becker:superposition}. For homogeneous, coaxial stretch tensors $U_1$, $U_2$ this law can be stated as
\[
	\Biot(U_1\cdot U_2) \;=\; \Biot(U_1) \,+\, \Biot(U_2)\,,
\]
where $\Biot(U)$ denotes the Biot stress tensor corresponding to the right Biot stretch tensor $U$. Here, $U_1$ and $U_2$ are called coaxial if their principal axes coincide.

\subsubsection{The decomposition of stresses and strains}
An important application of the law of superposition involves the \emph{additive decomposition} of stresses and the corresponding \emph{multiplicative decomposition} of stretches. Consider an initial load (i.e. a Biot stress tensor) of the form
\[
	\Biot_1 = \dmatr{P}{0}{0}\,.
\]
In the case of isotropic materials, the corresponding stretches are identical for all directions orthogonal to the $x_1$-axis for symmetry reasons, which is the case if and only if the stretch tensor corresponding to $\Biot$ has the form
\[
	U_1 = \dmatr{a}{b}{b}
\]
for some $a,b\in\R^+$ with respect to the principal axes. Then, for $h_1=(ab^2)^{\frac13}$ and $p=\left(\frac{a}{b}\right)^{\frac13}$, we find
\begin{align*}
	U_1 &= \dmatr{\left(ab^2\cdot\frac{a^2}{b^2}\right)^{\frac13}}{\left(ab^2\cdot\frac{b}{a}\right)^{\frac13}}{\left(ab^2\cdot\frac{b}{a}\right)^{\frac13}}\\[4mm]
	&= h_1 \cdot \dmatr{p^2}{\frac1p}{\frac1p} = \dmatr{h_1}{h_1}{h_1} \cdot \dmatr{p}{\frac1p}{1} \cdot \dmatr{p}{1}{\frac1p}\,.
\end{align*}
In a similar way we can obtain the stretch tensors corresponding to the stresses
\[
	\Biot_2 = \dmatr{0}{Q}{0}	\qquad \text{ and } \qquad		\Biot_3 = \dmatr{0}{0}{R}
\]
respectively:
\begin{align*}
	U_2 &= h_2 \cdot \dmatr{\frac1q}{q^2}{\frac1q} = \dmatr{h_2}{h_2}{h_2} \cdot \dmatr{\frac1q}{q}{1} \cdot \dmatr{1}{q}{\frac1q}\,,\\[4mm]
	\qquad U_3 &= h_3 \cdot \dmatr{\frac1r}{\frac1r}{r^2} = \dmatr{h_3}{h_3}{h_3} \cdot \dmatr{\frac1r}{1}{r} \cdot \dmatr{1}{\frac1r}{r}
\end{align*}
for some $h_2,h_3,q,r\in\R^+$. Therefore every homogeneous deformation corresponding to a uniaxial load can be decomposed into a dilation and two shear deformations with perpendicular axes. Finally, for the general case of arbitrary stresses
\begin{align*}
	\Biot = \dmatr{P}{Q}{R} &= \dmatr{P}{0}{0} + \dmatr{0}{Q}{0} + \dmatr{0}{0}{R}\\[2mm]
	&=\Biot_1 + \Biot_2 + \Biot_3\,,
\end{align*}
the \textbf{law of superposition} yields the general formula
\begin{align*}
	U = U_1 \cdot U_2 \cdot U_3 &= h_1\,h_2\,h_3 \,\cdot\, \dmatr{p^2}{\frac1p}{\frac1p} \cdot \dmatr{\frac1q}{q^2}{\frac1q} \cdot \dmatr{\frac1r}{\frac1r}{r^2} = h_1\,h_2\,h_3 \,\cdot\, \dmatr{\frac{p^2}{q\,r}}{\frac{q^2}{p\,r}}{\frac{r^2}{p\,q}}
\end{align*}
for the stretch tensor $U$ corresponding to $\Biot$. Then $U$ can be decomposed into a dilation and two shears as well:
\[
	U = \dmatr{h_1\,h_2\,h_3}{h_1\,h_2\,h_3}{h_1\,h_2\,h_3}
		\cdot \dmatr{\frac{p^2}{q\,r}}{\frac{q\,r}{p^2}}{1}
		\cdot \dmatr{1}{\frac{p\,q}{r^2}}{\frac{r^2}{p\,q}}\,.
\]
Furthermore we can find an additive decomposition
\begin{align}
	\Biot \;&=\; \dmatr{\frac{P+Q+R}{3}}{\frac{P+Q+R}{3}}{\frac{P+Q+R}{3}}\label{eq:stressDecomposition}\\[2mm]
	&\qquad +\; \dmatr{-\frac{Q+R-2P}{3}}{\frac{Q+R-2P}{3}}{0}
	\;+\; \dmatr{0}{\frac{P+Q-2R}{3}}{-\frac{P+Q-2R}{3}}\nonumber
\end{align}
of $\Biot$ into a spherical stress and two pure shear stresses. Note that the decomposition of $U$ can be found in Becker's second table \Bref{becker:table2} while the decomposition of $\Biot$ can be found in the third table \Bref{becker:table3}.

In decomposing the strains $U_1$, $U_2$ and $U_3$ into a dilation and two shear strains, the planes of shear were chosen arbitrarily for every strain (or, more precisely, such that the two resulting shear ratios were identical). However, we may also choose two fixed planes (or, equivalently, fixed axes of tension and contraction) and decompose all strains into shears along the same axes:
\begin{align*}
	U_1 &= h_1 \cdot \dmatr{p^2}{\frac1p}{\frac1p} = h_1 \cdot \dmatr{p^2}{\frac{1}{p^2}}{1} \cdot \dmatr{1}{p}{\frac{1}{p}}\,,\\
	U_2 &= h_2 \cdot \dmatr{\frac1q}{q^2}{\frac1q} = h_2 \cdot \dmatr{\frac1q}{q}{1} \cdot \dmatr{1}{q}{\frac1q}\,,\\
	U_3 &= h_3 \cdot \dmatr{\frac1r}{\frac1r}{r^2} = h_3 \cdot \dmatr{\frac1r}{r}{1} \cdot \dmatr{1}{\frac{1}{r^2}}{r^2}\,.\\
\end{align*}
Note that the axes chosen here are identical to those chosen for the decomposition of $U$. Using this approach we obtain a modified version of Becker's first table:
\begin{table}[h]
\begin{center}
\begin{tabular}{l|ccc|ccc|ccc}\hline
Active force & \multicolumn{3}{|c}{P} & \multicolumn{3}{|c}{Q} & \multicolumn{3}{|c}{R}\\ \hline
Axis of strain & $x$ & $y$ & $z$ & $x$ & $y$ & $z$ & $x$ & $y$ & $z$\\ \hline
Dilation$\upperPhant$ & $h_{1}$ & $h_{1}$ & $h_{1}$ & $h_{2}$ & $h_{2}$ & $h_{2}$ & $h_{3}$ & $h_{3}$ & $h_{3}$\\[1ex]
Shear & $p^2$ & $\frac{1}{p^2}$ & 1 & $\frac1q$ & $q$ & 1 & $\frac{1}{r}$ & $r$ & 1\\[1ex]
Shear & 1 & $p$ & $\frac1p$ & 1 & $q$ & $\frac1q$  & 1 & $\frac{1}{r^2}$ & $r^2$\\[2mm] \hline
\end{tabular}
\end{center}
\caption{Modified version of Becker's first table}
\end{table}

\subsubsection{The uniaxial case}
Consider, again, a uniaxial initial stress of the form $\Biot = \dmatrs{0}{Q}{0}$. Then $\Biot$ can be decomposed into
\[
	\Biot \;=\; \underbrace{\frac{Q}{3} \dmatr{-1}{1}{0}}_{=:\Biot_1}
	\;+\; \underbrace{\frac{Q}{3} \dmatr{0}{1}{-1}}_{=:\Biot_2}
	\;+\; \underbrace{\frac{Q}{3} \dmatr{1}{1}{1}}_{=:\Biot_3}\,.
\]
\unitlength1.0mm
\begin{figure}[H]\centering
\begin{picture}(140,65)
  \put(4,0){\epsfig{file=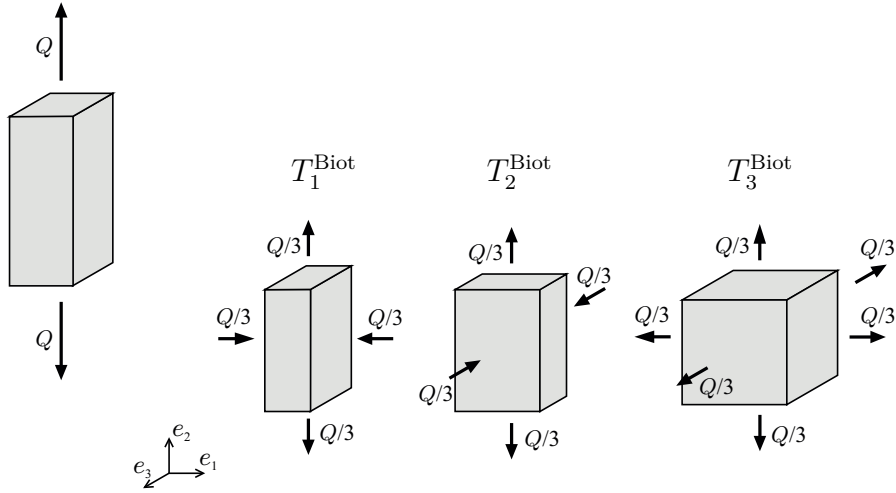, height=65mm, angle=0}}
  \put(42,41){$\Biot_1$}
  \put(68,41){$\Biot_2$}
  \put(100,41){$\Biot_3$}
\end{picture}
\caption{Load and deformation of the uniaxial case can be decomposed into two finite shear and a dilational mode.}
\label{figure:uniaxialDecomposition}
\end{figure}
\unitlength0.357mm
By Becker's assumption the two shear stresses $\Biot_1$ and $\Biot_2$ correspond to shear strains
\[
	U_1 = \dmatr{\frac1\alpha}{\alpha}{1} \quad \text{ and } \quad U_2 = \dmatr{1}{\alpha}{\frac1\alpha}
\]
respectively, while the volumetric stress $\Biot_3$ corresponds to a dilational strain
\[
	U_3 = \dmatr{h}{h}{h}\,.
\]
Then, according to the law of superposition, the strain corresponding to $\Biot$ is given by
\[
	U \;=\; U_1\cdot U_2\cdot U_3 \;=\; \dmatr{\frac{h}{\alpha}}{h\,\alpha^2}{\frac{h}{\alpha}}\,.
\]
The resulting principal stretch $h\,\alpha^2$ in direction of the applied load is referred to by Becker as the \Bquote{length of the strained mass} \Bref{becker:strainedMass} in his further computations for the uniaxial case.

Now we consider another uniaxial load given by $\Biothat = \dmatrs{P}{0}{0}$. Then $\Biothat$ can be decomposed in a number of different ways. For example we could choose, for symmetry reasons, a decomposition similar to that of $\Biot$, i.e.
\[
	\Biothat \;=\; \frac{P}{3}\dmatr{1}{-1}{0}
	\;+\; \frac{P}{3}\dmatr{1}{0}{-1}
	\;+\; \frac{P}{3}\dmatr{1}{1}{1}\,,
\]
yielding the strain
\[
	\Uhat \;=\; \dmatr{\alphahat}{\frac{1}{\alphahat}}{1} \cdot \dmatr{\alphahat}{1}{\frac{1}{\alphahat}} \cdot \dmatr{\hhat}{\hhat}{\hhat} \;=\; \dmatr{\hhat\,\alphahat^2}{\frac{\hhat}{\alphahat}}{\frac{\hhat}{\alphahat}}\,.
\]
It is also possible to decompose $\Biothat$ into shears coaxial to $\Biot_1$ and $\Biot_2$, i.e.
\[
	\Biothat \;=\; -\frac{2\,P}{3}\dmatr{-1}{1}{0}
	\;+\; \frac{P}{3}\dmatr{0}{1}{-1}
	\;+\; \frac{P}{3}\dmatr{1}{1}{1}\,.
\]
Note that the resulting strain is independent of this choice: since
\[
	-\frac{2\,P}{3}\dmatr{-1}{1}{0} \;=\; \frac{P}{3}\dmatr{1}{-1}{0} + \frac{P}{3}\dmatr{1}{-1}{0}\,,
\]
the law of superposition yields
\begin{align*}
	\Uhat \;&=\; \dmatr{\alphahat}{\frac{1}{\alphahat}}{1} \cdot \dmatr{\alphahat}{\frac{1}{\alphahat}}{1}
	\;\cdot\;\dmatr{1}{\alphahat}{\frac{1}{\alphahat}} \cdot \dmatr{\hhat}{\hhat}{\hhat} \;=\; \dmatr{\hhat\,\alphahat^2}{\frac{\hhat}{\alphahat}}{\frac{\hhat}{\alphahat}}
\end{align*}
in this case as well.

\subsubsection{Decomposition along fixed axes}
The additive decomposition of the stress tensor into a volumetric stress and two shears along fixed axes can also be expressed in basic algebraic terms. We will identify the set of all diagonal matrices in $\Rnn$ with the Euclidean space $\R^3$ in the canonical way.
Since the set
\[
	\mathcal{B} \;=\; \Big\{ \vect{-1}{1}{0}\,,\; \vect{0}{1}{-1}\,,\; \vect{1}{1}{1} \Big\}\,,
\]
which corresponds to two shear stresses and a volumetric stress in diagonal form, is a basis of $\R^3$, every $\vects PQR \in\R^3$ can be written as
\begin{equation}\label{eq:vectorDecomposition}
	\vect PQR \;=\; A\,\vect{-1}{1}{0} \,+\, B\,\vect{0}{1}{-1} \,+\, C\,\vect111 \;=\; \signmatr{-1&0&1\\ 1&1&1\\ 0&-1&\hphantom{-}1}\cdot \vect ABC\,.
\end{equation}
For given $P,Q,R\in\R$ we can therefore find the coefficients $A,B,C$ of the decomposition
\[
	\dmatr PQR \;=\; A\,\signdmatr{-1}{\hphantom{-}1}{\hphantom{-}0} \,+\, B\,\signdmatr{0}{\hphantom{-}1}{-1} \,+\, C\,\dmatr111
\]
by rewriting equation \eqref{eq:vectorDecomposition} to read
\begin{align*}
	\vect ABC \;&=\; \signmatr{-1&0&1\\ 1&1&\hphantom{-}1\\ 0&-1&1}\inv\cdot \vect PQR \;=\; \frac13 \,\signmatr{-2&1&1\\ 1&1&-2\\ 1&\hphantom{-}1&1}\cdot \vect PQR \;=\; \frac13\,\vect{-2P+Q+R}{P+Q-2R}{P+Q+R}\,.
\end{align*}
The decomposition obtained in this way is the same as given in \eqref{eq:stressDecomposition}.

\section{Geometry and statical analysis of finite shear deformation}
\label{section:geometryOfShear}

\unitlength1.0mm
\fboxsep5mm
\fboxrule0.3mm
In this section we investigate the geometric and static motivation of the relation between shear deformations and shear stresses assumed by Becker. The unit cube $\Omega_0$ aligned to the orthogonal coordinate system $e_1,e_2,e_3$ in the reference configuration is considered again. We distinguish six variants of orthogonal deformation along the edges of the unit cube such that one edge preserves its length, compare Fig. \ref{FiniteShear3D}.
\begin{figure}[H]\centering
 \begin{picture}(115,98)
  \put(0,0){\epsfig{file=\imagepath/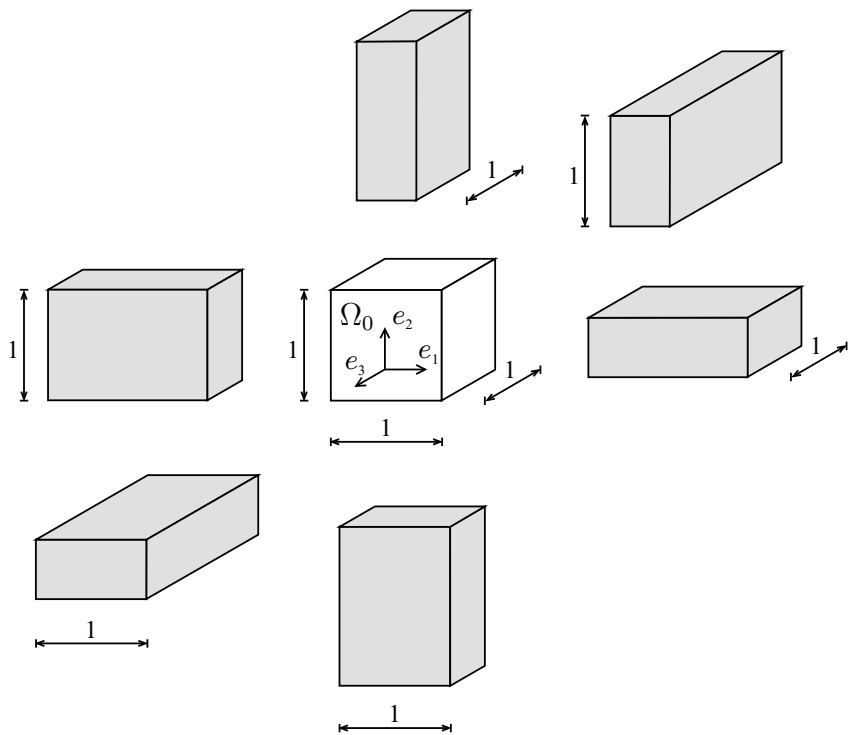, width=112mm, angle=0}}
  \put(44.3,54.2){$\Omega_0$}
 \end{picture}
 \caption{Six variants of plane, orthogonal deformation along the edges of a unit cube.}
 \label{FiniteShear3D}
\end{figure}
Specifying the deformation $\Phi_\alpha:\Omega_0\to\Omega_1$ as pure finite shear with ratio $\alpha$, the volume in the actual configuration $\Omega_1$ is invariant. Thus, stretching one edge in the plane of deformation by $\alpha$ must result in a contraction $1/\alpha$ of the corresponding edge, see also Fig. \ref{FiniteShear1}. The volume invariance holds for any intermediate configuration $\Omega_i$ and goes along with the multiplicative and nonlinear behaviour of finite deformation. We consider the intermediate configuration defined by the symmetry condition
\begin{align}\label{SymPhi}
\Phi_{\!\!\sqrt{\alpha}}: \Omega_0 \mapsto  \Omega_i \, , \quad \Phi_{\!\!\sqrt{\alpha}}: \Omega_i \mapsto
\Omega_1.
\end{align}
\begin{figure}[H]\centering
 \begin{picture}(140,55)
  \put(0,0){\epsfig{file=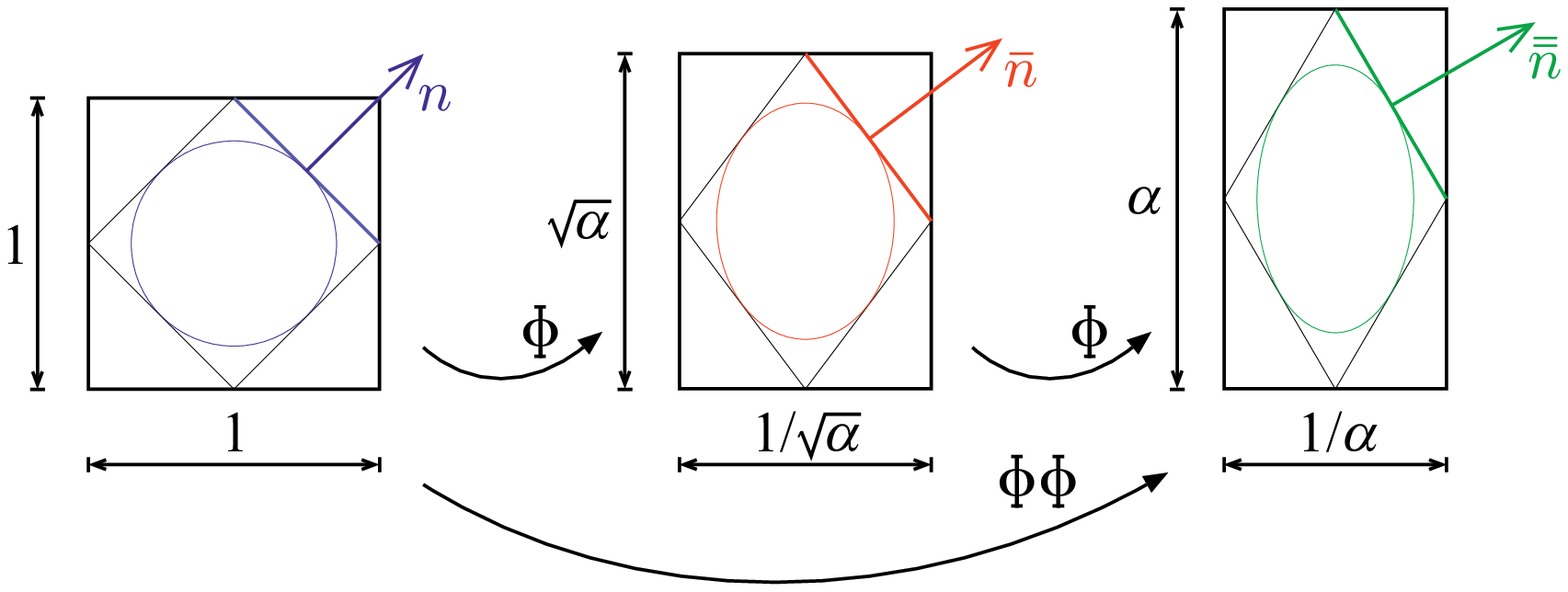, width=130mm, angle=0}}
  \put(20,44){$\Omega_0$}
  \put(67,48){$\Omega_i$}
  \put(112,52){$\Omega_1$}
 \end{picture}
 \caption{Finite shear of a unit cube.}
 \label{FiniteShear1}
\end{figure}
Let us inscribe a circle and a square into $\Omega_0$ as drawn on the left hand side of Fig. \ref{FiniteShear1}. Then $\Phi_{\!\!\sqrt{\alpha}}$ deforms the circle into an ellipse\footnote{The equation of the shear ellipse is also mentioned by Becker \Bref{becker:shearEllipse}.} and the square into a rhombus. Increasing values for
$\alpha$ decrease the angle $\psi$ between the normal vector
\begin{align}\label{Normal_nbar}
\bar{n}=\frac{1}{\sqrt{\alpha+1/\alpha}} \matr{\sqrt{\alpha}\\ 1/\sqrt{\alpha}}
=\frac{1}{\sqrt{\alpha^2+1}} \matr{\alpha\\1}
\end{align}
and the horizontal direction $e_1$, compare Fig. \ref{FiniteShear2a}.
\begin{figure}[H]\centering
 \begin{picture}(84,37)
  \put(0,0){\epsfig{file=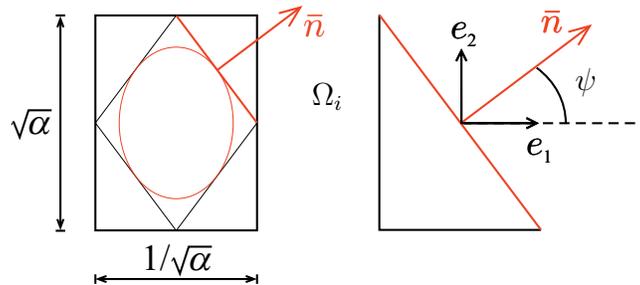, width=85mm, angle=0}}
  \put(40,24){$\Omega_i$}
  \put(75,26){$\psi$}
 \end{picture}
 \caption{Inclination $\psi$ of the normal $\bar n$ in the intermediate configuration $\Omega_i$.}
 \label{FiniteShear2a}
\end{figure}
Becker denotes the scalar product $\innerproduct{n,e_i}$ between the normal $n$ of a plane and the coordinate axis $e_i$ as \Bquote{direction cosines} $n_i$ of the plane \Bref{becker:directionCosines}. For the normal $\bar n$ we find
\begin{align}\label{n1_n2}
n_1=\innerproduct{\bar n, e_1}=\sqrt{\alpha^2+1} \quad , \, n_2=\innerproduct{\bar n, e_2}=\sqrt{1 +
1/\alpha^2} \quad\Longrightarrow\quad n_1=\alpha \, n_2 \,,
\end{align}
thus the inclination $\psi$ of the normal $\bar n$ is given by
\begin{align}\label{Inclination_psi}
\cot{\psi}=\frac{n_1}{n_2}=\alpha \,.
\end{align}
Since here the direction $e_1$ is the contractile axis, it follows from the calculations in section \ref{section:planesOfNoDistortion} that $\bar n$ is normal to the plane of no distortion. The rhombus in Fig. \ref{FiniteShear2a} therefore describes the two directions of no distortion. We explain this fact geometrically for the line perpendicular to $\bar n$ in Fig. \ref{FiniteShear3}. Drawing this line onto the deformed body $\Omega_1$ and then releasing the deformation leads to a rotated but undistorted line in $\Omega_0$. Thus, the intermediate configuration $\Omega_i$ displays an important geometrical aspect in a pure shear deformation.
\begin{figure}[H]\centering
 \begin{picture}(140,55)
  \put(0,0){\epsfig{file=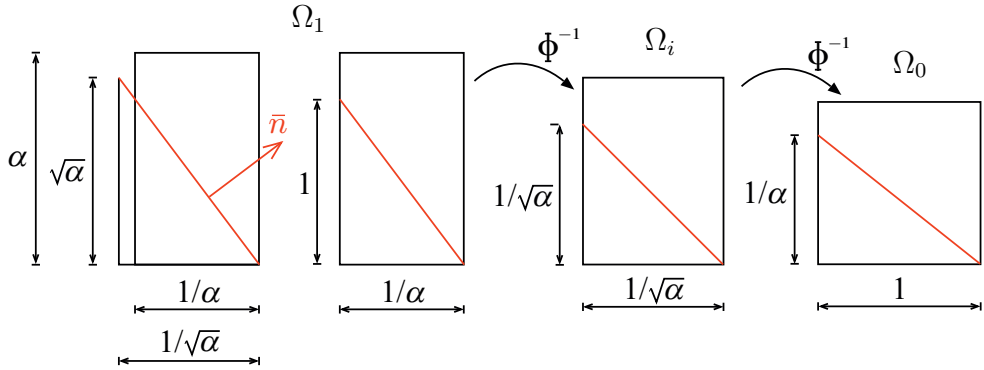, width=130mm, angle=0}}
  \put(38,45){$\Omega_1$}
  \put(85,42){$\Omega_i$}
  \put(118,39){$\Omega_0$}
 \end{picture}
 \caption{The plane of no distortion preserves its length after release of the finite shear deformation.}
 \label{FiniteShear3}
\end{figure}
These properties, Becker argues, suggest that the normal component of the Cauchy stress in the directions of no distortion has to vanish, i.e. $\mathcal{N}=0$, \Bquote{for otherwise there could be no planes of zero distortion} \Bref{becker:planesOfNoDistortionAndTangentialStrain}.
Combining this assumption with the equality $n_1/n_2 = \alpha$ from \eqref{n1_n2}, it follows that the components of the principal Cauchy stresses need to fulfil
\begin{align}\label{QBedingung}
-\sigma_1\,\alpha = \sigma_2/\alpha \,,\quad \sigma_3=0\,,
\end{align}
where $-\sigma_1\,\alpha$ and $\sigma_2/\alpha$ are the resultant loads; note that here, $e_1$ is the contractile axis. From the principal Cauchy stresses in equation \eqref{QBedingung} he infers that the action of \Bquote{two equal loads [\dots] of opposite signs at right angles to one another} \Bref{becker:shearAxiom} must correspond to a simple finite shear.\footnote{A more detailed description of the decomposition of the Cauchy stress Becker employs to arrive at this result can be found in Appendix \ref{section:basicDecomposition}.}

Next, we discuss one of the six variants in Fig. \ref{FiniteShear3D} without loss of generality. The deformation in Fig. \ref{FiniteShear4} results in principal stresses of horizontal and vertical direction. Defining the Cauchy stress as \emph{load per actual area}, the quantity $- \sigma_1\,\alpha$ represents a compressive force in the horizontal axis and $\sigma_2 / \alpha$ is a tensional force in the vertical direction. Equation \eqref{QBedingung} yields that these forces are of the same amount, say $Q$. Then, for a static analysis, we cut the body in $\Omega_1$ by a plane of no distortion.
The lower resp. upper part of the body is in an equilibrium balanced by
\begin{align}\label{EquilibriumPOND}
\mathcal{T} = Q
         \matr{-1/\alpha \\
           1} \,, \, {\rm resp.} \quad
\mathcal{T} = Q \matr{
           1-1+1/\alpha \\
           -1}
     = Q \matr{
           1/\alpha \\
           -1} \,.
\end{align}
The scalar product of $\mathcal{T}$ with $\bar n$ from equation \ref{Normal_nbar}
vanishes and thus, the equilibrium is free of normal components in the plane of no
distortion as required, compare Fig. \ref{FiniteShear4}.
\begin{figure}[H]\centering
 \begin{picture}(140,75)
  \put(0,0){\epsfig{file=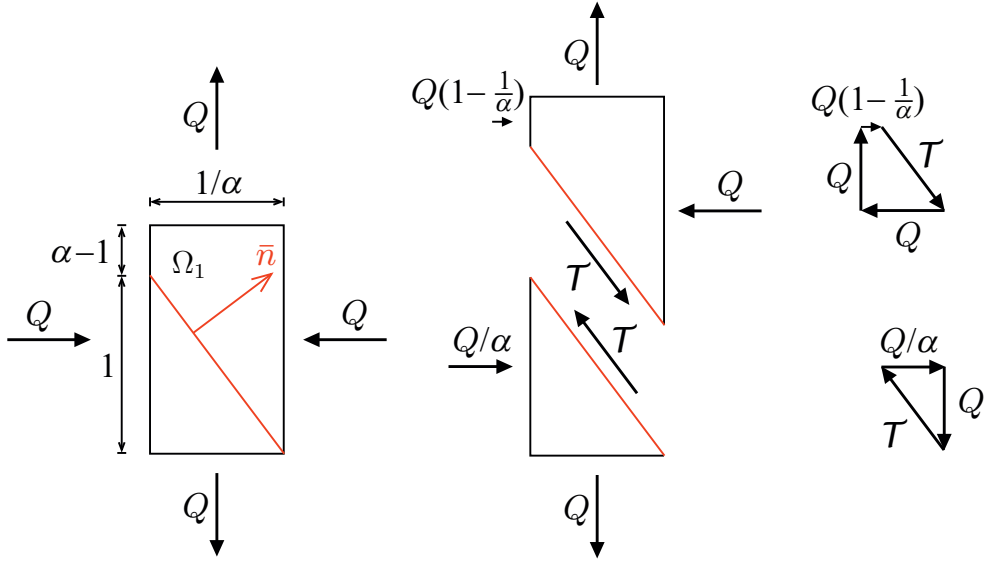, width=130mm, angle=0}}
  \put(22,38){$\Omega_1$}
 \end{picture}
 \caption{Static analysis in the actual configuration along the cut of a plane of no distortion.}
 \label{FiniteShear4}
\end{figure}
Becker concludes that the six variants of simple finite shearing in Fig. \ref{FiniteShear3D} are caused by pairwise forces as sketched in Fig. \ref{FiniteShear3DForces}. This is in concordance with one of five possible invariance conditions for the definition of pure shear \cite{blinowski1998}, namely that the shear tensor is a planar deviator. Note that in the case of isotropic material it is obvious to claim that $P=Q=R$.
\begin{figure}[H]\centering
 \begin{picture}(140,130)
  \put(0,0){\epsfig{file=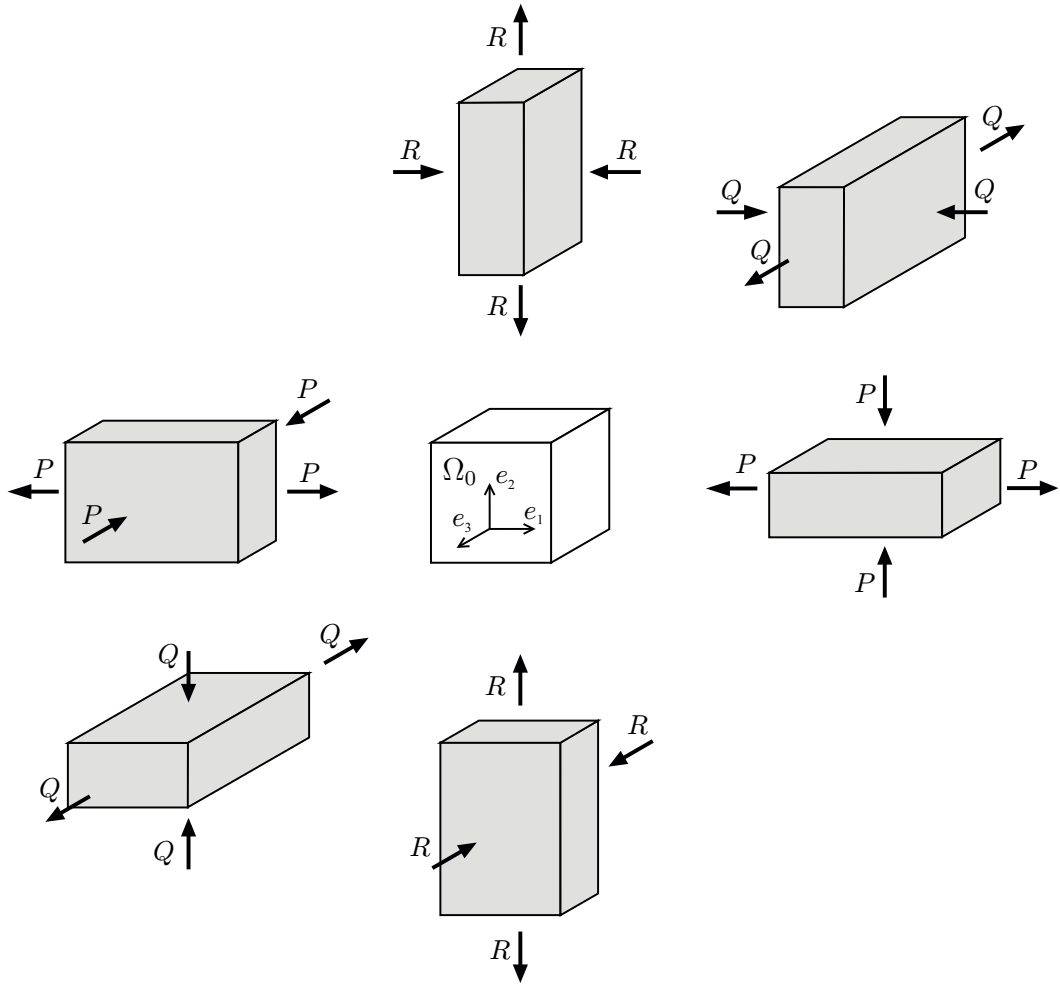, width=140mm, angle=0}}
  \put(58,67){$\Omega_0$}
 \end{picture}
 \caption{Six variants of finite pure shear and the corresponding loads.}
 \label{FiniteShear3DForces}
\end{figure}
\subsection{Mohr's stress circle for finite shear}

In this section we discuss the plane of no distortion in the context of \emph{Mohr's stress circle}\footnote{Mohr \cite{mohr1882} published his work on the representation and transformation of two-dimensional stresses by means of a circle in 1882. In the context of beams, Culmann \cite{culmann1866} published a similar idea in 1866, using a different proof.}. The Mohr circle is a graphical method to determine the stress components acting on a plane rotated against the coordinate system. Given any symmetric stress tensor in $\R^{2\times2}$, Mohr's stress circle allows for a graphical solution of the spectral decomposition. Vice versa, knowing Mohr's stress circle and its spectral decomposition, the tensor components are graphically given for arbitrary but orthogonal coordinate systems, e.g. for a coordinate system in alignment with the plane of no distortion.

It is important to note that the plane of maximal shear stress $\rho_{max}$ does not coincide with the plane of no distortion for finite shear. Considering the loading in Fig. \ref{FiniteShear4} together with equation \eqref{QBedingung} results in the principal stresses
\begin{align}\label{FiniteShearStress}
\sigma_1 = - Q / \alpha \,,\quad \sigma_2= Q \, \alpha.
\end{align}
Therefore, Mohr's stress circle for the finite pure shear loading is drawn in Fig.
\ref{Mohr}.
\begin{figure}[H]\centering
 \begin{picture}(140,68)
  \put(0,0){\epsfig{file=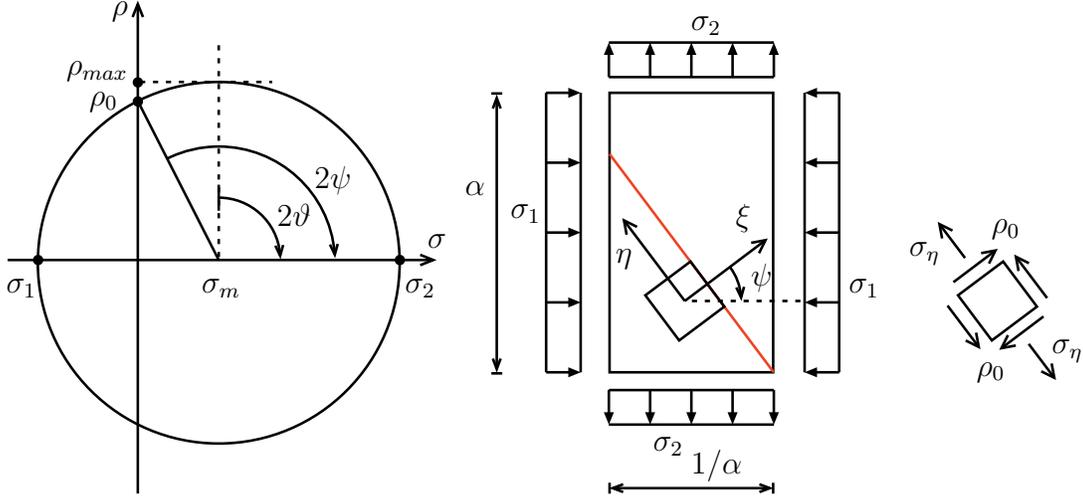, height=66mm, angle=0}}
  \put(14,64){$\rho$}
  \put(8,56){$\rho_{max}$}
  \put(11,52){$\rho_{0}$}
  \put(41,41){$2\psi$}
  \put(36,36){$2\vartheta$}
  \put(56,33){$\sigma$}
  \put(0,27){$\sigma_1$}
  \put(26,27){$\sigma_m$}
  \put(53,27){$\sigma_2$}
  \put(61,40){$\alpha$}
  \put(91,3){$1/\alpha$}
  \put(67,37){$\sigma_1$}
  \put(112,27){$\sigma_1$}
  \put(91,62){$\sigma_2$}
  \put(86,6){$\sigma_2$}
  \put(97,36){$\xi$}
  \put(81,31){$\eta$}
  \put(99,28){$\psi$}
  \put(131,35){$\rho_0$}
  \put(129,16){$\rho_0$}
  \put(120,32){$\sigma_{\eta}$}
  \put(139,19){$\sigma_{\eta}$}
 \end{picture}
 \caption{Mohr's circle for Cauchy stresses in simple finite shear loading.}
 \label{Mohr}
\end{figure}
Becker introduces the quantity $2s=\alpha-1/\alpha$ as the \emph{\enquote{amount of the shear}} \cite{becker1892}. Thus, the center of the circle is shifted from the origin of
axis by
\begin{align}\label{MohrCenter}
\sigma_m = \frac{\sigma_1 + \sigma_2}{2} =
\frac{Q}{2}\left(\alpha-\frac{1}{\alpha}\right) = Q \, s.
\end{align}
For $\alpha>1$ the angle $\vartheta = \pi/4$ pointing to the plane of maximum shear
stress $\rho_{max}$ does not coincide with the angle $\psi$ describing the plane of
no distortion. Demanding $\sigma_{\xi}=0$ in the plane of no distortion determines
the rotation angle of the $\xi$-$\eta$ coordinate system by
\begin{align}\label{psi}
\psi = \half \arccos{\left(\frac{\alpha^2-1}{\alpha^2+1}\right)}=\arccot(\alpha)\,,
\end{align}
which is in accordance with equation \eqref{Inclination_psi}. For $\psi=\arccot(\alpha)$
the stress components are given by
\begin{align}\label{pondstress}
\sigma_{\xi}=0 \,,\quad \sigma_{\eta}=Q \, \frac{\alpha^2-1}{\alpha} \,,\quad
\sigma_{\xi\eta}=\rho_0=Q \,.
\end{align}
Thus, in the plane of no distortion the shear stress $\rho_0$ is a simple function
of the loading $Q$ and independent of $\alpha$.
\subsection{Tangential force and failure criteria in finite shear}\label{section:maxTangLoadPOND}

To discuss the tangential force in the spirit of continuum mechanics, we analyze the
equilibrium of stresses in the vicinity of a point. In case of simple finite shear
the surrounding of a point is a circle in $\Omega_0$ resp. a shear ellipse in
$\Omega_1$, compare Fig. \ref{FiniteShear1}. Considering the diameter of the circle
to be $1$, the diameter $r$ of the ellipse is given in terms of $\alpha$ and the
direction cosines $n_1$, $n_2$ from equation \ref{n1_n2} by
\begin{align}\label{EllipsisDiameter}
\frac{1}{r^2} = \alpha^2 \, n_2^2 + \frac{n_1^2}{\alpha^2} \,.
\end{align}
Considering the finite shear loading from Fig. \ref{FiniteShear4} with principal
Cauchy stresses from equations \ref{FiniteShearStress} gives a resultant force ${\mathcal
R}$ with constant magnitude\footnote{Becker uses only $Q/3$ as loading in his
footnote on rupture, page 339} $Q$ on any line cutting the ellipsis through the mid
point
\begin{align}\label{EllipsisMagnitudeR}
\displaystyle {\mathcal R}^2 = r^2 \, \norm{\sigma \, n}^2 = r^2 \, \norm{\left(
                                                        \begin{array}{c}
                                                          -\frac{Q}{\alpha} \, n_1 \\
                                                          Q \alpha \, n_2 \\
                                                        \end{array}
                                                      \right)}^2
               = r^2 \, Q^2 \, \left( \frac{n_1^2}{\alpha^2} + \alpha^2 \, n_2^2 \right) = Q^2
               \,.
\end{align}
The normal $n$ of the line defines the direction cosines $n_1$ and $n_2$ as
explained in equation \ref{n1_n2}. Next, we investigate the magnitude of ${\mathcal R}$
normal to the line, which is given by
\begin{align}\label{EllipsisMagnitudeN}
\displaystyle {\mathcal N}^2 = r^2 \, \innerproduct{\sigma \, n , \, n}^2 = r^2 \,
\innerproduct{\left(\begin{array}{c}
                    -\frac{Q}{\alpha} \, n_1 \\
                    Q \alpha \, n_2 \\
                    \end{array}
                    \right) , \left(\begin{array}{c}
                    n_1 \\
                    n_2 \\
                    \end{array}
                    \right)}^2
               = r^2 \, Q^2 \, \left( \frac{- n_1^2}{\alpha} + \alpha \, n_2^2 \right)^2 \,.
\end{align}
The magnitude of stress ${\mathcal T}$ tangential to the plane follows from the Pythagorean theorem by
\begin{align}\label{EllipsisMagnitudeT}
\displaystyle {\mathcal T}^2 = {\mathcal R}^2 - {\mathcal N}^2 \,.
\end{align}
Thus the plane of maximal tangential force is due to ${\mathcal N}^2 = 0$ and we
conclude from equation \eqref{EllipsisMagnitudeN}
\begin{align}\label{PondByMaxTangLoad}
\frac{- n_1^2}{\alpha} + \alpha \, n_2^2 = 0  \quad\Longleftrightarrow\quad n_1^2 = n_2^2
\alpha^2 \,.
\end{align}
The planes of no distortion fulfil equation \ref{PondByMaxTangLoad} and attend Becker's
discussion on failure criteria\footnote{The role of the
plane of no distortion in failure criteria is further discussed by Becker in
\cite{becker1904} and \cite{becker1907}.} \Bquote{Rupture by shearing is determined by maximum tangential load, not {\rm [Cauchy]} stress} \Bref{becker:ruptureCriterion}, c.f. section \ref{section:maximumTangentialLoadBecker}. In the present example the maximum shear
stress appears in a cut inclined with $\psi=45^\circ$ to the horizontal axis which does not, however, maximize the tangential load. The situation is illustrated in Fig. \ref{EllipseSchnitte}.
\begin{figure}[H]\centering
\begin{picture}(140,53)
  \put(4,0){\epsfig{file=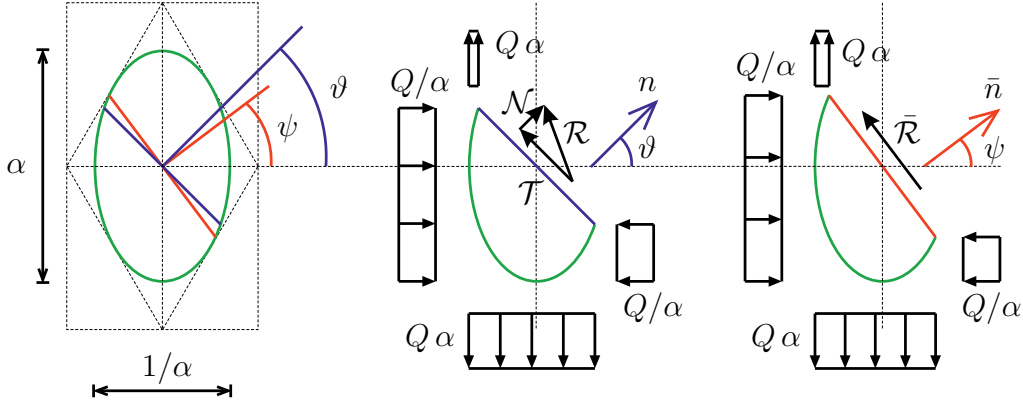, height=53mm, angle=0}}
  \put(0,30){$\alpha$}
  \put(18,3){$1/\alpha$}
  \put(36,35){$\psi$}
  \put(43,40){$\vartheta$}
  \put(51,41){$Q/\alpha$}
  \put(53,7){$Q \, \alpha$}
  \put(65,46){$Q \, \alpha$}
  \put(82,11){$Q/\alpha$}
  \put(84,41){$n$}
  \put(130,40){$\bar{n}$}
  \put(97,43){$Q/\alpha$}
  \put(99,7){$Q \, \alpha$}
  \put(111,45){$Q \, \alpha$}
  \put(127,11){$Q/\alpha$}
  \put(74,34){${\mathcal R}$}
  \put(66,38){${\mathcal N}$}
  \put(68,26){${\mathcal T}$}
  \put(118,34){$\bar{\mathcal R}$}
  \put(130,32){$\psi$}
  \put(84,32){$\vartheta$}
\end{picture}
\caption{Equilibrium and resultant ${\mathcal R}$ resp. $\overline{\mathcal R}$ on lines cutting the finite shear ellipsis; note that $\overline{\mathcal R}^2 > {\mathcal T}^2$.}
\label{EllipseSchnitte}
\end{figure}
Note that the shear stress $\rho_0 = Q/1=Q$ in the plane of no distortion is lower than the Tresca and the von Mises stresses for this kind of loading:
\begin{align}
\sigma^{\rm Tresca}_{\rm y}&=Q \, (\alpha + 1/\alpha)\,, \label{Tresca}\\[2mm]
\sigma^{\rm Mises}_{\rm y}&=Q \, \sqrt{\alpha^2+1+1/\alpha^2}\,, \label{vonMises}\\[2mm]
\sigma_{\xi\eta} = \sigma^{\rm Becker}_{\rm y}&=Q\,. 
\end{align}
Thus, the inequality $\sigma^{\rm Becker}_{\rm y}<\sigma^{\rm Mises}_{\rm y}<\sigma^{\rm Tresca}_{\rm y}$ for $\alpha>0$ shows that
the tangential force in the plane of no distortion represents a more conservative lower bound
as failure criterion.

It is worth mentioning that the Tresca and the von Mises stresses would account for
both the loading $Q$ and the deformation $\alpha$. However, decoupling the failure
criteria from the deformation suggests a basic model, which is also simple from an
experimental point of view.

\unitlength0.357mm
\fboxsep1.8mm
\fboxrule0.3mm

\section{The axiomatic approach}\label{section:deduction}
Our aim in this section is to formulate Becker's stress-stretch relation for ideally elastic, isotropic materials in terms of the Biot stress tensor $\Biot$ and the right Biot stretch tensor $U = \sqrt{F^TF}$, i.e. to find a stress response function $U\mapsto \Biot(U)$ or an inverse response function $\Biot\mapsto U(\Biot)$ which fulfils Becker's assumptions listed in section \ref{section:assumptions}.
In order to deduce such a law of ideal elasticity we will introduce a number of \emph{axioms} corresponding to Becker's \emph{assumptions} to uniquely characterize this stress-stretch relationship. In this we will, at first, closely follow Becker's approach, utilizing all of the given axioms. However, we will later show that the same results can be deduced with only a subset of those assumptions. Furthermore, in order to obtain a more general result, we will formulate our computations in terms of an arbitrary stress tensor instead of only the Biot stress.

\subsection{The basic axioms for an isotropic stress-stretch relation}\label{subsection:basicAxioms}

Let $\stress:\PSymn\to\Symn,\: \stretch\mapsto \stress(\stretch)$ denote a matrix function which maps the set of positive definite symmetric matrices $\PSymn$ to the set of symmetric matrices $\Symn$. In accordance with our interpretation of $\stress$ as a stress response function relation we will refer to the argument $\stretch$ as the stretch and to $\stress(\stretch)$ as the stress tensor. Furthermore, if the mapping is invertible we will, for given $\hat\stress\in\Symn$, often write $\stretch(\hat\stress)$ to denote the unique $\stretch\in\PSymn$ with $\hat\stress = \stress(\stretch)$.\\
Of course we also assume that $\stress$ fulfils the axioms stated in this section.

The first two axioms are common postulates for an arbitrary stress-stretch relation.
\axiombox{Axiom 0.1: Continuous stress response function}{
The function $\stress:\PSymn\to\Symn,\: \stretch\mapsto \stress(\stretch)$ is continuous.
}

Note that this is a weakened version of Becker's assumption that the function mapping stretch and stress is analytic.

\axiombox{Axiom 0.2: Unique stress free reference configuration}{
The equivalence
\[
	\stress(\stretch) = \dmatr{0}{0}{0} = 0 \quad\Longleftrightarrow\quad \stretch=\dmatr{1}{1}{1} = \id\,,
\]
holds.
}

Axiom 2 states that the undeformed (and possibly rotated) reference configuration is the only stress free configuration. Again, this is a weakened version of one of Becker's assumptions, namely that the stress response function is invertible.

Another basic assumption is that of \emph{isotropy}: the response of many materials can be idealized to be independent of the direction of applied stresses.
\axiombox{Axiom 0.3: Isotropy}{
The equality
\[
	\Biot(Q^T \stretch Q) = Q^T \,\Biot(\stretch)\, Q
\]
holds for all $Q\in\On$ and $\stretch\in\PSymn$.
}
Note that isotropy is sometimes defined for $Q\in\SOn$ only. However, since $-Q\in\SOn$ for $Q\in\On\setminus\SOn$, even under this narrower definition we find
\begin{align*}
	\stress(Q^T\,\stretch\,Q) \;&=\; \stress((-Q^T)\,\stretch\,(-Q)) \;=\; (-Q^T) \,\stress(\stretch)\, (-Q) \;=\; Q^T \,\stress(\stretch)\, Q
\end{align*}
for all $Q\in\On\setminus\SOn$ as well.

Since the assumption of isotropy implies that $\stretch$ and $\stress(\stretch)$ are coaxial \cite[Theorem 4.2.4]{Ogden83} for all $\stretch\in\PSymn$, Axiom 0.3 restricts the possible choices for the considered stretch-stress pair $\stretch,\stress$ to coaxial pairs of tensors.
However, this axiom allows us to formulate the other axioms regarding strains and stresses in terms of the principal stretches and principal stresses only \cite{hutchinson1981}:
\begin{quote}
	\emph{\enquote{For an isotropic nonlinear elastic solid, the principal directions of Cauchy stress $\sigma$ must coincide with the axes of the Eulerian strain ellipsoid. Also, to fully specify the state of strain in a material element we need only know the three principal stretches $\lambda_i$ relative to some reference configuration and the principal directions of strain. Thus, the constitutive law is completely determined once the relations between the principal components of Cauchy stress $\sigma_i$ and principal stretches $\lambda_i$ are known.}}
\end{quote}
Since any given $\stretch\in\PSymn$ is unitarily diagonalizable we can write $\stretch$ as
\[
	\stretch = Q^T \; \dmatr{\lambda_1}{\lambda_2}{\lambda_3} \; Q
\]
with $Q\in\On$, where $\lambda_i$ denote the (positive) eigenvalues of $\stretch$. Then $\stress(\stretch)$ is given by
\[
	\stress(\stretch) = Q^T\cdot\stress(Q \stretch Q^T)\cdot Q = Q^T\cdot\stress \Big(\dmatr{\lambda_1}{\lambda_2}{\lambda_3}\Big)\cdot Q\,.
\]
This allows us to focus on stretches given in a diagonal representation, i.e. stretches of the form
\[
	\stretch = \diag(\lambda_1,\lambda_2,\lambda_3) = \dmatr{\lambda_1}{\lambda_2}{\lambda_3}
\]
with $\lambda_1,\lambda_2,\lambda_3 \in\R^+$.

\subsection{Becker's three main axioms}

The first two main axioms for our law of ideal isotropic elasticity refer to two special cases of deformation, pure shear and pure volumetric dilation.
\axiombox{Axiom 1: Pure shear stresses correspond to pure shear stretches}{
For every $\alpha\in\R^+$ there exists $s\in\R$ such that
\[
	\stretch=\matr{\alpha&0&0\\0&\frac{1}{\alpha}&0\\0&0&1} \qquad\Longleftrightarrow\qquad \stress(\stretch) = \matr{s&0&0\\0&-s&0\\0&0&0}\,.
\]
}
If $\stress$ is an isotropic tensor function (Axiom 0.3), then Axiom 1 can also be stated in a more general way. Let $\stretch = \matrs{B_{11}&B_{12}&0\\ B_{12}&B_{22}&0\\ 0&0&1}$ with $\det \matrs{B_{11}&B_{12}\\B_{12}&B_{22}} = 1$. Then the two eigenvalues of $\matrs{B_{11}&B_{12}\\B_{12}&B_{22}}$ are mutually reciprocal, thus $\stretch$ can be diagonalized to
\[
	\stretch \;=\; \matr{Q_{11}&Q_{12}&0\\ Q_{12}&Q_{22}&0\\ 0&0&1}^T \cdot\, \dmatr{\alpha}{\frac1\alpha}{1} \,\cdot\, \matr{Q_{11}&Q_{12}&0\\ Q_{12}&Q_{22}&0\\ 0&0&1}
\]
with $\matr{Q_{11}&Q_{12}\\ Q_{12}&Q_{22}}\in\OO(2)$ and $\alpha\in\R^+$. Using Axiom 1 and Axiom 0.3 we compute
\begin{align*}
	\stress(\stretch) \;&=\; \matr{Q_{11}&Q_{12}&0\\ Q_{12}&Q_{22}&0\\ 0&0&1}^T \cdot\, \stress\Big( \dmatr{\alpha}{\frac1\alpha}{1} \Big) \cdot \matr{Q_{11}&Q_{12}&0\\ Q_{12}&Q_{22}&0\\ 0&0&1}\\[2mm]
	&=\; \matr{Q_{11}&Q_{12}&0\\ Q_{12}&Q_{22}&0\\ 0&0&1}^T \cdot \dmatr{s}{-s}{0} \cdot \matr{Q_{11}&Q_{12}&0\\ Q_{12}&Q_{22}&0\\ 0&0&1} \;=\; \matr{A_{11}&A_{12}&0\\ A_{12}&A_{22}&0\\ 0&0&1}
\end{align*}
with $\tr \matrs{A_{11}&A_{12}\\ A_{12}&A_{22}} =\tr\dmatrs{s}{-s}{0} = 0$. In the isotropic case, Axiom 1 may therefore be equivalently stated as follows (c.f. Fig. \ref{figure:pureshear3D} on page \pageref{figure:pureshear3D}): for every $\matrs{B_{11}&B_{12}\\B_{12}&B_{22}} \in \SL(2)$ there exists $\matrs{A_{11}&A_{12}\\A_{12}&A_{22}}\in\sl(2)$ such that
\[
	\stretch=\matr{B_{11}&B_{12}&0\\ B_{12}&B_{22}&0\\ 0&0&1} \qquad\Longleftrightarrow\qquad \stress(\stretch) = \matr{A_{11}&A_{12}&0\\ A_{12}&A_{22}&0\\ 0&0&1}\,,
\]
where $\SL(n)$ denotes the group of all $X\in\GL(n)$ with $\det X = 1$ and $\sl(n)$ is the corresponding Lie algebra of all trace free matrices in $\R^{n\times n}$.

To further understand the relation between shear stress and shear stretch constituted by Axiom 1 we consider two examples. First, assume that the stress tensor $\stress$ corresponding to a stretch tensor $\stretch$ is a trace free \emph{pure shear stress} of the form
\[
	\stress(\stretch) = \matr{0&s&0\\s&0&0\\0&0&0}\,.
\]
The eigenvalues of $\stress(\stretch)$ are the principal stresses $\stress_1=s$,\, $\stress_2=-s$,\, $\stress_3=0$, thus $\stress(\stretch)$ can be diagonalized to
\[
	\stress(\stretch) = Q^T \dmatr{s}{-s}{0} Q
\]
with $Q\in\SOn$. If $\stretch\mapsto\stress(\stretch)$ is a surjective isotropic mapping satisfying Axiom 1, then
\begin{align*}
	\qquad Q\,\stress(\stretch)\,Q^T = \dmatr{s}{-s}{0} \;\;&\Longrightarrow\;\; \stress(Q\,\stretch\,Q^T) = \dmatr{s}{-s}{0}\\[2mm]
	&\Longrightarrow\quad\;\; Q\,\stretch\,Q^T = \dmatr{\alpha}{\frac1\alpha}{1} \;\;\Longrightarrow\;\; \stretch = Q^T\dmatr{\alpha}{\frac1\alpha}{1}\,Q
\end{align*}
for some $\alpha\in\R^+$. Thus the stretch $\stretch$ corresponding to a pure shear stress is a \emph{pure shear} of the form
\[
	\stretch \;=\; Q^T \matr{\alpha&0&0\\0&\frac{1}{\alpha}&0\\0&0&1} Q \;=\; \matr{B_{11}&B_{12}&0\\B_{12}&B_{22}&0\\0&0&1}%
\]
with $\det\matrs{B_{11}&B_{12}\\B_{12}&B_{22}} = 1$.

Now assume that a deformation gradient $F$ is a \emph{simple glide} of the form
\[
	F = \matr{1&\gamma&0\\0&1&0\\0&0&1}\,.
\]
Then the right Biot stretch tensor $U$ has the form
\[
	U = \sqrt{F^TF} = \matr{1&\gamma&0\\\gamma&\gamma^2+1&0\\0&0&1}^\frac12\,.
\]
Since $1$ is an eigenvalue of $U$ and $\det U = \det F = 1$, the remaining eigenvalues of $U$ must be of the form $\lambda_1=\alpha$ and $\lambda_2=\frac{1}{\alpha}$. Thus the principal stretches of the deformation are $\lambda_1=\alpha$, $\lambda_2=\frac1{\alpha}$, $\lambda_3=1$ for some $\alpha\in\R^+$ and $U$ can be diagonalized to
\[
	U = Q^T \matr{\alpha&0&0\\0&\frac1{\alpha}&0\\0&0&1} Q
\]
with $Q\in\On$. Then, if the function $U\mapsto\stress(U)$ mapping $U$ to a stress tensor $T$ is isotropic and satisfies Axiom 1, we can compute
\begin{align*}
	\stress(U) \;&=\; \stress(Q^T  \matr{\alpha&0&0\\0&\frac1{\alpha}&0\\0&0&1} Q) \;=\; Q^T \:\stress\,\Big( \matr{\alpha&0&0\\0&\frac1{\alpha}&0\\0&0&1} \Big)\: Q \;=\; Q^T \matr{s&0&0\\0&-s&0\\0&0&0} Q
\end{align*}
for some $s\in\R$. Therefore the stress tensor $\stress$ corresponding to a simple glide is a pure shear stress.\\

The second axiom relates \emph{spherical stresses}, i.e. purely normal stresses with the same magnitude in each direction, to \emph{volumetric stretches}, i.e. uniform stretches in all directions.
\axiombox{Axiom 2: Spherical stresses correspond to volumetric stretches}{
For every $\lambda\in\R^+$ there exists $a\in\R$ such that
\[
	\stretch=\lambda\cdot\id = \dmatr{\lambda}{\lambda}{\lambda} \qquad\Longleftrightarrow\qquad \stress(\stretch) = \dmatr{a}{a}{a} = a\cdot\id\,.
\]
}

This relation between spherical stresses and volumetric stretches for isotropic elastic materials is highly intuitive: if the initial load is equal in every direction the resulting deformation should be equal in all directions as well and vice versa. However, whether this feature is true for all magnitudes of applied loads depends on the chosen constitutive framework. It is known that Axiom 2 is not satisfied for a number of well-known isotropic nonlinear elastic formulations, such as the Mooney-Rivlin energy and the Ogden type energy \cite{destrade2012}. The loss of uniqueness of the symmetric solution is encountered in Rivlin's cube problem \cite{rivlin1948}.

While the first two axioms refer only to stresses and stretches of a specific diagonal form, our third and final axiom states a \emph{law of superposition} that holds for all \emph{coaxial} stress-stretch pairs. Recall that we call symmetric matrices $A,B\in\Symn$ \emph{coaxial} if their principal axes coincide, which is the case if and only if $A$ and $B$ are simultaneously diagonalizable as well as if and only if $A$ and $B$ commute.

\axiombox{Axiom 3: Law of superposition}{
Let $\stretch_1,\stretch_2\in\PSymn$ be coaxial. Then
\[
	\stress(\stretch_1\cdot \stretch_2) = \stress(\stretch_1) + \stress(\stretch_2)\,.
\]
}

Note that, for an invertible stress-stretch relation, the third axiom could equivalently be stated as
\begin{equation}\label{eq:lawOfSuperpositionInverted}
	\stretch(\stress_1+\stress_2) = \stretch(\stress_1) \cdot \stretch(\stress_2)
\end{equation}
for all coaxial $\stress_1,\stress_2\in\Symn$.

This law of superposition can be summarized as follows: the (multiplicative) concatenation of stretch tensors should effect the (additive) superposition of the corresponding stress tensors. This nonlinear connection is closely related to a modern approach \cite{Neff_Eidel_Osterbrink_2013} involving the theory of Lie groups: the deformation tensors correspond to the (multiplicative) group $\GLp(3)$ while the stress tensors can be represented by the (linear) Lie algebra $\gl(3)$. By focusing on symmetric stresses $\stress$ as well as on positive definite symmetric stretches $\stretch$ we can also relate the stretch to the subset $\PSym(3)\subset\GLp(3)$ and the stress to the subspace $\Sym(3)\subset\gl(3)$. Since the canonical homomorphism mapping the additive structure of $\Sym(3)$ to the multiplicative group structure of $\PSym(3)$ in the way described by equation \eqref{eq:lawOfSuperpositionInverted} is the \emph{exponential function}
\[
	\exp: \Symn\to\PSymn\,,
\]
which is invertible with its inverse given by the \emph{principal logarithm}
\[
	\log: \PSymn\to\Symn\,,
\]
it is to be expected that the resulting stress-stretch relation $\stretch\mapsto\stress(\stretch)$ is, in turn, logarithmic in nature.

This approach is also closely related to the later deduction of the quadratic Hencky strain energy by Heinrich Hencky, who employed a similar law of superposition to obtain a logarithmic stress response function. However, Hencky considered the superposition of stresses in the \emph{deformed configuration}: the Cauchy stress $\sigma$ in his 1928 article \cite{hencky1928} as well as the Kirchhoff stress $\tau$ in his 1929 article \cite{hencky1929} (an English translation of both papers can be found in \cite{henckyTranslation}). Becker on the other hand assumes a law of superposition for \emph{initial loads}: his law refers to the Biot stress tensor $\Biot$. A more detailed comparison of Becker's and Hencky's work can be found in Appendix \ref{section:historicalContext}.

\subsection{Deduction of the general stress-stretch relation from the axioms}
We will now show that the general stress-stretch relation is determined by the given axioms up to only two constitutive parameters. Since our law of elasticity is isotropic by assumption we will mostly consider deformations given in the diagonal form
\[
	F = \dmatr{\lambda_1}{\lambda_2}{\lambda_3}\,.
\]
Recall from section \ref{subsection:basicAxioms} that the stress-stretch relation is uniquely determined by the stress response to such deformations.

\subsubsection{Basic properties}
Before we explicitly compute the stress-stretch relation from the axioms, we will first deduce some basic properties. The following properties of symmetry and invertibility follow directly from the law of superposition and the uniqueness of the stress-free reference state.
\begin{lemma}\label{lemma:tensionCompression}
Let $\stretch\in\PSymn$. Then $\stress(\stretch\inv) = -\stress(\stretch)$.
\end{lemma}
\begin{proof}
Since $\stretch$ and $\stretch\inv$ obviously commute, the law of superposition implies
\[
	\stress(\stretch) + \stress(\stretch\inv) = \stress(\stretch\cdot \stretch\inv) = \stress(\id) = 0\,,
\]
where the last equality is due to Axiom 0.2.
\end{proof}
\begin{remark}
The symmetry property given in Lemma \ref{lemma:tensionCompression} is generally not equivalent to the symmetric \emph{tension-compression symmetry} in hyperelasticity, which is defined by the equality
\[
	W(F\inv) = W(F)
\]
for the energy function $W$ and all $F\in\GLpn$. A hyperelastic stress-stretch relation is tension-compression symmetric if and only if
\[
	\tau(V\inv) = -\tau(V)
\]
for all $V\in\PSymn$, where $\tau$ denotes the \emph{Kirchhoff stress tensor}.\\
\end{remark}
\begin{lemma}\label{lemma:injectivity}
The mapping $\stretch\mapsto\stress(\stretch)$ is injective, i.e. $T: \PSymn\to\range(\stress),\, \stretch\mapsto \stress(\stretch)$ is invertible.
\end{lemma}
\begin{proof}
Let $\stretch_1, \stretch_2 \in\PSymn$ with $\stress(\stretch_1) = \stress(\stretch_2)$. Then using Lemma \ref{lemma:tensionCompression} we find
\[
	0 = \stress(\stretch_1) - \stress(\stretch_2) = \stress(\stretch_1) + \stress(\stretch_2\inv) = \stress(\stretch_1\cdot \stretch_2\inv)\,,
\]
thus Axiom 0.2 yields $\:\stretch_1\cdot \stretch_2\inv = \id$ and therefore $\stretch_1 = \stretch_2$.
\end{proof}
\begin{remark}
We will denote the inverse of the stress response by writing $\stretch(\hat\stress)$ for $\hat\stress\in\range(\stress)$ to denote the unique $\stretch\in\PSymn$ with $\hat\stress = \stress(\stretch)$.
\end{remark}
Combined with the continuity of the stress-stretch relation, the law of superposition allows us to compute the stress response to arbitrary powers of stretches. For a further discussion of non-rational powers of matrices as well as primary matrix functions in general we refer to \cite{higham2008}.
\begin{lemma}\label{lemma:powersOfStretches}
Let $\stretch\in\PSymn$. Then
\[
	\stress(\stretch^r) = r\cdot \stress(\stretch)
\]
for all $r\in\R$.
\end{lemma}
\begin{proof}
Since
\[
	\stress((Q\cdot \stretch\cdot Q^T)^r) = \stress(Q\cdot \stretch^r\cdot Q^T) = Q\cdot \stress(\stretch^r)\cdot Q^T
\]
we will assume without loss of generality that $\stretch$ is in diagonal form, i.e.  $\stretch=\diag(\lambda_1, \lambda_2, \lambda_3)$ with $\lambda_1,\lambda_2,\lambda_3\in\R^+$. Then $\stretch^r = \diag(\lambda_1^r, \lambda_2^r, \lambda_3^r)$ for each $r\in\R$.

For $n\in\N$ we can use the law of superposition to compute
\begin{align*}
	\stress(\stretch_r) &= \stress\Big( \dmatr{\lambda_1^n}{\lambda_2^n}{\lambda_3^n} \Big) \;=\; \stress\Big( \prod_{k=1}^n \dmatr{\lambda_1}{\lambda_2}{\lambda_3} \Big)\\[2mm]
	&= \sum_{k=1}^n \: \stress\Big(\dmatr{\lambda_1}{\lambda_2}{\lambda_3} \Big) \;=\; n \cdot \stress\Big(\dmatr{\lambda_1}{\lambda_2}{\lambda_3} \Big)\,.
\end{align*}
Furthermore we find
\begin{align*}
	0 = \stress(\id) = \stress(\stretch^n \cdot \stretch^{-n}) = n\cdot \stress(\stretch) + \stress(\stretch^{-n}) \quad\Longrightarrow\quad \stress(\stretch^{-n}) = -n\cdot \stress(\stretch)\,,
\end{align*}
thus $\stress(\stretch^z) = z\cdot \stress(\stretch)$ for all $z\in \Z$\,. Similarly we compute
\[
	z\cdot \stress(\stretch) = \stress(\stretch^z) = T\big((\stretch^{\frac{z}{n}})^n\big) = n\cdot \stress(\stretch^{\frac{z}{n}}) \quad\Longrightarrow\quad \stress(\stretch^{\frac{z}{n}}) = \frac{z}{n}
\]
for $z\in\Z$ and $n\in\N$, therefore $\stress(\stretch^q) = q\cdot \stress(\stretch)$ for all $q\in\Q$. Finally, we simply point to the continuity of $\stress$, which follows from Axiom 0.1, to conclude $\stress(\stretch^r) = r\cdot \stress(\stretch)$ for all $r\in\R$.
\end{proof}
\begin{remark}
Using our notation $\stress\mapsto \stretch(\stress)$ for the inverse stress-stretch relation this proposition may equivalently be stated as \[\stretch(r\cdot \stress) = \stretch(\stress)^r \qquad \forall\; r\in\R\,, \;\; \stress\in\range(\stress)\,.\]
\end{remark}
It is obvious that in a one-dimensional setting, Lemma \ref{lemma:powersOfStretches} would already characterize the stresss response as the logarithm to a fixed base. However, this is not immediately clear in the general case since not every stretch $\stretch\in\PSymn$ can be written as the real power of a single fixed matrix.\\
Note also that the assumption of continuity is in fact necessary for the proof of Lemma \ref{lemma:powersOfStretches}.

\subsubsection{Spherical stresses}
While Axiom 2 relates dilations to purely spherical stresses, no assumption about the \emph{amount} of stress is made. By using Lemma \ref{lemma:powersOfStretches}, however, it is easy to give an explicit formula for $\stress(\stretch)$ for arbitrary pure dilations $\stretch$.
\begin{lemma}\label{lemma:generalDilation}
There exists $d\in\R$ such that
\[
	\stress(\lambda\cdot\id) = d\cdot\log(\lambda\cdot\id) = d\cdot \ln(\lambda) \cdot \id
\]
for all $\lambda\in\R^+$.
\end{lemma}
\begin{proof}
Choose $\lambda_0 \in\R^+$ with $\lambda_0 \neq 1$. Then, according to Axiom 2, the stress response to $\lambda_0\cdot\id$ is given by
\[
	\stress(\lambda_0\cdot\id) = a_0\cdot\id
\]
for some $a_0\in\R$, and we define $d = \frac{a_0}{\ln\lambda_0}$.
Now let $\lambda\in\R^+$ with $\lambda\neq1$. Then $\lambda=\lambda_0^{\frac{\ln\lambda}{\ln\lambda_0}}$ and
\begin{align*}
	\stress(\lambda\cdot\id) &= \stress(\dmatr{\lambda}{\lambda}{\lambda}) = \stress(\dmatr{\lambda_0^{\frac{\ln\lambda}{\ln\lambda_0}}}{\lambda_0^{\frac{\ln\lambda}{\ln\lambda_0}}}{\lambda_0^{\frac{\ln\lambda}{\ln\lambda_0}}}) = {\frac{\ln\lambda}{\ln\lambda_0}}\cdot \stress(\dmatr{\lambda_0}{\lambda_0}{\lambda_0})\\[2mm]
	&= {\frac{\ln\lambda}{\ln\lambda_0}} \cdot \dmatr{a_0}{a_0}{a_0} = {\frac{a_0}{\ln\lambda_0}} \cdot \ln(\lambda) \cdot \dmatr{1}{1}{1} = d\cdot\ln(\lambda)\cdot\id \;=\; d\cdot\log(\lambda\cdot\id)\,.
\end{align*}
Finally, if $\lambda=1$, then Axiom 0.2 implies $\stress(\lambda\cdot\id) = \lambda(\id) = 0 = c\cdot\log(\id)$.
\end{proof}
\begin{remark}\label{remark:generalDilationInverse}
Note that this proposition can equivalently be stated as
\[
	\stretch(a\cdot\id) = \exp\Big(\frac1d \cdot a\cdot\id\Big) = e^{\frac{a}{d}} \cdot \id \qquad \forall\, a\in\R\,.
\]
\end{remark}

\subsubsection{Shear stresses} Let us now consider a pure shear stretch of the form
\[
	\stretch = \dmatr{\alpha}{\frac{1}{\alpha}}{1}
\]
with the ratio of shear $\alpha\in\R^+$. Again, while Axiom 1 only provides a general relation between shear stretches and shear stresses, the law of superposition yields a quantitative result for the case of pure shears.

\begin{lemma}\label{lemma:generalShear}
There exists $c\in\R$ such that
\begin{equation}
	\stress(\dmatr{\alpha}{\frac{1}{\alpha}}{1}) = c\cdot \log\dmatr{\alpha}{\frac{1}{\alpha}}{1} = c\cdot \dmatr{\ln(\alpha)}{-\ln(\alpha)}{0} \label{eq:shearFormulaPlaneXY}
\end{equation}
for all $\alpha\in\R^+$.
\end{lemma}
\begin{proof}
The proof is analogous to that of Lemma \ref{lemma:generalDilation}: choose $\alpha_0\in\R^+$ with $\alpha_0\neq1$. Then there exists $s_0$ such that
\[
	\stress\Big( \dmatr{\alpha_0}{\frac{1}{\alpha_0}}{1} \Big) = \dmatr{s_0}{-s_0}{0}\,,
\]
according to Axiom 1, and we define $c=\frac{s_0}{\ln\alpha_0}$.\\
Now let $\alpha\in\R^+$. Again we can use Axiom 0.2 and the equality $\log(\id)=0$ to show that the equality obviously holds for $\alpha=1$, hence we can assume $\alpha\neq1$ without loss of generality. Then $\alpha=\alpha_0^{\frac{\ln\alpha}{\ln\alpha_0}}$ and
\begin{align*}
	\stress(\dmatr{\alpha}{\frac{1}{\alpha}}{1}) &= \stress(\dmatr{\alpha_0^{\frac{\ln\alpha}{\ln\alpha_0}}}{(\frac1{\alpha_0})^{\frac{\ln\alpha}{\ln\alpha_0}}}{1^{\frac{\ln\alpha}{\ln\alpha_0}}})\\[2mm]
	&= \frac{\ln\alpha}{\ln\alpha_0}\cdot \stress(\dmatr{\alpha_0}{\frac{1}{\alpha_0}}{1})= {\frac{\ln\alpha}{\ln\alpha_0}} \cdot \dmatr{s_0}{-s_0}{0}\\[2mm]
	&= {\frac{s_0}{\ln\lambda_0}} \cdot  \dmatr{\ln\alpha}{-\ln\alpha}{0} = c\cdot\log(\dmatr{\alpha}{\frac{1}{\alpha}}{1})\,.\qedhere
\end{align*}
\end{proof}
\begin{remark}\label{remark:generalShearInverse}
Again, this proposition can also be stated as
\[
	\stretch(\dmatr{s}{-s}{0}) = \exp\,\Big(\frac1c \cdot\! \dmatr{s}{-s}{0}\Big) = \dmatr{e^{\frac1c \cdot s}}{e^{-\frac1c \cdot s}}{1}\qquad \forall\, s\in\R\,.
\]
\end{remark}
Although Lemma \ref{lemma:generalShear} refers only to deformations without strain along the $x_3$-axis, the following corollary shows that a similar property holds for shear deformations along the other principal axes as well.
\begin{corollary}\label{cor:generalShearGeneralAxes}
Let $\alpha\in\R^+$. Then
\begin{align}
	&\stress(\dmatr{\alpha}{1}{\frac{1}{\alpha}}) = c\cdot \log\dmatr{\alpha}{1}{\frac{1}{\alpha}} \label{eq:shearFormulaPlaneXZ}
\intertext{as well as}
	&\stress(\dmatr{1}{\alpha}{\frac{1}{\alpha}}) = c\cdot \log\dmatr{1}{\alpha}{\frac{1}{\alpha}} \label{eq:shearFormulaPlaneYZ}
\end{align}
with $c\in\R$ as given in Lemma \ref{lemma:generalShear}.
\end{corollary}
\begin{proof}
Let $Q=\matrs{1&0&0\\0&0&1\\0&1&0}\in\On$. Then
\begin{align*}
	\stress\,\Big(\dmatr{\alpha}{1}{\frac1\alpha}\Big)
	&= \stress\,\Big(Q^T\cdot\dmatr{\alpha}{\frac1\alpha}{1}\cdot Q\Big) = Q^T\cdot \stress\,\Big(\dmatr{\alpha}{\frac1\alpha}{1}\Big)\cdot Q\\[2mm]
	&= Q^T \cdot c\cdot \log\dmatr{\alpha}{\frac{1}{\alpha}}{1} \cdot Q = c \cdot \log\dmatr{\frac{1}{\alpha}}{\alpha}{1}\,,
\end{align*}
proving \eqref{eq:shearFormulaPlaneXZ}. To show \eqref{eq:shearFormulaPlaneYZ} we let $Q=\matrs{0&1&0\\0&0&1\\1&0&0}\in\SOn$ and find
\begin{align*}
	\stress\,\Big(\dmatr{1}{\alpha}{\frac1\alpha}\Big)
	&= \stress\,\Big(Q^T\cdot\dmatr{\alpha}{\frac1\alpha}{1}\cdot Q\Big) = Q^T\cdot \stress\,\Big(\dmatr{\alpha}{\frac1\alpha}{1}\Big)\cdot Q\\[2mm]
	&= Q^T \cdot c\cdot \log\dmatr{\alpha}{\frac{1}{\alpha}}{1} \cdot Q = c \cdot \log\dmatr{\frac{1}{\alpha}}{\alpha}{1}\,. \qedhere
\end{align*}
\end{proof}

\subsubsection{The general case}

Finally we consider the general case of an arbitrary stretch tensor $\stretch$.
\begin{proposition}\label{prop:deductionFromAllAxioms}
A stress response function $\stretch\mapsto \Biot(\stretch)$ fulfils Axioms 0.1--0.3 and Axioms 1--3 if and only if there exist constants $G,\Lambda\in\R$, $G\neq 0$, $3\,\Lambda+2\,G\neq 0$ such that
\begin{equation}\label{eq:generalStressStretchRelationFromAllAxioms}
	\stress(\stretch) = 2\,G\cdot\log(\stretch) + \Lambda\cdot\tr[\log \stretch]\cdot\id
\end{equation}
or, equivalently, constants $G,\bulk\in\R\setminus\{0\}$ with
\begin{equation}\label{eq:generalStressStretchRelationDeviatorVersionFromAllAxioms}
	\stress(\stretch) = 2\,G\cdot\dev_3\log(\stretch) + \bulk\cdot\tr[\log \stretch]\cdot\id
\end{equation}
for all $\stretch\in\PSymn$, where $\log:\PSymn\to\Symn$ is the principal matrix logarithm and $\dev_3 X = X-\frac13\tr(X)\cdot \id$ denotes the deviatoric part of $X\in\R^{3\times3}$.
\end{proposition}
\begin{remark}
A justification for the use of the Lam\'e constants $G$, $\Lambda$ and the bulk modulus $\bulk$ in this formulae will be given by means of linearization in section \ref{subsection:infinitesimalDeformations}.
\end{remark}
\begin{proof}
First we consider a stretch tensor $\stretch$ in the diagonal form
\[
	\stretch = \dmatr{p}{q}{r}\,.
\]
Then $\stretch$ can be decomposed multiplicatively into three stretches $\stretch_1$, $\stretch_2$ and $\stretch_3$:
\begin{align*}
	\stretch \;&=\; \dmatr{(p\,q\,r)^{\frac13}}{(p\,q\,r)^{\frac13}}{(p\,qr\,)^{\frac13}} \;\cdot\; \dmatr{\big(\frac{p^2}{q\,r}\big)^{\frac13}}{\big(\frac{q^2}{p\,r}\big)^{\frac13}}{\big(\frac{r^2}{p\,q}\big)^{\frac13}}\\[2mm]
	&=\; \underbrace{\dmatr{p\,q\,r}{p\,qr\,}{p\,q\,r}^{\afrac13}}_{\stretch_1} \;\cdot\; \underbrace{\dmatr{\frac{p^2}{q\,r}}{\frac{q\,r}{p^2}}{1}^{\afrac13}}_{\stretch_2} \;\cdot\; \underbrace{\dmatr{1}{\frac{p\,q}{r^2}}{\frac{r^2}{p\,q}}^{\afrac13}}_{\stretch_3}\,.
\end{align*}
Using the law of superposition we find $\stress(\stretch) = \stress(\stretch_1) + \stress(\stretch_2) + \stress(\stretch_3)$. Lemma \ref{lemma:generalDilation} allows us to compute
\begin{align*}
	\stress(\stretch_1) &= \stress((p\,q\,r)^{\frac13}\cdot\id) = d\cdot\ln((p\,q\,r)^{\frac13})\cdot\id \\&= \frac{d}{3}\cdot\ln(p\,q\,r\,)\cdot\id = \frac{d}{3}\cdot\ln(\det(\stretch))\cdot\id = \frac{d}{3}\cdot\tr[\log(\stretch)]\cdot\id
\end{align*}
with a constant $d\in\R$, while Lemma \ref{lemma:generalShear} and Corollary \ref{cor:generalShearGeneralAxes} simply yield
\[
	\stress(\stretch_2) = c\cdot\log(\stretch_2)\,, \qquad \stress(\stretch_3) = c\cdot\log(\stretch_3)
\]
with $c\in\R$. Therefore
\begin{align*}
	\stress(\stretch_2) + \stress(\stretch_3) &= c\cdot\log(\stretch_2) + c\cdot\log(\stretch_3)\\[2mm]
	&= c \cdot \dmatr{\ln\big[(\frac{p^2}{q\,r})^{\frac13}\big]}  {\ln\big[(\frac{q\,r}{p^2})^{\frac13}\big]}  {0}
	+ c\cdot \dmatr{0}  {\ln\big[(\frac{p\,q}{r^2})^{\frac13}\big]}  {\ln\big[(\frac{r^2}{p\,q})^{\frac13}\big]}\\[2mm]
	&= c \cdot \dmatr{\ln\big[(\frac{p^2}{q\,r})^{\frac13}\big]}  {\ln\big[(\frac{q\,r}{p^2})^{\frac13}\big] + \ln\big[(\frac{p\,q}{r^2})^{\frac13}\big]}  {\ln\big[(\frac{r^2}{p\,q})^{\frac13}\big]}\\[2mm]
	&= c \cdot \dmatr{\ln\big[(\frac{p^2}{q\,r})^{\frac13}\big]}  {\ln\big[(\frac{q^2}{p\,r})^{\frac13}\big]}  {\ln\big[(\frac{r^2}{p\,q})^{\frac13}\big]}\\[2mm]
	&= c \cdot \dmatr{\ln\big[(\frac{1}{p\,q\,r})^{\frac13}\cdot p\big]}  {\ln\big[(\frac{1}{p\,q\,r})^{\frac13}\cdot q\big]}  {\ln\big[(\frac{1}{p\,q\,r})^{\frac13}\cdot r\big]}\\[2mm]
	&= c \cdot \dmatr{\ln(p)}{\ln(q)}{\ln(r)} \;+\; c \cdot \dmatr{\ln\big[(\frac{1}{p\,q\,r})^{\frac13}\big]}  {\ln\big[(\frac{1}{p\,q\,r})^{\frac13}\big]}  {\ln\big[(\frac{1}{p\,q\,r})^{\frac13}\big]}\,,%
\end{align*}
hence
\begin{align*}
	\stress(\stretch_2) + \stress(\stretch_3) &= c\cdot \log\dmatr{p}{q}{r} \;+\; \frac{c}{3}\cdot \dmatr{-\ln(p\,q\,r)}{-\ln(p\,q\,r)}{-\ln(p\,q\,r)}\\[2mm]
	&= c\cdot \log(\stretch) \;-\; \frac{c}{3}\cdot \ln(\det \stretch) \cdot \id\\
	&= c\cdot (\log(\stretch) \;-\; \frac13\cdot \tr[\log \stretch] \cdot \id) \;=\; c\cdot\dev_3\log \stretch\,.
\end{align*}
Thus $\stress(\stretch)$ computes to
\begin{align*}
	\stress(\stretch) &= \stress(\stretch_2) + \stress(\stretch_3) + \stress(\stretch_1) = c\cdot\dev_3\log(\stretch) + \frac{d}{3}\cdot\tr[\log(\stretch)]\cdot\id\,.
\end{align*}
Finally, for arbitrary $\stretch\in\PSymn$ we can choose $Q\in\On$ and a diagonal matrix $D$ such that $\stretch=Q^TDQ$. Utilizing the isotropy property as well as the above computations we find
\begin{align*}
	\stress(\stretch) &= \stress(Q^TDQ) = Q^T\,\Biot(D)\,Q\\
	&= Q^T \cdot\Big(c\cdot\dev_3\log(D) \;+\; \frac{d}{3}\cdot \tr[\log(D)]\cdot \id\Big)\cdot Q\\
	&= c\cdot\dev_3\log(Q^TDQ) \;+\; \frac{d}{3}\cdot \tr[\log(\stretch)]\cdot (Q^TQ) \;=\; c\cdot\dev_3\log(\stretch) \;+\; \frac{d}{3}\cdot \tr[\log(\stretch)]\cdot \id\,,
\end{align*}
thus we obtain equation \eqref{eq:generalStressStretchRelationFromAllAxioms} with $G=\frac{c}{2}$ and $\bulk=\frac{d}{3}$. It is also easy to see that the restrictions $G\neq0$ and $\bulk\neq0$ follow directly from the injectivity of the response function. Furthermore, with $\Lambda = \bulk-\frac{2\,G}{3}$ we obtain the equivalent representation
\begin{align*}
	\stress(\stretch) &= 2\,G\cdot[\log \stretch - \frac13\tr[\log \stretch]\cdot \id] + \bulk\cdot\id\\
	&= 2\,G\cdot\log(\stretch) + \Big(\bulk-\frac{2\,G}{3}\Big)\tr[\log \stretch]\cdot\id \;=\; 2\,G\cdot\log(\stretch) + \Lambda\cdot\tr[\log \stretch]\cdot\id\,.
\end{align*}
It remains to show that the stress response function \eqref{eq:generalStressStretchRelationFromAllAxioms} does indeed satisfy all our axioms. Since the matrix logarithm and the trace operator are continuous functions\footnote{The matrix logarithm is, in fact, an analytic function on $\PSymn$.} on $\PSymn$, Axiom 0.1 obviously holds. The isotropy of the matrix logarithm immediately implies
\begin{align*}
	2\,G\cdot \log((Q^T\stretch Q)) \;+\; \Lambda\cdot\tr[\log (Q^T\stretch Q)] \cdot \id &= 2\,G\cdot Q^T\log(\stretch)Q \;+\; \Lambda\cdot\tr[Q^T (\log \stretch)Q] \cdot \id\\
	&= 2\,G\cdot Q^T\log(\stretch)Q \;+\; \Lambda\cdot\tr[\log \stretch] \cdot \id\\
	&= Q^T\cdot \big(2\,G\cdot \log(\stretch) \;+\; \Lambda\cdot\tr[\log \stretch] \cdot \id\big) \cdot Q\,,
\end{align*}
thus Axiom 0.3 holds as well. To show Axiom 0.2 we employ the equivalent representation formula \eqref{eq:generalStressStretchRelationDeviatorVersionFromAllAxioms}: 
for $G,\bulk \neq 0$ we first note that the mapping
\[
	X \;\mapsto\; 2\,G\cdot\dev_3 X + \bulk\cdot\tr[X]\cdot\id
\]
is an isomorphism from $\Symn$ onto itself. Thus
\[
	2\,G\cdot\dev_3\log(\stretch) + \bulk\cdot\tr[\log \stretch]\cdot\id \;=\; 0
\]
if and only if $\log \stretch = 0$, which is the case if and only if $\stretch=\id$.

We will now consider the remaining three axioms.\\
Axiom 1:\; For $\stretch=\diag(\alpha,\frac1\alpha,1)$ we directly compute
\begin{align*}
	2\,G\cdot\dev_3\log(\stretch) + \bulk\cdot\tr[\log \stretch]\cdot\id%
	&= 2\,G\cdot\dev_3\dmatr{\ln \alpha}{-\ln \alpha}{0} \;+\; \bulk\cdot\left[ \ln\alpha + \ln \frac1\alpha \right]\cdot\id\\[2mm]
	&= 2\,G\cdot\dmatr{\ln \alpha}{-\ln \alpha}{0}\,,
\end{align*}
thus Axiom 1 is fulfilled with $s = 2\,G\,\ln\alpha$.\medskip\\
Axiom 2:\; For $\stretch=\lambda\cdot\id$ we compute
\begin{align*}
	2\,G\cdot\dev_3\log(\stretch) + \bulk\cdot\tr[\log \stretch]\cdot\id%
	&= 2\,G\cdot\dev_3\dmatr{\ln \lambda}{\ln \lambda}{\ln \lambda} \;+\; \bulk\cdot\tr\left[ 3\cdot \ln \lambda \right]\cdot\id\\[2mm]
	&= \bulk\cdot [3\,\ln \lambda] \cdot \id\,,
\end{align*}
thus Axiom 2 is fulfilled with $a = 3\,\bulk\,\ln\alpha$.\medskip\\
Axiom 3:\, First assume that $\stretch_1, \stretch_2 \in\PSymn$ have the diagonal form
\[
	\stretch_1 = \dmatr{\lambda_1}{\lambda_2}{\lambda_3}\,, \qquad \stretch_2 = \dmatr{\lambdahat_1}{\lambdahat_2}{\lambdahat_3}\,.
\]
Then
\begin{align*}
	\log(\stretch_1\cdot \stretch_2) &= \log \dmatr{\lambda_1\,\lambdahat_1}{\lambda_2\,\lambdahat_2}{\lambda_3\,\lambdahat_3} = \dmatr{\ln (\lambda_1\,\lambdahat_1)}{\ln (\lambda_2\,\lambdahat_2)}{\ln (\lambda_3\,\lambdahat_3)}\\[2mm]
	&= \dmatr{\ln (\lambda_1) + \ln(\lambdahat_1)}{\ln (\lambda_2) + \ln(\lambdahat_2)}{\ln (\lambda_3) + \ln(\lambdahat_3)} \;=\; \log(\stretch_1) + \log(\stretch_2)
\end{align*}
and therefore
\begin{align*}
	\stress(\stretch_1\cdot \stretch_2) &= 2\,G\cdot\log(\stretch_1\cdot \stretch_2) + \Lambda\cdot\tr[\log \stretch_1\cdot \stretch_2]\cdot\id\\
	&= 2\,G\cdot[\log(\stretch_1) + \log(\stretch_2)] + \Lambda\cdot\tr[\log (\stretch_1) + \log(\stretch_2)]\cdot\id\\
	&= 2\,G\cdot\log(\stretch_1) + 2\,G\cdot\log(\stretch_2) + \Lambda\cdot\tr[\log (\stretch_1)]\cdot\id + \Lambda\cdot\tr[\log(\stretch_2)]\cdot\id\\
	&= \stress(\stretch_1) + \stress(\stretch_2)\,.
\end{align*}
Now let $\stretch_1$ and $\stretch_2$ denote arbitrary coaxial matrices. Then $\stretch_1$ and $\stretch_2$ can be simultaneously diagonalized, i.e. there exist diagonal matrices $D_1$ and $D_2$ as well as $Q\in\On$ with
\[
	\stretch_1 = Q^T D_1 Q\,, \qquad \stretch_2 = Q^T D_2 Q\,.
\]
Then
\begin{align*}
	\stress(\stretch_1\cdot \stretch_2) &= \stress(Q^T D_1 D_2 Q) = Q^T\cdot\stress(D_1 D_2)\cdot Q\\
	&= Q^T\cdot [\stress(D_1) + \stress(D_2)] \cdot Q = Q^T\,\stress(D_1)\,Q + Q^T\,\stress(D_2)\,Q\\
	&= \stress(Q^T\,D_1\,Q) + \stress(Q^T\,D_2\,Q) = \stress(\stretch_1) + \stress(\stretch_2)\,.\qedhere
\end{align*}
\end{proof}

While Becker assumed Axioms 1 and 2 to hold, we will now show that they are, in fact, not necessary to characterize Becker's law of elasticity but can be deduced from Axioms 0.1--0.3 and Axiom 3 alone.

\begin{lemma}
If Axioms 0.1, 0.2, 0.3 and 3 hold, then Axioms 1 and 2 hold as well.
\end{lemma}
\begin{proof}
First, for $\alpha = 1$ and $\lambda = 1$, we find
\[
	\dmatr{\alpha}{\frac1\alpha}{1} \;=\; \dmatr{\lambda}{\lambda}{\lambda} \;=\; \dmatr{1}{1}{1} \;=\; \id
\]
and thus
\[
	\stress\Big( \dmatr{\alpha}{\frac1\alpha}{1} \Big) \;=\; \stress\Big( \dmatr{\lambda}{\lambda}{\lambda} \Big) \;=\; \stress(\id) \;=\; 0
\]
due to axiom 0.2. We can therefore assume without loss of generality that $\lambda \neq 1 \neq \alpha$.\\
Axiom 1: Let $\alpha\in\R^+$ with $\alpha \neq 1$ and $\stretch=\diag(\alpha,\frac1\alpha,1)$. Then, because the principal axes of $\stress(\stretch)$ and $\stretch$ coincide, $\stress(\stretch)$ is in diagonal form as well:
\[
	\stress(\stretch) = \stress\,\Big(\dmatr{\alpha}{\frac1\alpha}{1}\Big) = \dmatr{a}{b}{c}
\]
for some $a,b,c\in\R$. With $Q=\matrs{0&1&0\\1&0&0\\0&0&1}\in\On$, the property of isotropy allows us to compute
\[
	\stress\Big(\dmatr{\frac1\alpha}{\alpha}{1}\Big) = \stress(Q^T\cdot \stretch\cdot Q) = Q^T\cdot\stress(\stretch)\cdot Q = \dmatr{b}{a}{c}
\]
and using the law of superposition we find
\begin{align*}
	\stress(\id) &=\; \stress\Big(\dmatr{\alpha}{\frac1\alpha}{1} \cdot \dmatr{\frac1\alpha}{\alpha}{1}\Big) \;=\; \stress\Big(\dmatr{\alpha}{\frac1\alpha}{1}\Big) \;+\; \stress\Big(\dmatr{\frac1\alpha}{\alpha}{1}\Big)\\[2mm]
	&=\; \dmatr{a}{b}{c} \;+\; \dmatr{b}{a}{c} \;=\; \dmatr{a+b}{a+b}{2c}\,.
\end{align*}
Since $\stress(\id)=0$ we conclude $b=-a$ as well as $c=0$, thus $\stress(\stretch)$ has the form
\begin{equation}
\label{eq:deductionOfPureShearFormulaFromSuperposition}
	\stress\,\Big(\dmatr{\alpha}{\frac1\alpha}{1}\Big) \;=\; \dmatr{s}{-s}{0}
\end{equation}
with $s = a$. As was shown in the proof of Lemma \ref{lemma:injectivity}, a function satisfying Axioms 3 and 0.2 is injective, hence
\[
	\stress(\stretch) = \matr{s&0&0\\0&-s&0\\0&0&0} \qquad\Longleftrightarrow\qquad \stretch=\matr{\alpha&0&0\\0&\frac{1}{\alpha}&0\\0&0&1}\,.
\]
Axiom 2: Now let $\lambda\in\R^+$ with $\lambda \neq 1$ and $\stretch=\diag(\lambda,\lambda,\lambda) = \lambda\cdot\id$. Then
\[
	\stress(\stretch) = \stress\,\Big(\dmatr{\lambda}{\lambda}{\lambda}\Big) = \dmatr{a}{b}{c}
\]
for some $a,b,c\in\R$. We let $Q=\matrs{0&0&1\\1&0&0\\0&1&0}\in\SOn$ to find
\begin{align*}
	\stress(\stretch) = \stress(\lambda\cdot\id) &= \stress(Q^T\cdot(\lambda\cdot\id)\cdot Q)\\
	& = Q^T\cdot \stress(\lambda\cdot\id)\cdot Q = Q^T\cdot\dmatr{a}{b}{c}\cdot Q = \dmatr{b}{c}{a}\,.
\end{align*}
Therefore $a=b$ and $b=c$, hence $\stress(\stretch)$ has the form
\begin{equation}
\label{eq:deductionOfPureDilationFormulaFromSuperposition}
	\stress(\lambda\cdot\id) \;=\; \dmatr{a}{a}{a} \;=\; a\cdot\id
\end{equation}
with $a\in\R$. Then the injectivity of $\stress$ yields
\[
	\stress(\stretch) = a\cdot\id \qquad\Longleftrightarrow\qquad \stretch=\lambda\cdot\id\,,
\]
concluding the proof.
\end{proof}
From this Lemma and Proposition \ref{prop:deductionFromAllAxioms}, it immediately follows that the reduced set of axioms is sufficient to characterize the stress response function. This result is summarized in the following proposition.
\begin{proposition}\label{prop:deductionFromSuperposition}
Let $\stress:\PSymn\to\Symn$ be a continuous isotropic tensor function with
\[
	\stress(\stretch) \;=\; 0 \quad\Longleftrightarrow\quad \stretch \;=\;\id
\]
and
\[
	\stress(\stretch_1\cdot\stretch_2) \;=\; \stress(\stretch_2) \,+\, \stress(\stretch_2)
\]
for all $\stretch_1,\, \stretch_2 \in \PSymn$. Then there exist constants $G,\Lambda\in\R$, $G\neq 0$, $3\,\Lambda+2\,G\neq 0$ such that
\begin{equation}\label{eq:generalStressStretchRelation}
	\stress(\stretch) = 2\,G\cdot\log(\stretch) + 3\,\Lambda\cdot\tr[\log \stretch]\cdot\id
\end{equation}
or, equivalently, constants $G,\bulk\in\R\setminus\{0\}$ with
\begin{equation}\label{eq:generalStressStretchRelationDeviatorVersion}
	\stress(\stretch) = 2\,G\cdot\dev_3\log(\stretch) + \bulk\cdot\tr[\log \stretch]\cdot\id
\end{equation}
for all $\stretch\in\PSymn$.
\end{proposition}

If we assume beforehand that the stress response function is invertible, we can also deduce this general law in terms of the inverse stress-stretch relation: let $\stress$ denote a given stress tensor of the form
\[
	\stress = \dmatr{P}{Q}{R}\,.
\]
Then $\stress$ can be written in form of the additive decomposition
\begin{align*}
	\stress \;&=\; \underbrace{\frac{P+Q+R}{3}\:\cdot\:\dmatr{1}{1}{1}}_{\stress_1}
	\;+\; \underbrace{\frac{Q+R-2P}{3}\:\cdot\:\dmatr{-1}{1}{0}}_{\stress_2}  \;+\; \underbrace{\frac{P+Q-2R}{3}\:\cdot\:\dmatr{0}{1}{-1}}_{\stress_3}
\end{align*}
into two pure shear stresses and one spherical stress. Using Remarks \ref{remark:generalDilationInverse} and \ref{remark:generalShearInverse} we compute
\[
	\stretch(\stress_1) = e^{\frac{P+Q+R}{3\,d}} \cdot \id
\]
as well as
\[
	\stretch(\stress_2) = \dmatr{e^{-\frac{Q+R-2P}{3\,c}}} {e^{\frac{Q+R-2P}{3\,c}}} {1}
\]
and
\[
	\stretch(\stress_3) = \dmatr{1} {e^{\frac{P+Q-2R}{3\,c}}} {e^{-\frac{P+Q-2R}{3\,c}}}\,.
\]
Therefore the law of superposition yields
\begin{align*}
	\stretch(\stress) &= \stretch(\stress_1) \;\cdot\; \stretch(\stress_2) \;\cdot\; \stretch(\stress_3) \\[2mm]
	&= e^{\frac{P+Q+R}{3\,d}}  \;\cdot\;  \dmatr{e^{-\frac{Q+R-2P}{3\,c}}} {e^{\frac{Q+R-2P}{3\,c}}} {1}  \;\cdot\;  \dmatr{1} {e^{\frac{P+Q-2R}{3\,c}}} {e^{-\frac{P+Q-2R}{3\,c}}} \\[2mm]
	&= e^{\frac{\tr\stress}{3\,d}}  \;\cdot\;  \dmatr{e^{\frac{2P-Q-R}{3\,c}}} {e^{\frac{2Q-P-R}{3\,c}}} {e^{\frac{2R-P-Q}{3\,c}}} \\[2mm]
	&= e^{\frac{\tr\stress}{3\,d}}  \;\cdot\;  \exp \dmatr{\frac{2P-Q-R}{3\,c}} {\frac{2Q-P-R}{3\,c}} {\frac{2R-P-Q}{3\,c}} \;\;=\;\; e^{\frac{\tr\stress}{3\,d}} \;\cdot\; \exp\Big(\frac1c \cdot\dev_3\stress\Big)\,.
\end{align*}
With constants $G=\frac{c}{2}$ and $\bulk=\frac{d}{3}$ our law of ideal elasticity can therefore be stated as
\begin{equation}\label{eq:gerneralStressStretchRelationInverted}
	\stretch(\stress) \;=\; \exp\Big(\frac{1}{2\,G}\dev_3 T + \frac{1}{9\,\bulk} \tr(T)\cdot\id\Big) \;=\; e^{\frac{1}{9\,\bulk}\cdot\tr\stress} \;\cdot\; \exp\Big(\frac{1}{2\,G}\cdot\dev_3 \stress\Big)\,.
\end{equation}

\subsubsection{Becker's stress response function}
Finally, since Becker assumed Axioms 0.1--0.3 and 1--3 to hold for the Biot stress $\stress = \Biot$ and the right Biot stretch tensors $\stretch = U$, we conclude that  Becker's law of elasticity is given by
\begin{equation}\label{eq:generalStressStretchRelationBecker}
	\Biot(U) = 2\,G\cdot\log(U) + \Lambda\cdot\tr[\log U]\cdot\id
\end{equation}
or, equivalently, by
\begin{align}
		\Biot(U) &= 2\,G\cdot\dev_3\log(U) + \bulk\cdot\tr[\log U]\cdot\id\\
		&= \frac{E}{1+\nu}\cdot\dev_3\log(U) + \frac{E}{3(1-2\nu)}\cdot\tr[\log U]\cdot\id\label{eq:eq:generalStressStretchRelationPoissonVersionBecker}
\end{align}
with Young's modulus $E=\frac{9\,K\,G}{3K+G}$ and Poisson's number $\nu=\frac{3K-2G}{2\,(3K+G)}$.

\subsection{Application to other stresses and stretches}
If we apply Proposition \ref{prop:deductionFromSuperposition} to other coaxial stress-stretch pairs, the resulting law of elasticity will, in general, differ from that given by Becker. Two examples of such combinations are especially important: the left stretch tensor $\stretch = V = \sqrt{FF^T}$ with the Cauchy stress tensor $\stress = \sigma$ as well as the left stretch with the Kirchhoff stress tensor $\stress = \tau$. Those cases where considered by Heinrich Hencky in 1928 and 1929, respectively \cite{hencky1928, hencky1929}. His approach was remarkably similar to Becker's: from the assumption of a law of superposition for these two stresses he deduced two laws of idealized elasticity.
\begin{corollary}\label{cor:hencky}
If the Cauchy stress $\sigma$ is a continuous isotropic function of the left stretch tensor $V$ with
\[
	\sigma(V) \;=\; 0 \quad\Longleftrightarrow\quad V \;=\;\id
\]
and
\[
	\sigma(V_1\cdot V_2) \;=\; \sigma(V_2) \,+\, \sigma(V_2)
\]
for all $V_1,\, V_2 \in \PSymn$, then there exist constants
$G,\bulk\in\R\setminus\{0\}$
such that
\begin{equation}\label{eq:Hencky1928Law}
	\sigma(V) = 2\,\shear\cdot\dev_3\log(V) + \bulk\cdot\tr[\log V]\cdot\id
\end{equation}
for all $V\in\PSymn$.

If the Kirchhoff stress $\tau$ is a continuous isotropic function of the left Biot stretch tensor $V$ with
\[
	\tau(V) \;=\; 0 \quad\Longleftrightarrow\quad V \;=\;\id
\]
and
\[
	\tau(V_1\cdot V_2) \;=\; \tau(V_2) \,+\, \tau(V_2)
\]
for all $V_1,\, V_2 \in \PSymn$, then there exist constants
$G,\bulk\in\R\setminus\{0\}$
such that
\begin{equation}\label{eq:HenckyHyperelasticLaw}
	\tau(V) = 2\,\shear\cdot\dev_3\log(V) + \bulk\cdot\tr[\log V]\cdot\id
\end{equation}
for all $V\in\PSymn$.
\end{corollary}

It was shown by Hencky that only the latter of these two stress response functions constitutes a \emph{hyperelastic} law of elasticity: the stress-stretch relation in equation \eqref{eq:HenckyHyperelasticLaw} can be obtained from the \emph{quadratic Hencky strain energy}
\begin{equation}
	W(V) = \shear\,\norm{\dev_3 \log V}^2 + \frac{\bulk}{2}\,[\tr(\log V)]^2\,.
\end{equation}
This energy function has also been given another rigorous justification, based on purely differential geometric reasoning, as the geodesic distance of the deformation gradient $F$ to the special orthogonal group $\SOn$ with respect to the canonical left-invariant metric on $\GLpn$ \cite{Neff_Eidel_Osterbrink_2013, Neff_Osterbrink_hencky13,Neff_Nagatsukasa_logpolar13}. The continuing application of the Hencky strain energy or modifications thereof is described in \cite{NeffGhibaLankeit}, see also \cite{NeffGhibaSteigmann}.

\subsection{The axioms in the linear case of infinitesimal elasticity}
Some of Becker's assumptions seem to have been adapted from simple results for the linearised theory of elasticity. To explain his motivation it is insightful to discover some of these parallels. The most general stress-strain relation for isotropic homogeneous materials in the case of linear elasticity is
\[
	\sigma \;=\; 2\,\shear\,\eps + \Lambda\,\tr(\eps)\cdot\id \;=\; 2\,\shear\,\dev_3\eps + \bulk\,\tr(\eps)\cdot\id \;=\; \frac{E}{1+\nu}\,\dev_3\eps + \frac{E}{3(1-2\nu)}\,\tr(\eps)\cdot\id\,,
\]
where $\sigma$ is the linearized stress tensor and $\eps = \sym \grad u = \half(\grad u + \grad u^T)$ is the linearized strain tensor of the deformation $\varphi(x) = x + u(x)$ with the displacement $u:\Omega_0\subset\R^3\to\R^3$. Note that this linear relation is invertible with
\[
	\eps = \frac{1}{2\,G}\dev_3 \sigma + \frac{1}{9\,\bulk}\tr(\sigma)\cdot\id\,,
\]
similar to the first equality in equation \eqref{eq:gerneralStressStretchRelationInverted}.

Since the trace operator is the linear approximation of the determinant at $\id$, i.e. $\det(\id+H) = 1 + \tr(H) + \mathcal{O}(\norm{H}^2)$, the first order approximation
\[
	\det(\grad\varphi) \;=\; \det(\id+\grad u) \;\approx\; 1 + \tr(\grad u)
\]
holds for sufficiently small $\norm{\grad u}$. Therefore, the condition
\[
	\det(U) = \det(\grad\varphi) = 1
\]
can be linearized to the equation
\[
	\tr (\grad u) \;=\; \tr \eps \;=\; 0\,.
\]
For $\Lambda \neq -\frac23\,\shear$, this is the case if and only if
\[
	0 = \tr \sigma = (2\,\shear + 3\,\Lambda) \cdot \tr\eps\,,
\]
thus (linearized) isochoric deformations always correspond to trace free stress tensors $\sigma$ and vice versa, analogous to Axiom 1 for the nonlinear case. Similarly volumetric stresses occur if and only if the strain is (linearly) volumetric (Axiom 2):
\[
	\eps = \matr{a&0&0\\ 0&a&0\\ 0&0&a} \quad \Longleftrightarrow \quad \sigma = \matr{(2\shear+3\Lambda)a&0&0\\ 0&(2\shear+3\Lambda)a&0\\ 0&0&(2\shear+3\Lambda)a}\,.
\]
Furthermore it is easy to see that the linearized shear strain
\[
	\eps = \sym\left[ \matr{1&\gamma&0\\ 0&1&0\\ 0&0&1} - \id\right] = \matr{0&\frac\gamma2&0\\ \frac\gamma2&0&0\\ 0&0&0}
\]
corresponds to the shear stress (Axiom 1)
\[
	\sigma = \matr{0&\shear\cdot\gamma&0\\ \shear\cdot\gamma&0&0\\ 0&0&0}\,.
\]
Finally, we consider two deformation gradients $\grad\varphi_1 = \id+\grad u_1$ and $\grad\varphi_2 = \id+\grad u_2$ with the corresponding strain tensors $\eps_1$ and $\eps_2$. Then
\begin{align*}
	\grad \varphi_1 \cdot \grad\varphi_2 \;=\; (\id+\grad u_1)\cdot(\id+\grad u_2) \;&=\; \id + \grad u_1 + \grad u_2 + \grad u_2 \cdot \grad u_2\,.
\end{align*}
By omitting the higher order term $\grad u_1 \cdot \grad u_2$, we find the linear approximation
\[
	\grad\varphi_1\cdot\grad\varphi_2 \;\approx\; \id + \grad u_1 + \grad u_2\,,
\]
hence the strain tensor $\eps$ corresponding to $\grad\varphi_1\cdot\grad\varphi_2$ has the linear approximation
\[
	\eps \;\approx\; \sym(\grad u_1 + \grad u_2) \;=\; \eps_1 + \eps_2\,.
\]
Thus, in the linear case, the multiplicative superposition of deformation gradients corresponds to an additive composition of the strain tensors. The law of superposition (Axiom 3) can therefore be linearized to
\[
	\sigma(\eps_1 + \eps_2) = \sigma(\eps_1) + \sigma(\eps_2)\,,
\]
which obviously holds for all $\eps_1,\eps_2\in\Symn$.

The linear analogies of the three main axioms can therefore be summarized as follows.
\axiombox{Axiom 1, linear version:}{
The equivalence 
\[
	\eps = \matr{0&\frac\gamma2&0\\ \frac\gamma2&0&0\\ 0&0&0} \qquad\Longleftrightarrow\qquad \sigma(\eps) = \matr{0&\shear\cdot\gamma&0\\ \shear\cdot\gamma&0&0\\ 0&0&0}
\]
holds for all $\gamma\in\R$.
}
\axiombox{Axiom 2, linear version:}{
The equivalence
\[
	\eps = \matr{a&0&0\\ 0&a&0\\ 0&0&a} \quad \Longleftrightarrow \quad \sigma(\eps) = \matr{(2\shear+3\Lambda)a&0&0\\ 0&(2\shear+3\Lambda)a&0\\ 0&0&(2\shear+3\Lambda)a}
\]
holds for all $a\in\R$.
}
\axiombox{Axiom 3, linear version:}{
The equality
\[
	\sigma(\eps_1 + \eps_2) = \sigma(\eps_1) + \sigma(\eps_2)
\]
holds for all $\eps_1,\eps_2\in\Symn$.
}
Note that many of the properties listed in section \ref{section:reflectionsOnConstitutiveAssumptions} have linearized counterparts which are satisfied by the general linear model, e.g. the (linearized) tension-compression symmetry $\sigma(-\eps) = -\sigma(\eps)$.

\subsubsection{Linearised shear}
Becker's comments on the finite shear response show similiarities to the linear case as well. We consider the \emph{linearised shear stress}
\[
	\sigma = \matr{0&s\\ s&0}
\]
with the corresponding \emph{linear shear strain}\footnote{The linear shear strain appears in the linearisation of a simple shear shear deformation, which can be written as
\[\matr{1&\gamma\\0&1} = \id+\matr{0&\frac\gamma2\\\frac\gamma2&0} + \matr{0&\frac\gamma2\\-\frac\gamma2&0} = \id+\eps+W\,,\] where $\eps\in\Symn$ is a linear shear strain and $W\in\so(3)$ corresponds to an \emph{infinitesimal rotation}.}
\[
	\eps = \matr{0&\frac\gamma2\\ \frac\gamma2&0}
\]
in the two dimensional case. Then for given $n=(n_1,n_2)^T\in\R^2$ with $\norm{n}=1$ we can compute
\[
	\norm{\sigma\,n} = \norm{\matr{s\,n_2\\ s\,n_1}} = s\cdot\norm{n} = s\,.
\]
In order to find a direction of \emph{maximum tangential linearised stress} (not the tangential load), we decompose the resultant traction $\sigma\,n$ in direction of a given unit normal vector $n$ into a normal and a tangential part:
\[
	\norm{\sigma\,n}^2 = \underbrace{\innerproduct{\sigma\,n,\: n}^2}_{\text{normal}} + \underbrace{\innerproduct{\sigma\,n,\: n_\perp}^2}_{\text{tangential}}\,,
\]
where $n_\perp$ is a unit vector normal to $n$. Since for such a stress the resultant $\norm{\sigma\,n}^2 = s^2$ is constant\footnote{Note carefully that this is no longer true in the nonlinear case: $\frac{\norm{\sigma n}}{\norm{n}}$ is not constant for $\sigma=\diag(s\,\alpha,-s\,\alpha\inv,1)$ and $\alpha>1$, i.e. for a (nonlinear) Cauchy stress tensor $\sigma$ corresponding to a pure shear load $\Biot = \diag(s,-s,0)$ and a shear deformation $F=\diag(\alpha,\frac1\alpha,1)$.}, i.e. independent of the unit normal $n$, the amount of tangential stress
\[
	\innerproduct{\sigma\,n,\: n_\perp}^2 = s^2 - \innerproduct{\sigma\,n,\: n}^2
\]
assumes its maximum among all $n\in\R^2$ with $\norm{n}=1$ if and only if $\innerproduct{\sigma\,n,\: n}^2$ attains its minimum. We find
\begin{align*}
	\innerproduct{\sigma\,n,\: n}^2 = \innerproduct{\matr{s\,n_2\\ s\,n_1}, \; \matr{n_1\\ n_2}} = s^2\,n_1\,n_2\,,
\end{align*}
which is minimal if and only if $n_1=0$ or $n_2=0$. Since the directions of the principal axes are given by the eigenvectors $(1,1)^T$ and $(1,-1)^T$ of $\eps$, the vectors $n_1$ and $n_2$ cut these axes at angles of $45^\circ$.

\section{Applications and properties of Becker's law of elasticity}\label{section:properties}
\subsection{Infinitesimal deformations}\label{subsection:infinitesimalDeformations}
For small deformations the linear approximation
\begin{align*}
	T(\id+\eps) \;&=\; 2\,G\cdot\log(\id+\eps) \;+\; \Lambda\cdot\tr[\log (\id+\eps)]\cdot\id\\
	\;&=\; 2\,G\cdot (\eps + \mathcal{O}(\norm{\eps}^2)) \;+\; \Lambda\cdot(\tr[\eps + \mathcal{O}(\norm{\eps}^2)])\cdot\id \;=\; 2\,G\cdot \eps \;+\; \Lambda\cdot\tr (\eps)\cdot\id \;+\; \mathcal{O}(\norm{\eps}^2)
\end{align*}
shows that the stress-stretch relation is compatible with the model of linear elasticity if and only if $G$ and $\Lambda$ are the two Lam\'e constants. In this case the additional constraints
\[
	\shear > 0\,, \quad \bulk>0
\]
follow from the uniform positivity of the linear strain energy density
\begin{equation*}
	W_{\mathrm{lin}}(\eps) \;=\; G\norm{\dev_3\sym\eps}^2 \,+\, \frac\bulk2[\tr(\eps)]^2 \;=\; G\,\norm{\sym\eps}^2 \,+\, \frac\Lambda2\,[\tr(\eps)]^2\,.
\end{equation*}

\subsection{Uniaxial stresses}
If the initial load is given by a Biot stress tensor of the form
\[
	\Biot = \dmatr{0}{Q}{0}\,,
\]
we can give an explicit formula for the stretch tensor $U$ corresponding to $\Biot$: since $\tr\Biot=Q$ and
\[
	\dev_3\Biot = \Biot-\frac13\tr[\Biot]\cdot\id = \dmatr{-\frac{Q}{3}} {\frac{2\,Q}{3}} {-\frac{Q}{3}}
\]
we can use equation \eqref{eq:gerneralStressStretchRelationInverted} to find
\[
	U(\Biot) = e^{\frac{Q}{9\,\bulk}}  \;\cdot\;  \dmatr{e^{-\frac{Q}{2G\cdot3}}} {e^{\frac{2\,Q}{2G\cdot3}}} {e^{-\frac{Q}{2G\cdot3}}} = e^{\frac{Q}{9\,\bulk}}  \;\cdot\;  \dmatr{e^{-\frac{Q}{6\,G}}} {e^{\frac{Q}{3\,G}}} {e^{-\frac{Q}{6\,G}}}\,.
\]
In particular, the deformation along the axis of stress is given by\footnote{More details on the conversion of the material parameters can be found in Appendix \ref{section:moduli}.}
\begin{equation}\label{eq:uniaxialStressStretchRelationMainAxis}
	\lambda_2 \;=\; e^{\frac{Q}{9\,\bulk}} \cdot e^{\frac{Q}{3\,G}} \;=\; e^{Q\cdot(\frac{1}{9\bulk}+\frac{1}{3G})} \;=\; e^{\frac{Q}{E}}\,,
\end{equation}
while the deformation along the axes orthogonal to the stress axis is
\begin{equation}\label{eq:uniaxialStressStretchRelationOtherAxes}
	\lambda_1 = \lambda_3 = e^{\frac{Q}{9\,\bulk}} \cdot e^{-\frac{Q}{6\,G}} = e^{Q\cdot(\frac{1}{9\bulk}-\frac{1}{6G})} \;=\; e^{-\frac{\nu\,Q}{E}}\,.
\end{equation}
The factors $e^{\frac{Q}{9\,\bulk}}$ and $e^{\frac{Q}{3\,G}}$ appearing in equation \eqref{eq:uniaxialStressStretchRelationMainAxis} are the \emph{dilational stretch} and the \emph{shear stretch}, respectively, as given by Becker in equation \eqref{5} on page 345 as his main result. Furthermore, equation \eqref{eq:uniaxialStressStretchRelationOtherAxes} shows that in the case $9\,\bulk = 6\,G$, which corresponds to $\nu=0$ for Poisson's ratio $\nu$, the stretch along the unstressed axes is $1$. Therefore, as in the linear model, there is no lateral contraction in Becker's model for $\nu=0$. A similar result holds for Hencky's elastic law \cite{vallee1978,vallee2008}.

\subsubsection{Application to incompressible materials}\label{section:incompressibleCase}
To apply the uniaxial stress response to \emph{incompressible materials} we will now consider the limit $\bulk\to\infty$, i.e. we approximate the incompressible case through the \emph{nearly incompressible} case with a sufficiently large ratio $\frac{\bulk}{G}$. From equations \eqref{eq:uniaxialStressStretchRelationMainAxis} and \eqref{eq:uniaxialStressStretchRelationOtherAxes} we readily obtain
\[
	\lim_{\bulk\to\infty} \lambda_2 = e^{\frac{Q}{3\,G}}
\]
as well as
\[
	\lim_{\bulk\to\infty} \lambda_1 = \lim_{\bulk\to\infty} \lambda_3 = e^{-\frac{Q}{6\,G}}\,,
\]
thus in our theory uniaxial stress induces deformations of the form
\[
	U = \dmatr{\frac{1}{\sqrt\lambda}} {\lambda} {\frac{1}{\sqrt\lambda}} = \dmatr{e^{-\frac{Q}{6\,G}}} {e^{\frac{Q}{3\,G}}} {e^{-\frac{Q}{6\,G}}}
\]
in the incompressible case. Going to the inverse we obtain the formula
\begin{equation}\label{eq:imbertFormula}
	Q = 3\,G\cdot \ln(\lambda) = E\cdot\ln(\lambda)\,,
\end{equation}
where $E = 3\,G$ denotes \emph{Young's modulus} for incompressible materials.

Equation \eqref{eq:imbertFormula} is identical to the uniaxial stress-stretch relation given by Imbert in 1880 as a phenomenological model for the deformation of vulcanized rubber bands under tension \cite[p. 53]{imbert1880}. Similarly, in 1893 Hartig applied the same logarithmic law to describe the uniaxial tension and compression of rubber \cite{hartig1893}. A comparison of Becker's results for very large strain to experimental data by L.R.G. Treloar for the uniaxial deformation of vulcanized rubber \cite{jones1975} as well as the corresponding stress responses for the quadratic Hencky energy and Ogden's elasticity model \cite{ogden2004} are shown in Fig. \ref{figure:incompressibleCurve}. Another possible way to apply Becker's law to incompressible materials is described in section \ref{section:valanisLandel}.
\begin{figure}[h]
	\centering
	\begin{tikzpicture}[scale=1]
		\ifshowimages
			\input{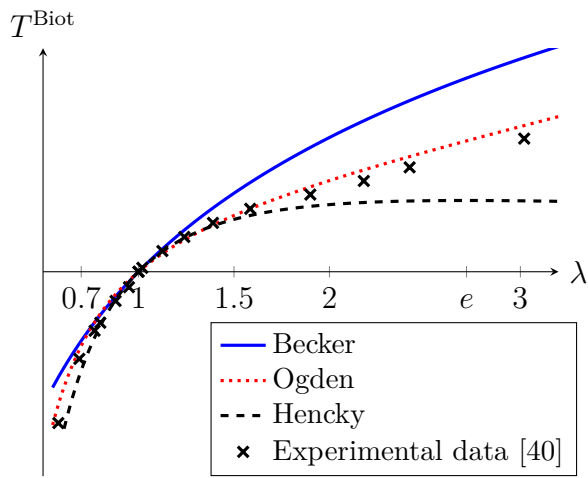}
		\fi
	\end{tikzpicture}
	\caption{Comparison of stress responses for incompressible materials.}
	\label{figure:incompressibleCurve}
\end{figure}

\subsection{Becker's law of elasticity for simple shear}\label{section:simpleGlideComputation}
Consider a simple glide deformation of the form
\[
	F = \matr{1&\gamma&0\\ 0&1&0\\ 0&0&1}
\]
with $\gamma>0$, c.f. section \ref{section:simpleGlide}. Then the polar decomposition of $F = R\cdot U$ into the right Biot stretch tensor $U=\sqrt{F^TF}$ of the deformation and the orthogonal polar factor $R = FU\inv$ is given by
\[
	U \;=\; \frac{1}{\sqrt{\gamma^2+4}} \cdot \matr{2&\gamma&0\\ \gamma&\gamma^2+2&0\\ 0&0&\sqrt{\gamma^2+4}}\,, \qquad
	R \;=\; \frac{1}{\sqrt{\gamma^2+4}}\cdot\matr{2&\gamma&0\\-\gamma&2&0\\0&0&\sqrt{\gamma^2+4}}
\]
Further, $U$ can be diagonalized to
\begin{align*}
	U \;=\; L \,\cdot\, \dmatr{1}{\frac12(\sqrt{\gamma^2+4}+\gamma)}{\frac12(\sqrt{\gamma^2+4}-\gamma)} \,\cdot\, L\inv \;=\; L \,\cdot\, \dmatr{1}{\lambda_1}{\frac{1}{\lambda_1}} \,\cdot\, L\inv\,,
\end{align*}
where
\[
	L \;=\; \matr{0&2&-2\\ 0&\sqrt{\gamma^2+4}+\gamma&\sqrt{\gamma^2+4}-\gamma\\ 1&0&0}
\]
and $\lambda_1 = \frac12(\sqrt{\gamma^2+4}+\gamma)$ denotes the first eigenvalue of $U$, and the principal logarithm of $U$ is
\begin{align*}
	\log U \;&=\; L \,\cdot\, \log \dmatr{1}{\lambda_1}{\frac{1}{\lambda_1}} \,\cdot\, L\inv \;=\; L \,\cdot\, \dmatr{0}{\ln(\lambda_1)}{-\ln(\lambda_1)} \,\cdot\, L\inv\\[2mm]
	&=\; \frac{1}{\sqrt{\gamma^2+4}} \cdot \matr{-\gamma\,\ln(\lambda_1) & 2\,\ln(\lambda_1) & 0 \\ 2\,\ln(\lambda_1) & \gamma\,\ln(\lambda_1) & 0 \\ 0&0&0}\,.
\end{align*}
Then according to Becker's law of elasticity, the first Piola-Kirchhoff stress tensor $\PKone$ corresponding to $F$ computes to
\begin{align*}
	\PKone(F) \;&=\; R\cdot\Biot(U) \;=\; R\cdot (2\,G\cdot\log(U) \;+\; \Lambda\cdot\ln(\det U)\cdot\id) \;=\; 2\,G\cdot R \cdot \log U\\[2mm]
	&=\; \frac{2\,G}{\gamma^2+4} \cdot \matr{2&\gamma&0\\-\gamma&2&0\\0&0&\sqrt{\gamma^2+4}} \cdot \matr{-\gamma\,\ln(\lambda_1) & 2\,\ln(\lambda_1) & 0 \\ 2\,\ln(\lambda_1) & \gamma\,\ln(\lambda_1) & 0 \\ 0&0&0} \\[2mm]
	&=\; \frac{2\,G}{\gamma^2+4} \cdot \matr{0& (4+\gamma^2)\,\ln(\lambda_1)& 0 \\ (4+\gamma^2)\,\ln(\lambda_1) &0&0 \\ 0&0&0}\\[2mm]
	&=\; 2\,G\cdot \ln(\tfrac12(\sqrt{\gamma^2+4}+\gamma)) \cdot \matr{0& 1& 0 \\ 1 &0&0 \\ 0&0&0}\,.
\end{align*}
Finally, the Cauchy stress $\sigma$ for the simple glide deformation $F$ is
\[
	\sigma \;=\; \frac{1}{\det F} \cdot \PKone(F) \cdot F^T \;=\; 2\,G\cdot \ln(\tfrac12(\sqrt{\gamma^2+4}+\gamma)) \cdot \matr{\gamma&1&0 \\ 1&0&0 \\ 0&0&0}\,.
\]
Note that $\sigma$ is independent of the Lam\'e constant $\Lambda$. In particular, the \emph{simple shear stress} $\sigma_{12}$ corresponding to the amount of shear $\gamma$ is given by
\[
	\sigma_{12} \;=\; \ln(\tfrac12(\sqrt{\gamma^2+4}+\gamma))\,.
\]
for Becker's law of elasticity. Fig. \ref{figure:simpleGlideExperimentalData} shows a comparison of the simple shear stress resulting from different constitutive laws with experimental data measured by Treloar \cite{jones1975} for shear deformations of vulcanized rubber.
\begin{figure}[t]
	\centering
	\begin{tikzpicture}[scale=1]
		\ifshowimages
			\input{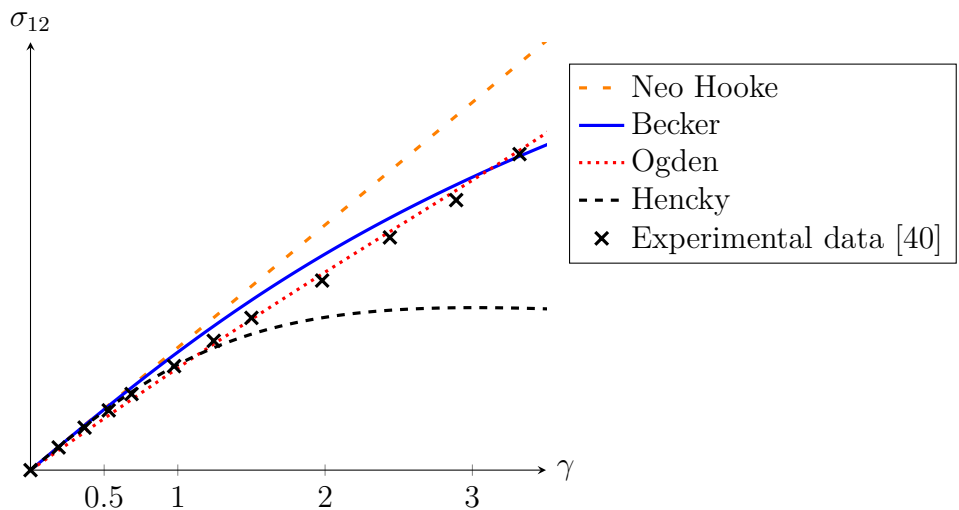}
		\fi
	\end{tikzpicture}
	\caption{Shear stress in a simple glide deformation.}
	\label{figure:simpleGlideExperimentalData}
\end{figure}

\subsection{A comparison of Becker's and Hencky's laws of elasticity}
\label{section:beckerHenckyComparison}
The stress-stretch relation corresponding to the quadratic Hencky strain energy is, in terms of the Kirchhoff stress $\tau$, the left Biot stretch tensor $V$ and the Finger tensor $B$, given by
\begin{equation}\label{eq:tauHencky}
	\tauHencky = 2\,\shear\, \log V + \Lambda \tr(\log V) \cdot \id = \shear \log B + \frac\Lambda2 \tr(\log B)\cdot\id\,,
\end{equation}
while Becker's stress-stretch relation, expressed in terms of the Biot stress $\Biot$, the right stretch tensor $U$ and the right Cauchy-Green deformation tensor $C$, is
\[
	\BiotBecker = 2\,\shear\, \log U + \Lambda \tr(\log U) \cdot \id = \shear \log C + \frac\Lambda2 \tr(\log C)\cdot\id\,.
\]
Since, in general,
\[
	\Biot = U\cdot\PKtwo = \det(U)\cdot U\cdot F\inv \cdot \Cauchy \cdot F^{-T}
\]
and
\[
	\tau = \det(U)\cdot\sigma\,,
\]
where $\PKtwo$ is the symmetric second Piola-Kirchhoff stress and $\sigma$ is the Cauchy stress tensor, we find
\begin{align*}
	\Cauchy &= \det(U)\inv \cdot F \cdot U\inv \cdot \Biot \cdot F^T\\
	\Longrightarrow \qquad \tau &= \underbrace{F \cdot U\inv}_{=R} \,\cdot\, \Biot \cdot F^T = \underbrace{V\inv \cdot F}_{=R} \,\cdot\, \Biot \cdot F^T\,,
\end{align*}
where the last equality follows directly from the polar decomposition $F=R\,U=VR$ of the deformation gradient $F$. We employ the identity $C=F^TF = F\inv FF^T F = F\inv B F$ to compute
\begin{align*}
	\tauBecker = V\inv \cdot F \cdot \BiotBecker \cdot F^T &= V\inv F \cdot \left[\shear \log C + \frac\Lambda2 \tr(\log C)\cdot\id\right] \cdot F^T\\
	&= V\inv F \cdot \left[\shear \log (F\inv B F) + \frac\Lambda2 \tr(\log (F\inv B F))\cdot\id\right] \cdot F^T\\
	&= V\inv F \cdot \left[\shear \, F\inv(\log B)\,F + \frac\Lambda2 \tr(F\inv(\log B)F)\cdot\id\right] \cdot F^T\\
	&= \shear \: V\inv FF\inv(\log B)\,FF^T + \frac\Lambda2 \tr(\log B)\cdot V\inv \smash{\underbrace{FF^T}_{=B}}\\
	&= \shear \: V\inv (\log B)\,B + \frac\Lambda2 \tr(\log B)\cdot V\inv B\\
	&= V\inv\cdot \left[ \shear (\log B) + \frac\Lambda2 \tr(\log B)\cdot\id\right] \cdot B \quad=\quad V\inv\, \tauHencky\: B\,.
\end{align*}
The symmetric tensors $V\inv$, $\tauHencky$ and $B$ commute because their principal axes coincide, therefore
\begin{equation}\label{eq:tauBeckerInHenckyTerms}
	\tauBecker \:=\: V\inv\, \tauHencky\: B \:=\: V\inv\, B \cdot \tauHencky \:=\: RF\inv\: FF^T \cdot \tauHencky \:=\: R F^T\cdot\tauHencky \:=\: V\cdot\tauHencky\,.
\end{equation}
This identity allows us to obtain an upper estimate for the difference between the Kirchhoff stress corresponding to the Hencky energy and the one given by Becker's stress-stretch relation:
\[
	\tauBecker = V\cdot\tauHencky = \tauHencky + (V-\id)\cdot\tauHencky \quad \Longrightarrow \quad \norm{\tauBecker - \tauHencky} \leq \norm{V-\id}\cdot\norm{\tauHencky}\,,
\]
where $\norm{\,.\,}$ denotes the Frobenius matrix norm on $\Rnn$. Thus, for very small elastic strains $\norm{V-\id} \ll 1$, the corresponding Kirchhoff tensors $\tauBecker$ and $\tauHencky$ coincide to lowest order.

\subsection{Hyperelasticity}
Unlike Hencky's logarithmic stress-stretch relation, Becker's idealized response is generally not hyperelastic for arbitrary parameters $G$ and $K$. Incidentally, this result was also shown by Carroll \cite{carroll2009}, who gave an explicit example of a cycle of loading and unloading without conservation of energy, showing that the elastic behaviour modelled, in fact, by Becker's law\footnote{Although the stress response considered by Carroll is identical to the one deduced by Becker, Carroll seems not to be aware of Becker's work.} is not path-independent.
\begin{proposition}
The stress-stretch relation
\[
	\Biot(U) = 2\,G\cdot\log(U) \;+\; \Lambda\cdot\tr[\log U]\cdot\id
\]
is hyperelastic if and only if $\Lambda=0$ or, equivalently, $\nu=0$ for Poisson's number $\nu$. In this case\footnote{For example, the parameter $\nu=0$ is used to model the elastic behaviour of cork.} the energy is given by
\begin{equation}\label{eq:beckerEnergyFunction}
	\WBecker(U) \;=\; 2\,G\; [\,\iprod{U,\,\log(U) - \id} + 3\,] \;=\; 2\,G\; [\,\iprod{\exp(\log U),\,\: \log U - \id} + 3\,]\,,
\end{equation}
which is the maximum entropy function.
\end{proposition}
However, Becker's law is, of course, \emph{Cauchy-elastic} for all admissible choices of parameters since the Cauchy stress depends only on the state of deformation.

Note that the elastic energy $\WBecker$ given by equality \eqref{eq:beckerEnergyFunction} does not fulfil some of the constitutive properties listed in section \ref{section:reflectionsOnConstitutiveAssumptions}. For example, $\det F \to 0$ does not generally imply $\WBecker(F)\to\infty$; in fact, $\WBecker(F)$ remains finite even for $F=0$. However, the implication
\[
	\det F \to 0 \quad\Longrightarrow\quad \Biot(F) \to \infty
\]
holds true for Becker's elastic law, even in the case $\nu=0$.
\DeclareRobustCommand\tikzcaptionEnergy{\tikz[baseline=-0.6ex]{\draw[ultra thick, color=red, dotted] (0,0) to (.5,0);}}
\DeclareRobustCommand\tikzcaptionBiot{\tikz[baseline=-0.6ex]{\draw[ultra thick, color=blue] (0,0) to (.5,0);}}
\begin{figure}[h]
	\centering
	\begin{tikzpicture}[scale=1]
		\ifshowimages
			\input{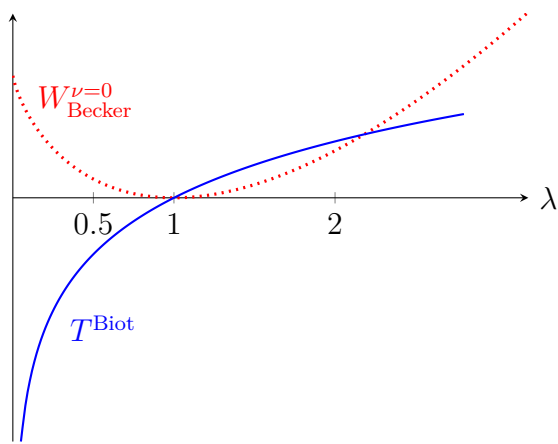}
		\fi
	\end{tikzpicture}
	\caption{Energy (\tikzcaptionEnergy) and Biot stress (\tikzcaptionBiot) according to Becker's law for extreme volumetric stretches $\lambda$; the tangent of $\WBecker$ in $0$ is vertical.}
\end{figure}

\subsubsection{Comparison to the Valanis-Landel energy}\label{section:valanisLandel}
In terms of the principal stretches $\lambda1_1,\lambda_2,\lambda_3$, i.e. the eigenvalues of a stretch tensor $U$, the energy function $\WBecker$ can be expressed as
\begin{align*}
	\WBecker(U) \;&=\; 2\,G\; [\,\iprod{U,\,\log(U) - \id} + 3\,]\\
	&=\; 2\,G \cdot \left[\,\sum_{i=1}^3 \lambda_i\cdot(\ln(\lambda_i)-1)\,\right] + 6\,G \;=:\; \WBeckerhat(\lambda_1,\lambda_2,\lambda_3)\,.
\end{align*}
This energy function is identical to the \emph{Valanis-Landel energy}, which was introduced in 1967 by K.C. Valanis and R.F. Landel \cite{valanis1967}. However, the Valanis-Landel energy is used as a model for \emph{incompressible} hyperelastic materials exclusively, while $\WBecker$ is only applicable to Becker's law of elasticity in the (compressible) case $\nu=0$ or, equivalently, for $\Lambda=0$. Since Becker only considers compressible materials, it is not clear how to extend his constitutive law to the incompressible case. One possible way, involving the limit $\bulk\to\infty$, was discussed in section \ref{section:incompressibleCase}. Another possibility, however, is to directly apply the incompressibility condition $\det F = 1$, $F$ the deformation gradient, to the hyperelastic model induced by the energy $\WBecker$. This approach leads to a different result for uniaxial deformations: using the general formula \cite{ogden2004}
\[
	t \eqq \frac{\rm d}{{\rm d}\lambda} \; \What\left(\lambda,\tfrac{1}{\sqrt{\lambda}},\tfrac{1}{\sqrt{\lambda}}\right)\,,
\]
where $t$ is the (uniaxial) load, $\lambda$ is the stretch and $\What$ is an energy function of an incompressible hyperelastic material expressed in the principal stretches, we find
\begin{align}
	t_{\rm Becker} \;&=\; \frac{\rm d}{{\rm d}\lambda} \; \WBeckerhat\left(\lambda,\tfrac{1}{\sqrt{\lambda}},\tfrac{1}{\sqrt{\lambda}}\right)\label{eq:incompressibleHyperelastic} \\
	&=\; 2\,G\cdot \frac{\rm d}{{\rm d}\lambda} \; \left[ \lambda\cdot(\ln(\lambda)-1) + 2\,\lambda^{-\afrac12}\cdot(\ln(\lambda^{-\afrac12})-1) \right] \eqq G\cdot \ln(\lambda)\cdot(2+\lambda^{-\afrac32})\,. \nonumber
\end{align}
Fig. \ref{figure:incompressibleCurveHyperelasticAlternative} shows the Biot stress response for uniaxial deformations, computed from the two different applications \eqref{eq:imbertFormula} and \eqref{eq:incompressibleHyperelastic} of Becker's law to the incompressible case.
\DeclareRobustCommand\tikzcaptionLimit{\tikz[baseline=-0.6ex]{\draw[ultra thick, color=blue, dashed] (0,0) to (.5,0);}}
\DeclareRobustCommand\tikzcaptionHyper{\tikz[baseline=-0.6ex]{\draw[ultra thick, color=red] (0,0) to (.5,0);}}
\begin{figure}[H]
	\centering
	\begin{tikzpicture}[scale=1]
		\ifshowimages
			\input{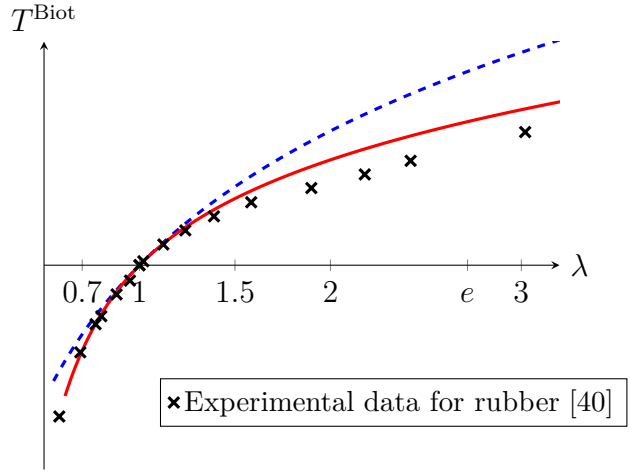}
		\fi
	\end{tikzpicture}
	\caption{Becker's law for the uniaxial deformation of incompressible materials, obtained via the limit case $\bulk\to\infty$ (\tikzcaptionLimit) and by applying the incompressibility restraint $\det F=1$ to the energy function $\WBecker$ (\tikzcaptionHyper).}
	\label{figure:incompressibleCurveHyperelasticAlternative}
\end{figure}

\subsection{Becker's law of elasticity in terms of other stresses and stretches}
In section \ref{section:beckerHenckyComparison}, we have already established the relation between the left Biot stretch tensor $V$ and the Kirchhoff stress tensor $\tau$ for Becker's law of elasticity: combining equations \eqref{eq:tauBeckerInHenckyTerms} and \eqref{eq:tauHencky} we find
\begin{equation}\label{eq:tauBecker}
	\tau(V) \;=\; V \cdot (2\,G\log V + \Lambda\tr(\log V)\cdot\id) \;=\; 2\,G\cdot V \cdot \log V + \Lambda \tr(\log V)\cdot V\,.
\end{equation}
Since $\tau = \det(V)\cdot\sigma$ in general, the Cauchy stress tensor $\sigma$ can be expressed as
\begin{equation}\label{eq:sigmaBecker}
	\sigma(V) \;=\; \frac{2\,G}{\det(V)}\cdot V \cdot \log V + \frac{\Lambda}{\det(V)} \cdot \tr(\log V)\cdot V\,.
\end{equation}
Furthermore, we can obtain a representation of the symmetric second Piola-Kirchhoff stress tensor $\PKtwo$ from the general formula $\Biot = U\cdot\PKtwo$:
\begin{align}\label{eq:pktwoBecker}
	\PKtwo(U) \;=\; U\inv\cdot\Biot(U) \;&=\; 2\,G\cdot U\inv\cdot\log(U) \;+\; \Lambda\cdot\tr[\log U]\cdot U\inv\\
	&=\; (2\,G\cdot\log(U) \;+\; \Lambda\cdot\tr[\log U])\cdot U\inv\,.\nonumber
\end{align}

\subsection{Constitutive inequalities}
\subsubsection{Invertibility of the force-stretch relation}
The \emph{invertibility of the force-stretch relation}, also known as Truesdell's \emph{IFS condition} \cite[p. 156]{truesdell65}, is fulfilled by a stress-stretch relation if and only if the mapping $U\mapsto\Biot(U)$ is invertible. Becker's law of elasticity satisfies this condition, as was already shown in section \ref{section:deduction}.
\subsubsection{The M-condition}\label{section:MCondition}
Since the principal matrix logarithm $\log: \PSymn\to\Symn$ is strictly monotone \cite{NeffMartin2013}, the \emph{Krawietz M-condition} \cite{krawietz1975}
\begin{equation}\label{eq:krawietz}
	\innerproduct{\Biot(U_1)-\Biot(U_2),\: U_1-U_2} > 0 \qquad \forall\: U_1, U_2\in\PSym(3),\;\; U_1 \neq U_2\,,
\end{equation}
where $\innerproduct{X,Y} = \tr(Y^TX)$ denotes the canonical inner product on $\Rnn$, is satisfied by the stress-stretch relation
\[
	\Biot(U) = 2\,G\,\cdot\log(U)\,,
\]
i.e. in the special case $\Lambda=0$. However, it is not satisfied in the general case
\[
	\Biot(U) = 2\,G\,\cdot\log(U) + \Lambda\cdot\tr[\log U]\cdot\id
\]
for sufficiently large $\bulk>0$: choosing
\[
	U_1 = \dmatr{2}{\tel4}{1} \quad \text{ and } \quad U_2 = \id
\]
we find
\begin{align*}
	T(U_1) &= 2\,G\cdot\log(U_1) \;+\; \Lambda\cdot\tr[\log U]\cdot\id\\
	&= 2\,G\cdot \dmatr{\ln 2}{\ln \tel4}{0} \;+\; \Lambda\cdot\ln(\det U_2)\cdot\id\\[2mm]
	&= 2\,G\cdot \dmatr{\ln 2}{-\ln 4}{0} \;+\; \Lambda\cdot\ln\left(\half\right)\cdot\id \;\;=\;\; 2\,G\cdot \dmatr{\ln 2}{-\ln 4}{0} \;-\; \Lambda\cdot\ln(2)\cdot\id
\end{align*}
as well as $T(U_2) = T(\id) = 0$ and thus
\begin{align*}
	\innerproduct{\Biot(U_1)-\Biot(U_2),\: U_1-U_2} &= \innerproduct{2\,G\cdot \dmatr{\ln 2}{-\ln 4}{0} \;-\; \Lambda\cdot\ln(2)\cdot\id\,, \quad \dmatr{1}{-\frac34}{0}}\\
	&= 2\,G\cdot\Big[\ln(2) + (-\ln(4))\cdot\Big(-\frac34\Big)\Big] - \Lambda\cdot\ln(2)\cdot\Big[1 - \frac34\Big]\\
	&= 2\,G\cdot\Big[\ln(2) + \frac{3\cdot\ln(2^2)}{4}\Big] - \frac{\Lambda\cdot\ln(2)}{4}\\
	&= 2\,G\cdot\Big[\frac{4\cdot\ln(2)}{4} + \frac{6\cdot\ln(2)}{4}\Big] - \frac{\Lambda\cdot\ln(2)}{4}\\
	&= \frac{\ln(2)}{4} \cdot \Big[ 20\,G - \Lambda \Big] \;<\; 0
\end{align*}
for $\Lambda>20\,G$.

\subsubsection{Hill's inequality}
Since the energy function of any hyperelastic law satisfying the M-condition is convex in terms of the right Biot-stretch tensor $U$, it follows from section \ref{section:MCondition} that the mapping
\[
	U \,\mapsto\, \WBecker(U) \;=\; 2\,G\; [\,\iprod{U,\,\log(U) - \id} + 3\,]
\]
is convex on $\PSymn$. However, the mapping $X\mapsto \innerproduct{\exp(X), X-\id}$ is not convex on $\Symn$. Therefore $\WBecker$ does not satisfy \emph{Hill's inequality} \cite{hill1970}, which holds for an energy function $W:\PSymn\to\R,\; U\mapsto W(U)$ if and only if the mapping
\[
	\Wtilde:\Symn\to\R\,, \quad \Wtilde(X) = W(\exp(X))\,, 
\]
is convex. This condition is often restated as the convexity of the mapping $\log U \mapsto \Wtilde(\log U)$ for $U\in\PSymn$. This inequality is independent of the rank-one convexity of the energy: for example, while it is easy to see that the quadratic Hencky strain energy fulfils Hill's inequality, it is not rank-one convex \cite{bruhns2001, Neff_Diss00}.

\subsubsection{The Baker-Ericksen inequality}
A stress-stretch relation fulfils the \emph{Baker-Ericksen inequality} if
\[
	(\sigma_i - \sigma_j) \cdot (\lambda_i - \lambda_j) \,>\, 0 \qquad \text{for all } \lambda_i\neq\lambda_j\,,
\]
where $\lambda_k$ denotes the $k$-th principal stretch and $\sigma_k$ denotes the corresponding principal Cauchy stress, i.e. the corresponding eigenvalue of the Cauchy stress tensor $\sigma$.
\begin{proposition}
The stress-stretch relation
\[
	\sigma(V) \;=\; \frac{2\,G}{\det(V)}\cdot V \cdot \log V + \frac{\Lambda}{\det(V)} \cdot \tr(\log V)\cdot V
\]
does not satisfy the Baker-Ericksen inequality for any $G>0$, $3\,\Lambda+2\,G>0$.
\end{proposition}
\begin{proof}
We assume without loss of generality that $V$ is in the diagonal form $V=\diag(\lambda_1,\lambda_2,\lambda_3)$ with $\lambda_1,\lambda_2,\lambda_3>0$. Then
\[
	\log(V) \;=\; \log \dmatr{\lambda_1}{\lambda_2}{\lambda_3} \;=\; \dmatr{\ln \lambda_1}{\ln \lambda_2}{\ln \lambda_3}
\]
and
\begin{align*}
	\sigma(V) \;&=\; \frac{2\,G}{\lambda_1\lambda_2\lambda_3}\cdot V \cdot \log V + \frac{\Lambda}{\lambda_1\lambda_2\lambda_3} \cdot \ln(\lambda_1\lambda_2\lambda_3)\cdot V\nonumber\\[2mm]
	&=\; \frac{2\,G}{\lambda_1\lambda_2\lambda_3}\cdot \dmatr{\lambda_1\cdot\ln\lambda_1}{\lambda_2\cdot\ln\lambda_2}{\lambda_3\cdot\ln\lambda_3} + \frac{\Lambda\cdot \ln(\lambda_1\lambda_2\lambda_3)}{\lambda_1\lambda_2\lambda_3} \cdot \dmatr{\lambda_1}{\lambda_2}{\lambda_3}\,.
\end{align*}
The principal stresses $\sigma_k$ are the diagonal entries of $\sigma$, thus
\begin{equation}\label{eq:principalStresses}
	\sigma_k \;=\; \frac{\lambda_k}{\lambda_1\lambda_2\lambda_3}\cdot(2\,G\,\ln\lambda_k \,+\, \Lambda\,\ln(\lambda_1\lambda_2\lambda_3))\,.
\end{equation}
We let $\lambda_1 = \frac1e$, $\lambda_2 = \frac{1}{e^2}$ and $\lambda_3 = e^3$ to find
\[
	\sigma_1 \;=\; \frac{\frac1e}{\frac1e\cdot\frac{1}{e^2}\cdot e^3} \cdot (2\,G\,\ln\left(\frac1e\right) \,+\, \Lambda\,\ln(\tfrac1e\cdot\tfrac{1}{e^2}\cdot e^3)) \;=\; \frac{2\,G}{e}\cdot \ln \left(\frac1e\right) \;=\; -\frac{2\,G}{e}
\]
as well as
\[
	\sigma_2 \;=\; \frac{\frac{1}{e^2}}{\frac1e\cdot\frac{1}{e^2}\cdot e^3} \cdot (2\,G\,\ln\left(\frac{1}{e^2}\right) \,+\, \Lambda\,\ln(\tfrac1e\cdot\tfrac{1}{e^2}\cdot e^3)) \;=\; \frac{2\,G}{e^2}\cdot \ln \left(\frac{1}{e^2}\right) \;=\; -\frac{4\,G}{e^2}\,.%
\]
Since $\lambda_1 > \lambda_2$ we find
\[
	(\sigma_1 - \sigma_2) \cdot (\lambda_1 - \lambda_2) \;>\; 0
	\quad\Longleftrightarrow\quad \sigma_1 \;>\; \sigma_2
	\quad\Longleftrightarrow\quad -\frac{2\,G}{e} \;>\; -\frac{4\,G}{e^2}
	\quad\Longleftrightarrow\quad 1 \;<\; \frac{2}{e}\,,
\]
showing that the Baker-Ericksen inequality does not hold in this case.
\end{proof}
Therefore Becker's law does not satisfy the \emph{rank-one convexity} condition either, since a rank-one convex stress-stretch relation always fulfils the Baker-Ericksen inequality. In contrast, Hencky's elastic law (c.f. \eqref{eq:HenckyHyperelasticLaw}) does fulfil the Baker-Ericksen inequality \cite{NeffGhibaLankeit}.

\subsubsection{The ordered force inequalities}
An isotropic stress-stretch relation satisfies the \emph{ordered force inequalities} (or \emph{OF inequalities}) if
\begin{equation}\label{eq:orderedForce}
	(T_i - T_j)\cdot(\lambda_i-\lambda_j) \geq 0 \quad \text{for all $i,j\in\{1,2,3\}\,, \; i\neq j$,}
\end{equation}
where $\lambda_i,\lambda_j$ are the principal stretches of a deformation and $T_i,T_j$ are the corresponding principal forces, i.e. the eigenvalues of the Biot stretch tensor $\Biot$. To show that Becker's law fulfils the OF inequalities for all $G>0$, $3\,\Lambda+2\,G>0$, we assume w.l.o.g. that a given stretch tensor $U$ is in the diagonal form $U=\diag(\lambda_1,\lambda_2,\lambda_3)$ and compute
\begin{align*}
	\Biot \;&=\; 2\,G\cdot\log(U) + \Lambda\cdot\tr[\log(U)]\cdot\id\\
	&=\; 2\,G\cdot \dmatr{\ln(\lambda_1)}{\ln(\lambda_2)}{\ln(\lambda_3)} + \Lambda\cdot\tr\dmatr{\ln(\lambda_1)}{\ln(\lambda_2)}{\ln(\lambda_3)} \cdot\id\\[2mm]
	&=\; \dmatr{2\,G\,\ln(\lambda_1)}{2\,G\,\ln(\lambda_2)}{2\,G\,\ln(\lambda_3)} \,+\, \Lambda\,\ln(\lambda_1\lambda_2\lambda_3)\cdot\id\,.
\end{align*}
The eigenvalues of $T_i$ of $\Biot$ corresponding to the principal stretches $\lambda_i$ are therefore
\[
	\lambda_i \;=\; 2\,G\,\ln(\lambda_i) + \Lambda\,\ln(\lambda_1\lambda_2\lambda_3)\,,
\]
thus \eqref{eq:orderedForce} can be written as
\begin{align}
	\Big(&2\,G\,\ln(\lambda_i) + \Lambda\,\ln(\lambda_1\lambda_2\lambda_3) - 2\,G\,\ln(\lambda_j) + \Lambda\,\ln(\lambda_1\lambda_2\lambda_3)\Big) \cdot (\lambda_i-\lambda_j) \;\geq\; 0\nnl
	\Longleftrightarrow \quad &2\,G\cdot(\ln(\lambda_i)-\ln(\lambda_j))\cdot(\lambda_i-\lambda_j) \;\geq\; 0 \label{eq:orderedForceBecker}
\end{align}
Due to the monotonicity of the natural logarithm, \eqref{eq:orderedForceBecker} holds for all $\lambda_1,\lambda_2,\lambda_3 \in \R^+$ and all $G>0$.

\subsection{Existence results}
The following proposition represents a basic existence result by Ciarlet \cite[Theorem 6.7-1]{ciarlet1988} for solutions to the so-called \emph{pure displacement problem} in nonlinear elasticity:
\begin{proposition}\label{prop:existenceCiarlet}
Let $\Omega\subset\R^3$ be a domain with a boundary $\Gamma$ of class $C^2$, and let
\[
	E \;=\; \frac12 (C-\id) \;=\; \frac12 \Big((\id+\grad u)^T(\id+\grad u)-\id\Big)
\]
denote the Green-Lagrange strain tensor of a deformation $\varphi(x) = x+u(x)$. Moreover, assume that the constitutive law is of the form
\[
	\PKtwo(E) \;=\; \Lambda \cdot \tr(E) \cdot \id \,+\, 2\,G\,E \,+\, \mathcal{O}(\norm{E}^2)
\]
with $\Lambda, G>0$, where $\PKtwo$ denotes the second Piola-Kirchhoff stress tensor. Then for each number $p>3$ there exists a neighbourhood $Z^p$ of the origin in the space $L^p(\Omega)$ and a neighbourhood $U^p$ of the origin in the subspace
\[
	V^p(\Omega) = \{v\in W^{2,p}(\Omega) \,|\, v = 0 \text{ on } \Gamma\}
\]
of the Sobolev space $W^{2,p}(\Omega)$ such that for each $f\in Z^p$, the boundary value problem
\[
	\left.\begin{aligned}
		-\div S_1(F) \;&=\; f &&\text{in $\Omega$}\\
		u \;&=\; 0 &&\text{on $\Gamma$}\\
		S_1(F) \;=\; (\id+\grad u) \cdot \PKtwo(&E(u))
	\end{aligned}\;\right\}
\]
has exactly one solution $u$ in $U^p$.
\end{proposition}

\noindent To show that Becker's stress-stretch relation fulfils the conditions of Theorem \ref{prop:existenceCiarlet} we compute
\begin{align*}
	\PKtwo \;=\; U\inv\cdot\Biot \;&=\; [2\,G\cdot \log(U) \;+\; \Lambda\cdot\tr[\log U]\cdot\id]\cdot U\inv\\
	&=\; [2\,G\cdot \log(\sqrt{C}) \;+\; \Lambda\cdot\tr[\log \sqrt{C}]\cdot\id]\cdot \sqrt{C}\,\inv\\
	&=\; [G\cdot \log(C) \;+\; \frac{\Lambda}{2}\cdot\tr[\log C]\cdot\id]\cdot \sqrt{C}\,\inv\,.
\end{align*}
For small enough $\norm{C-\id}$, we can employ the series expansion
\[
	\log(C) \;=\; (C-\id) - \frac12\cdot(C-\id)^2 + \; \dots
\]
of the matrix logarithm to find
\begin{align}
	\PKtwo \;&=\; [G\cdot \log(C) \;+\; \frac{\Lambda}{2}\cdot\tr[\log C]\cdot\id]\cdot \sqrt{C}\,\inv\nnl
	&=\; \left[ G\cdot (C-\id + \mathcal{O}(\norm{C-\id}^2)) \;+\; \frac{\Lambda}{2}\cdot\tr[(C-\id) + \mathcal{O}(\norm{C-\id}^2)]\cdot\id \right] \cdot \sqrt{C}\,\inv\nnl
	&=\; \left[ G\cdot (C-\id) \;+\; \frac{\Lambda}{2}\cdot\tr[C-\id]\cdot\id  \;+\; \mathcal{O}(\norm{C-\id}^2) \right] \cdot \sqrt{C}\,\inv\,. \label{eq:existenceComputationOne}
\end{align}
Since
\[
	\sqrt{C}\,\inv \;=\; \id - \frac12 \cdot (C-\id) +  \mathcal{O}(\norm{C-\id}^2)
\]
for small $\norm{C-\id}$, \eqref{eq:existenceComputationOne} can be expressed as
\begin{align*}
	\PKtwo \;&=\; \left[ G\cdot (C-\id) \;+\; \frac{\Lambda}{2}\cdot\tr[C-\id]\cdot\id  \;+\; \mathcal{O}(\norm{C-\id}^2) \right] \cdot \left[ \id - \frac12 \cdot (C-\id) + \mathcal{O}(\norm{C-\id}^2) \right]\\
	&=\; G\cdot (C-\id) \;+\; \frac{\Lambda}{2}\cdot\tr[C-\id]\cdot\id + \mathcal{O}(\norm{C-\id}^2) \;=\; \Lambda \cdot \tr(E) \cdot \id \,+\, 2\,G\,E \,+\, \mathcal{O}(\norm{E}^2)\,.
\end{align*}
Proposition \ref{prop:existenceCiarlet} can therefore be directly applied to Becker's law of elasticity.

\section*{Acknowledgements}
We discovered Becker's original paper in late August 2013 and carefully transcribed it with the help of Mrs. B. Sacha until the end of September 2013. In January 2014 we understood the meaning of Becker's tables on page \pageref{becker:table2} and finished a preliminary version in March 2014. We are grateful to Prof. Kolumban Hutter from the ETH Z\"urich for his helpful remarks on the transcription.
\newpage
\begin{appendix}
\section{Appendix}

\subsection{A brief history of logarithmic strain measures}\label{section:historicalContext}
Becker was not the first to consider a law of elasticity based on the logarithm of the principal stretches. In 1880, A. Imbert proposed a logarithmic stress response function as a model for the uniaxial tension of vulcanized rubber \cite{imbert1880}, while E. Hartig applied a similar law to the uniaxial deformation of rubber \cite{hartig1893} in 1893. However, both of these approaches are purely phenomenological: neither Imbert nor Hartig considers a theoretical framework or states underlying reasons for the use of a logarithmic strain measure, they merely employ the logarithm to give an approximation of data obtained through (uniaxial) experiments.

Although the present article by Becker was summarized in a review article in \emph{Beibl\"atter zu Wiedemanns Annalen der Physik} \cite{annalen1894beib} (by a reviewer only identified as \enquote{G. L\"ubeck, Berlin}) and cited in Lueger's \emph{Lexikon der gesamten Technik} \cite{lueger1894}, Becker's work seems to have gathered little attention outside  the field of geology.
The introduction of the logarithmic strain measure to the theory of elasticity is therefore often attributed to P. Ludwik, for example by Hencky \cite[p. 175]{hencky1931} or Truesdell\footnote{While Imbert's contributions are also mentioned by Truesdell, he only cites a summary by Mehmke \cite{mehmke1897}, who in turn refers to Hartig \cite{hartig1893} instead of Imbert's original paper.} \cite[p. 254]{truesdell60}. However, the earliest mention of the logarithmic strain by Ludwik appears in his 1909 monograph \emph{Elemente der technologischen Mechanik} \cite{ludwik1909} on plastic deformations\footnote{Ludwik arrived at the logarithmic strain measure through the integral $\int_{l_0}^l \frac{d\,l}{l} = \log \frac{l}{l_0}$ over the \emph{instantaneous strain} $\frac{d\,l}{l}$ for uniaxial elongations.}, while Becker derives a detailed connection between stresses and the logarithm\footnote{Although the matrix logarithm had already been investigated in 1892 by W.H. Metzler \cite{metzler1892}, Becker, like Imbert and Ludwik, only considers the (scalar) logarithm of individual stretches instead of the logarithm function in a tensorial setting. The efficient computation of the matrix logarithm is still an open field of research \cite{al2012,al2013}.} of the principal stretches in 1893. This error of attribution seems to originate from H. Hencky who, in a 1931 article \cite{hencky1931}, referred to a brief section on plastic deformations in \emph{H\"utte: Des Ingenieurs Taschenbuch} \cite{hutte1925} where Ludwik is cited. The same misattribution to Ludwik is given by Truesdell \cite{truesdell60}, who does not mention Becker at all.\footnote{Truesdell \cite[p. 270]{truesdell60} also attributes the deduction of the logarithmic measure of strain from a law of superposition to Richter \cite{richter1949log}, although Becker and Hencky used the same approach much earlier. Furthermore, Truesdell \cite[p. 144]{truesdell1952} claims that \emph{\enquote{Hencky himself did not give a systematic treatment}} when introducing the logarithmic strain measure.}

Becker's work can be seen as an early attempt to find an idealized law of nonlinear elasticity for finite deformations through deduction from a number of simple assumptions for the behaviour of an ideally elastic material, predating a remarkably similar approach by Hencky\footnote{Hencky's fundamental view of the natural sciences and their relation to mathematics are laid out in his philosophical article \emph{\"Uber die Beziehungen der Philosophie des ,,Als Ob'' zur mathematischen Naturbeschreibung} \cite{hencky1923}.}, who deduced a logarithmic law of elasticity from the assumption of a law of superposition in his 1928 article \emph{\"Uber die Form des Elastizit\"atsgesetzes bei ideal elastischen Stoffen} \cite{hencky1928, henckyTranslation}. Unlike Becker, however, Hencky gave an explicit motivation for his assumed law of superposition, which he later expanded upon in his 1929 article \emph{Das Superpositionsgesetz eines endlich deformierten relaxationsf\"ahigen elastischen Kontinuums und seine Bedeutung f\"ur eine exakte Ableitung der Gleichungen f\"ur die z\"ahe Fl\"ussigkeit in der Eulerschen Form} \cite{hencky1929super, henckyTranslation}: referring to Prandtl's distinction between \enquote{elastically determinate} and \enquote{elastically indeterminate constructs}\footnote{Prandtl \cite{prandtl1924} calls a system in which \emph{\enquote{already occurring prestresses have no significant influence on the stresses induced by additional loads, i.e. in which the stresses simply superimpose}} elastically determinate.} \cite{prandtl1924}, Hencky assumes that a law of elasticity for an \emph{ideally} elastic body should provide \emph{\enquote{elastic determinacy to the greatest extent for epistemological reasons}} \cite[p. 19]{henckyTranslation}, a requirement motivated by Dingler \cite{dingler1928experiment}.
From this he concludes that the multiplicative composition of coaxial stretches must effect the additive composition of the respective Cauchy stresses $\sigma$, leading to the stress response function
\begin{equation}\label{eq:historicHenckySigma}
	\sigma(V) = 2\,\shear\cdot\dev_3\log(V) + \bulk\cdot\tr[\log V]\cdot\id\,,
\end{equation}
as described in Corollary \ref{cor:hencky}. In a later 1929 article \cite{hencky1929, henckyTranslation}, however, Hencky corrected his statements, proposing then that the law of superposition must hold for the Kirchhoff stress tensor $\tau$ instead of the Cauchy stress. Although his reasoning for this correction is based on L. Brillouin's suggestion \cite{brillouin1925} that the Cauchy stress \emph{\enquote{is not a true tensor of weight 0 but a tensor density}} as well as a \emph{\enquote{lack of group properties for pure deformations in the general case}} \cite[p. 20]{henckyTranslation}, the fact that the stress-stretch relation
\begin{equation}\label{eq:historicHenckyTau}
	\tau(V) = 2\,\shear\cdot\dev_3\log(V) + \bulk\cdot\tr[\log V]\cdot\id
\end{equation}
resulting from this new approach with respect to the Kirchoff stress $\tau$ is \emph{hyperelastic} with the corresponding strain energy
\begin{equation}\label{eq:historicHenckyEnergy}
	W(V) = \shear\,\norm{\dev_3 \log V}^2 + \frac{\bulk}{2}\,[\tr(\log V)]^2
\end{equation}
can be seen as a motivating factor as well, especially since Hencky in his 1928 article explicitly computed that the stress response \eqref{eq:historicHenckySigma} does not lead to a path-independent energy potential and is therefore \emph{not} hyperelastic.

Although his deductions of the stress-stretch relations \eqref{eq:historicHenckySigma} and \eqref{eq:historicHenckyTau} from the respective laws of superposition are correct (c.f. Corollary \ref{cor:hencky}), Hencky does not provide explicit computations for either one. A proof for a generalized version of this deduction from the law of superposition was later given by H. Richter \cite{richter1949log}, who did extensive work on the matrix logarithm in finite elasticity \cite{richter1948, richter1949verzerrung, richter1950, richter1952elastizitatstheorie}.

More information on the historical development of nonlinear elasticity theory and logarithmic strain measures in particular as well as related articles by Becker, Hencky, Richter and other authors can be found under \url{http://www.uni-due.de/mathematik/ag_neff/neff_hencky}\,.

\subsection{The stress tensors}
\label{appendix:stressTensors}
Throughout his work Becker refers to \enquote{initial stresses} as well as \enquote{final stresses}. Since Becker only considers homogeneous deformations along fixed axes, there is some ambiguity as to which stress tensors are represented by these terms. However, Becker's remark that \Bquote{[in] a shear of ratio $\alpha$ with a tensile axis in the direction of $oy$, minus ${\Bnormal}_{x}\alpha$ is the negative stress acting in the direction of the $x$ axis into the area $\alpha$ on which it acts} \Bref{becker:finalStressDefinition} allows us to infer that \enquote{final stress} refers to the force per area in the deformed configuration. Furthermore the terms \enquote{load} and \enquote{initial stress} are often used interchangeably \Bref[e.g. ]{becker:initialStressDefinition}. Since Becker considers the deformation of a unit cube \Bref{becker:unitCube}, the load is the force acting on an area of size 1 in the undeformed configuration, hence his equating load and \enquote{initial stress} strongly suggests that the latter should be interpreted as \enquote{force per unit area of the undeformed configuration}.\\
Note that this information is not sufficient to completely characterize the two stress tensors: since Becker only considers the case of fixed principal axes, the principal directions of the stress tensors are undetermined. However, the assumption of isotropy ensures that the resulting law of elasticity only depends on the principal stress response to deformations along fixed axes. Thus the choice of tensorial directions is irrelevant to the resulting stress-stretch relation.\\
To simplify the resulting expressions we will therefore interpret the term \enquote{final stress} as the \emph{Cauchy stress tensor} $\sigma$ and the \enquote{initial stress} as the \emph{Biot stress tensor} $\Biot$.

\subsection{The basic decomposition of traction by Cauchy stress quadrics}\label{section:basicDecomposition}
Let $\Cauchy$ denote the symmetric Cauchy stress tensor here and throughout. With respect to its principal axes, $\sigma$ has the diagonal representation
\begin{equation}
\label{eq:diagonalStress}
	\Cauchy=\matr{\Cauchy_1&0&0\\0&\Cauchy_2&0\\0&0&\Cauchy_3}\,,
\end{equation}
where $\Cauchy_i$ denotes the $i$-th principal stress. Then for a given plane in the deformed configuration, the traction $t$ in direction $n$ is given by
\[
	t = \Cauchy \,n\,,
\]
where $n$ is the unit normal vector of the plane. If $n=(n_1,n_2,n_3)^T$ is the representation of $n$ with respect to the principal axes of $\Cauchy$, the traction $t$ computes to
\[
	t = \Cauchy \, n = \matr{\Cauchy_1&0&0\\0&\Cauchy_2&0\\0&0&\Cauchy_3} \cdot \matr{n_1\\n_2\\n_3} = \matr{\Cauchy_1\,n_1\\\Cauchy_2\,n_2\\\Cauchy_3\,n_3}\,.
\]
Therefore the magnitude $\Bresultant$ of the traction, which is also called the \emph{resultant stress on the plane} by Becker \Bref{becker:resultantStress}, is given by
\[
	\Bresultant^2 = \norm{t}^2 = \norm{\Cauchy \, n}^2 = \Cauchy_1^2\,n_1^2 + \Cauchy_2^2\,n_2^2 + \Cauchy_3^2\,n_3^2\,.
\]
By decomposing the traction $t = t_N + t_T$ into a tangential part $t_T$ parallel to the plane and a normal part $t_N$ we obtain the magnitude of \emph{normal stress} $\Bnormal$ via
\[
	\Bnormal = \innerproduct{t,n} = \innerproduct{\Cauchy\, n, n} = \innerproduct{\matr{\Cauchy_1\,n_1\\\Cauchy_2\,n_2\\\Cauchy_3\,n_3},\: \matr{n_1\\n_2\\n_3}} = \Cauchy_1\,n_1^2 + \Cauchy_2\,n_2^2 + \Cauchy_3\,n_3^2\,,
\]
as well as the magnitude of \emph{tangential stress} $\Btangent$: since $\Bresultant^2 = \Btangent^2 + \Bnormal^2$ by Pythagoras' theorem we obtain
\begin{align}
	\Btangent^2 \;&=\; \Bresultant^2 - \Bnormal^2\nnl
	&=\; \Cauchy_1^2\,n_1^2 \,+\, \Cauchy_2^2\,n_2^2 \,+\, \Cauchy_3^2\,n_3^2 \nnl
	&\qquad- (\Cauchy_1^2\,n_1^4 \,+\, \Cauchy_2^2\,n_2^4 \,+\, \Cauchy_3^2\,n_3^4 \,+\, 2\,\Cauchy_1\,\Cauchy_2\,n_1^2\,n_2^2 \,+\, 2\,\Cauchy_1\,\Cauchy_3\,n_1^2\,n_3^2 \,+\, 2\,\Cauchy_2\,\Cauchy_3\,n_2^2\,n_3^2)\nnl
	&=\; \Cauchy_1^2\,n_1^2(1-n_1^2) \,+\, \Cauchy_2^2\,n_2^2(1-n_2^2) \,+\, \Cauchy_3^2\,n_3^2(1-n_3^2)\nnl
	&\qquad -2\,(\Cauchy_1\,\Cauchy_2\,n_1^2\,n_2^2 \,+\, \Cauchy_1\,\Cauchy_3\,n_1^2\,n_3^2 \,+\, \Cauchy_2\,\Cauchy_3\,n_2^2\,n_3^2)\nnl
	&=\;  \Cauchy_1^2\,n_1^2(n_2^2+n_3^2) \,+\, \Cauchy_2^2\,n_2^2(n_1^2+n_3^2) \,+\, \Cauchy_3^2\,n_3^2(n_1^2+n_2^2)\nnl
	&\qquad -2\,(\Cauchy_1\,\Cauchy_2\,n_1^2\,n_2^2 \,+\, \Cauchy_1\,\Cauchy_3\,n_1^2\,n_3^2 \,+\, \Cauchy_2\,\Cauchy_3\,n_2^2\,n_3^2)\nnl
	&=\; n_1^2\,n_2^2\,(\Cauchy_1^2 - 2\,\Cauchy_1\,\Cauchy_2 + \Cauchy_2^2) \,+\, n_1^2\,n_3^2\,(\Cauchy_1^2 - 2\,\Cauchy_1\,\Cauchy_3 + \Cauchy_3^2) \,+\, n_2^2\,n_3^2\,(\Cauchy_2^2 - 2\,\Cauchy_2\,\Cauchy_3 + \Cauchy_3^2)\nnl
	&=\; (\Cauchy_1-\Cauchy_2)^2\,n_1^2\,n_2^2 \;+\, (\Cauchy_1-\Cauchy_3)^2\,n_1^2\,n_3^2 \;+\, (\Cauchy_2-\Cauchy_3)^2\,n_2^2\,n_3^2\,. \label{eq:tangentialMagnitude}
\end{align}
Note that the tangential Cauchy stress $\Btangent$ is not the \emph{tangential load} Becker refers to as a failure criterion \Bref{becker:ruptureCriterion}. In the case of a pure shear, the tangential load is maximal if $n$ is normal to the plane of no distortion (c.f. \ref{section:maximumTangentialLoadBecker}), while the tangential stress $\Btangent$ attains its maximum if $n$ is normal to \Bquote{planes making angles of $45^\circ$ with the axes} \Bref{becker:maximumTangentialStress}.
\subsection{Becker's computations of the directions of maximum tangential stress}\label{section:maximumTangentialLoadBecker}
As was discussed in section \ref{section:geometryOfShear}, the plane of no distortion is the plane of maximum tangential load in Becker's model. According to Becker (in the footnote on page \pageref{becker:shearEllipse}), the tangential load acting on a plane with unit normal $n = (n_1,n_2,0)^T$ is $\Btangent r$, where\footnote{Note carefully that we have switched $n_1$ and $n_2$ to fit the orientation of our coordinate system as explained in section \ref{section:pureFiniteShearAssumption}.} $1/r^2=\alpha^2 n_{1}^2+\alpha^{-2}n_{2}^2$. His computation of the plane of maximum tangential load depends on his assumption that a pure shear deformation $F$ corresponds to a pure shear stress tensor $\Biot$ (Axiom 1). 
In this case we can compute the Cauchy stress tensor:
\begin{align*}
	&F = \dmatr{\alpha}{\frac1\alpha}{1}\,,\;\; \alpha>1\,, \qquad \Biot = \dmatr{s}{-s}{0}\,,\;\; s\in\R\\[2mm]
	\Longrightarrow \quad& \sigma \;=\; \frac{1}{\det U}\cdot U\inv\cdot F\cdot \Biot\cdot F^T \;=\; 1\cdot F\inv\cdot F \cdot \Biot \cdot F \;=\; \dmatr{\alpha\,s}{-\frac{s}{\alpha}}{0}\,.
\end{align*}
Then, for a unit vector $n=(n_1,n_2,0)^T$, we find
\[
	\Bresultant^2 \;=\; \norm{\sigma\, n}^2 \;=\; s^2\cdot (\alpha^2\,n_1^2 + \alpha^{-2}n_2^2)
\]
as well as
\[
	\Bnormal^2 \;=\; \innerproduct{\sigma\, n\,, n}^2 \;=\; s^2\cdot (\alpha\,n_1^2 - \alpha^{-1}\,n_2^2)^2\,.
\]
As Becker states \Bref{becker:resultantLoad}, the resultant load
\[
	\Bresultant^2\,r^2 \;=\; \frac{1}{\alpha^2 n_{1}^2+\alpha^{-2}n_{2}^2} \cdot s^2\cdot (\alpha^2\,n_1^2 + \alpha^{-2}n_2^2) \;=\; s^2
\]
is independent of $n$. In order to find the normal $n$ to the plane of maximum tangential load, i.e. $n$ such that $\Btangent^2 r^2 = \Bresultant^2\,r^2 - \Bnormal^2$ is maximal, it is therefore sufficient to minimize
\[
	\Bnormal^2 r^2 \;=\; r^2\cdot s^2\cdot (\alpha\,n_1^2 - \alpha^{-1}\,n_2^2)^2\,.
\]
Since the term is nonnegative, the minimum is attained if $\Bnormal^2 r^2=0$, which is the case if $n_2^2 \;=\; \alpha^2\,n_1^2$. As we have seen in \eqref{eq:normalsToPOND} in section \ref{section:planesOfNoDistortion}, this equation characterizes the normals to the plane of no distortion, showing again that they are indeed the planes of maximum tangential load under Becker's assumptions.
\subsection{Conversion of the moduli}
\label{section:moduli}
Throughout his article, Becker refers to the \emph{modulus of cubical dilation} (or \emph{bulk modulus}) $\bulk$, the \emph{modulus of distortion} (or \emph{shear modulus}) $G$ and \emph{Young's modulus} $E$. His equation \Bref{becker:modulusConversion}
\[
	Q\left(\frac{1}{9\bulk} + \frac{1}{3\,G}\right) \;=\; Q / E
\]
follows directly from the well-known conversion formula $E=\frac{9\bulk\,G}{3\bulk+G}$ for these moduli:
\[
	\frac{1}{9\,\bulk} + \frac{1}{3\,G} \;=\; \frac{G + 3\,\bulk}{9\cdot\bulk\,G} \;=\; \frac1E\,.
\]
Similarly, with $\nu = \frac{3\,\bulk - 2\,G}{2(3\,\bulk+G)}$ denoting Poisson's ratio, we find
\[
	\frac{1}{9\,K} - \frac{1}{6\,G} \;=\; \frac{2\,G-3\,\bulk}{18\,\bulk\,G} \;=\; -\frac{3\,\bulk - 2\,G}{2(3\,\bulk+G)} \cdot \frac{2(3\,\bulk+G)}{18\,\bulk\,G} \;=\; -\nu\cdot \frac{3\,\bulk+G}{9\,\bulk\,G} \;=\; -\frac{\nu}{E}\,.
\]
\newpage
\section{Notation}
The following notation is employed throughout the article:
\begin{itemize}
\setcolumnlength{35mm}
\setsecondcolumnlength{21mm}
	\item[] \twoco{$\id \;=\; \diag(1,1,1) = \dmatrs111$}{}identity matrix
	\item[] \twoco{$\Omega_0\subset\R^3$}{}reference configuration
	\item[] \twoco{$\varphi: \Omega_0\to\R^3$}{}deformation mapping
	\item[] \twoco{$F \;=\; \grad\varphi(x)\in\Rnn$}{}deformation gradient
	\item[] \twoco{$U \;=\; \sqrt{F^T F}$}{}right Biot stretch tensor
	\item[] \twoco{$C \;=\; F^T F = U^2$}{}right Cauchy-Green deformation tensor
	\item[] \twoco{$V \;=\; \sqrt{F F^T}$}{}left Biot stretch tensor
	\item[] \twoco{$B \;=\; F F^T = V^2$}{}left Cauchy-Green deformation tensor
	\item[] \twoco{$R \;=\; FU\inv = V\inv F\in\SOn$}{}orthogonal polar factor of the deformation gradient
	\item[] \twoco{$\sigma$}{}Cauchy stress tensor, \enquote{true stress}
	\item[] \twoco{$\tau \;=\; \det(F)\,\cdot\, \sigma$}{}Kirchhoff stress tensor
	\item[] \twoco{$\PKone \;=\; \det(F) \,\cdot\, \sigma \, F^{-T}$}{}first Piola-Kirchhoff stress tensor, \enquote{nominal stress}
	\item[] \twoco{$\PKtwo \;=\; \det(F) \,\cdot\, F\inv \, \sigma \, F^{-T}$}{}symmetric second Piola-Kirchhoff stress tensor
	\item[] \twoco{$\Biot \;=\; U\PKtwo \,=\, R^T\PKone$}{}Biot stress tensor
	\item[] \twoco{$G$, $\Lambda$}{}Lam\'e constants
	\item[] \twoco{$\bulk$}{}bulk modulus
	\item[] \twoco{$E$}{}Young's modulus
	\item[] \twoco{$\nu$}{}Poisson's ratio
\end{itemize}
\end{appendix}
\newpage
\refstepcounter{section}%
\cftaddtitleline{toc}{section}{\protect\numberline{~}References}{\arabic{page}}
\newpage
\changetext{25mm}{-35mm}{15mm}{10mm}{0mm}
\renewcommand{\ln}{\operatorname{ln}}

\markright{\ }
\setcounter{equation}{0}

\refstepcounter{section}%
\cftaddtitleline{toc}{section}{\protect\numberline{~}\emph{The Finite Elastic Stress-Strain Function} by G.F. Becker}{\arabic{page}}
\markright{The Finite Elastic Stress-Strain Function by G.F. Becker}

\setcounter{page}{336}
\parskip 0.9ex
\parindent 10pt

\def\today{}
\thispagestyle{empty}
\begin{center}
{\small THE} \vspace*{0.5cm}\\
{\large A M E R I C A N}\vspace*{0.5cm}\\
{\Large\bf JOURNAL OF SCIENCE}.\\

\vspace*{0.5cm}
{\scriptsize EDITORS}\\
{\normalsize JAMES D. and EDWARD S.\,DANA.}\\
\vspace*{0.5cm}
{\scriptsize  ASSOCIATE EDITORS}\\
{\scriptsize  PROFESSORS} {\small JOSIAH P.\,COOKE, GEORGE L.\, GOODALE}\\
{\scriptsize AND} {\small JOHN TROWBRIDGE}, {\scriptsize OF CAMBRIDGE}.\\
\vspace*{0.5cm}
{\scriptsize  PROFESSORS} {\small H.\,A.\,NEWTON}
{\scriptsize AND} {\small A.\,E.\,VERRILL}, {\scriptsize OF NEW HAVEN},\\
\vspace*{0.5cm}
{\scriptsize  PROFESSOR} {\small GEORGE F.\,BARKER}, {\scriptsize OF PHILADELPHIA}.\\
\vspace*{0.5cm}
---------------
\vspace*{0.5cm}\\
{\small\bf THIRD SERIES.}\\ \vspace*{0.3cm}
{VOL. XLVI. -- [WHOLE NUMBER, CXLVI.]}\\ \vspace*{0.3cm}
{\small\bf Nos. 271 - 276.}\\ \vspace*{0.3cm}
{JULY TO DECEMBER, 1893.}\\ \vspace*{0.3cm}
{\tiny WITH XI PLATES.}\\
\vspace*{0.5cm}
---------------
\vspace*{0.5cm}\\
{NEW HAVEN, CONN.: J.\,D. \& E.\,S.\,DANA.\\\vspace*{0.2cm}
1893}.
\end{center}
\vspace*{1cm}
---------------
\newline
{\scriptsize GEORGE FERDINAND BECKER, {\bf The Finite Elastic Stress-Strain Function}, The American Journal of Science, (1893), 337-356 
\vspace*{0.5cm}\newline
\noindent
This new typesetting in $\LaTeX$ by B. Sacha, R. Martin and P. Neff, Lehrstuhl f\"ur Nichtlineare Analysis und Modellierung; Fakult\"at f\"ur Mathematik, Universit\"at Duisburg-Essen, 45117 Essen, Deutschland; email: patrizio.neff@uni-due.de, September 2013} 
\emptythanks

\title{ART. XLVIII. -- \it{The Finite Elastic Stress-Strain Function;}}
\author{by GEO.\,F.\,Becker}

\markright{G. F. Becker -- Finite Elastic Stress-Strain Function.}

\maketitle

{\it Hooke's Law}. -- The law proposed by Hooke to account for the results of experiments on elastic bodies is equivalent to: -- Strain is proportionate to the load, or the stress initially applied to an unstrained mass\Blabel{becker:originalHooke}. The law which passes under Hooke's name is equivalent to: -- Strain is proportional to the final stress required to hold a strained mass in equilibrium.\footnote{Compare Bull. Geol. Soc. Amer. vol. iv, 1893, p. 38.} \Blabel{becker:alternativeHooke} It is now universally acknowledged that either law is applicable  only to strains so small that their squares are negligible\Blabel{becker:hookeRejection}. There are excellent reasons for this limitation. Each law implies that finite external forces may bring about infinite densities or infinite distortions\Blabel{becker:hookeSingularity}, while all known facts point to the conclusion that infinite strains result only from the action of infinite forces. When the scope of the law is confined to minute strains, Hooke's own law and that known as his are easily shown to lead to identical results; and the meaning is then simply that the stress-strain curve is a continuous one cutting the axes of no stress and of no strain at an angle whose tangent is finite. Hooke's law in my opinion rests entirely upon experiment, nor does it seem to me conceivable that any process of pure reason ``should reveal the character of the dependence of the geometrical changes produced in a body on the forces acting upon its elements.''\footnote{Saint-Venant in his edition of Clebsch. p. 39.}

{\it Purpose of this paper.} -- So far as I know no attempt has been made since the middle of the last century to determine the character of the stress-strain curve for the case of finite stress.\footnote{J.\,Riccati, in 1747, a brief account of whose speculation is given in Todhunter's history of elasticity, proposed a substitute for Hooke's law.} I have been unable to find even an analysis of a simple finite traction and it seems that the subject has fallen into neglect, for this analysis is not so devoid of interest as to be deliberately ignored, simple though it is.

\vspace*{1cm}
\pagebreak

In the first part of this paper finite stress and finite strain will be examined from a purely kinematical point of view; then the notion of an ideal isotropic solid\Blabel{becker:isotropy} will be introduced and the attempt will be made to show that there is but one function which will satisfy the kinematical conditions consistently with the definition. This definition will then be compared with the results of experiment and substantially justified.

In the second part of the paper the vibrations of sonorous bodies will be treated as finite and it will be shown that the hypothesis of perfect isochronism, or perfect constancy of pitch, leads to the same law as before, while Hooke's law would involve sensible changes of pitch during the subsidence of the amplitude of vibrations.

{\it Analysis of shearing stress.} -- Let $\Bresultant$, $\Bnormal$ and $\Btangent$ be the resultant, normal and tangential stresses at any point\Blabel{becker:resultantStress}. Then if $\Cauchy_{1}, \Cauchy_{2}$ and
$\Cauchy_{3}$ are the so-called principal stresses and $n_{1}, n_{2}, n_{3}$ the direction cosines of a plane\Blabel{becker:directionCosines}, there are two stress quadrics established by Cauchy which may be written
\[
\begin{array}{lcl}
{\Bresultant}^2 & = & \Cauchy_{1}^2\, n_{1}^2 + \Cauchy_{2}^2\, n_{2}^2 + \Cauchy_{3}^2\, n_{3}^2\,,\\[1ex]
 \Bnormal   & = & \Cauchy_{1}\, n_{1}^2 + \Cauchy_{2}\, n_{2}^2 + \Cauchy_{3}\, n_{3}^2\,.
\end{array}
\]
Since also $\Btangent^2 = \Bresultant^2-\Bnormal^2$,
\[
\Btangent^2 = (\Cauchy_{1}-\Cauchy_{2})^2\, n_{1}^2\, n_{2}^2 + (\Cauchy_{1}-\Cauchy_{3})^2\, n_{1}^2\, n_{3}^2 + (\Cauchy_{2}-\Cauchy_{3})^2\, n_{2}^2\, n_{3}^2\,;
\]
and these formulas include the case of finite stresses as well as of infinitesimal ones.

In the special case of a plane stress in the $xy$ plane, $\Cauchy_{3}=0$ and $n_{3} = 0$, and the formulas become
\[
\begin{array}{lcl}
\Bresultant^2 & = & \Cauchy_{1}^2\, n_{1}^2 + \Cauchy_{2}^2\, n_{2}^2\,,\\[1ex]
\Bnormal   & = & \Cauchy_{1}\, n_{1}^2 + \Cauchy_{2}\, n_{2}^2,\\[1ex]
\Btangent^2 & = & (\Cauchy_{1}-\Cauchy_{2})^2\, n_{1}^2\, n_{2}^2\,.
\end{array}
\]
In the particular case of a shear (or a $\it{pure}$ shear) there are two sets of planes on which the stresses are purely tangential, for otherwise there could be no planes of zero distortion\Blabel{becker:planesOfNoDistortionAndTangentialStrain}. On these planes ${\Bnormal}=0$, and if the corresponding value of $n_{1}/n_{2}$ is $\alpha$,
\[
-\,\Cauchy_{1}\, \alpha = \Cauchy_{2}/\alpha\,.
\]
If this particular quantity\Blabel{becker:cauchyStressShearFormula} is called $Q/3$, one may write the equations of stress in a shear for any plane in the form
\[
\begin{array}{lcl}
\Bresultant^2 & = & \displaystyle \frac{Q^2}{9}\left(n_{2}^2\, \alpha^2 + 
                    \frac{n_{1}^2}{\alpha^2}\right)\,,\\[2ex]
\Bnormal   & = & \displaystyle \frac{Q}{3}  \left(n_{2}^2\, \alpha -
                    \frac{n_{1}^2}{\alpha}\right)\,,\\[2ex]
\Btangent^2 & = & \displaystyle \frac{Q^2}{9}\left(\alpha +
                     \frac{1}{\alpha}\right)^2 n_{1}^2\, n_{2}^2\,.
\end{array}
\]
\newpage
For the axes of the shear the tangential stress must vanish, so that $n_{1}$ or $n_{2}$ must become zero, and therefore the axes of $x$ and $y$ are the shear axes. If ${\Bnormal}_{x}$ and ${\Bnormal}_{y}$ are the normal axial stresses, one then has
\[
-\,{\Bnormal}_{x}\alpha ={\Bnormal}_{y}/\alpha = Q/3\,.
\]

A physical interpretation must now be given to the quantity $\alpha$. In a finite shearing strain of ratio $\alpha$, it is easy to see that the normal to the planes of no distortion makes an angle with the contractile axis of shear the cotangent of which is $\alpha$.\;\Blabel{becker:planesOfNoDistortion} If the tensile axis of the shear is the axis of $y$, and the contractile axis coincides with $x$, this cotangent is $n_{1}/n_{2}$. Hence in the preceding formulas $\alpha$ is simply the ratio of shear.

In a shear of ratio $\alpha$ with a tensile axis in the direction of $oy$, minus ${\Bnormal}_{x}\alpha$ is the negative stress acting in the direction of the $x$ axis into the area $\alpha$ on which it acts\Blabel{becker:finalStressDefinition}. It is therefore the load or initial stress\Blabel{becker:initialStressDefinition} acting as a pressure in this direction. Similarly
${\Bnormal}_{y}/\alpha$ is the total load or initial stress acting as a tension or positively in the direction $oy$. Hence a simple finite shearing strain must result from the action of two equal loads or initial stresses of opposite signs at right angles to one another\Blabel{becker:shearAxiom} the common value of the loads being in the terms employed $Q/3$.\footnote{This proposition I have also deduced directly from the conditions of equilibrium in Bull.\,Geol.\,Soc.\,Amer., vol.\,iv, 1893, page 36. It may not be amiss here to mention one or two properties of the stresses in a shear which are not essential to the demonstration in view. The equation of the shear ellipse\Blabel{becker:shearEllipse} may be written in polar co\"ordinates $1/r^2=\alpha^2 n_{2}^2+\alpha^{-2}n_{1}^2$. Hence the resultant load on any plane whatever\Blabel{becker:resultantLoad} is ${\Bresultant}r=\pm Q/3$. The final tangential stress is well known to be maximum for planes making angles of $45^\circ$ with the axes\Blabel{becker:maximumTangentialStress}; but it is easy to prove that the tangential load, ${\Btangent}r$, is maximum for the planes of no distortion. These are also the planes of maximum tangential strain\Blabel{becker:maximumTangentialStrain}. Rupture by shearing is determined by maximum tangential {\it load}, not stress\Blabel{becker:ruptureCriterion}.}

It is now easy to pass to a simple traction in the direction of $oy$ since the principle of superposition is applicable to this case. Imagine two equal shears in planes at right angles to one another combined by their tensile axes in the direction $oy$, and let the component forces each have the value $Q/3$. To this system add a system of dilational forces acting positively and equally in all directions\Blabel{becker:volumetricStress} with an intensity $Q/3$. Then the sum of the forces acting in the direction of $oy$ is $Q$ and the sum of forces acting at right angles to $oy$ is zero.

Inversely a simple finite load or initial stress of value $Q$ is resoluble into two shears and a dilation, each axial component of each elementary initial stress being exactly one-third of the total load. Thus the partition of force in a finite traction is exactly the same as it is well known to be in an infinitesimal traction, provided that the stress is regarded as initial and not final.\footnote{Thomson and Tait, Nat.\,Phil., section 682.}

\newpage
\enlargethispage{2cm}

{\it Application to system of forces}. -- Without any knowledge of the relations between stress and strain, the foregoing analysis can be applied to developing corresponding systems of stress and strain. Let a unit cube\Blabel{becker:unitCube} of an elastic substance presenting equal resistance in all directions be subjected to axial loads $P, Q, R$. Suppose these forces to produce respectively dilations of ratios $h_{1},h_{2},h_{3}$ and shears of ratios $p,q,r$. Then the following table\Blabel{becker:table1} shows the effects of each axial force on each axial dimension of the cube in any pure strain.

\Boriginal{
\begin{table}[h]
\begin{center}
\begin{tabular}{l|ccc|ccc|ccc}\hline
Active force & \multicolumn{3}{|c}{P} & \multicolumn{3}{|c}{Q} & \multicolumn{3}{|c}{R}\\ \hline
Axis of strain & $x$ & $y$ & $z$ & $x$ & $y$ & $z$ & $x$ & $y$ & $z$\\ \hline
Dilation & $h_{1}$ & $h_{1}$ & $h_{1}$ & $h_{2}$ & $h_{2}$ & $h_{2}$ & $h_{3}$ & $h_{3}$ & $h_{3}$\\[1ex]
Shear & $p$ & $1/p$ & 1 & $1/q$ & $q$ & 1 & $1/r$ & 1 & $r$\\[1ex]
Shear & $p$ & 1 & $1/p$ & 1 & $q$ & $1/q$  & 1 & $1/r$ & $r$\\ \hline
\end{tabular}
\end{center}
\end{table}
}

Grouping the forces and the strains by axes, it is easy to see that the components may be arranged as in the following table, which exhibits the compound strains in comparison with the compound loads which cause them, though without in any way indicating the functional relation between any force and the corresponding strain.

\vspace*{0.5cm}
\begin{table}[h]\begin{center}{Pure Strains.\Blabel{becker:table2}}\\[1ex]
\begin{tabular}{l|c|c|c}\hline
Axes & x & y & z\\ \hline
Dilation$\upperPhant$ & $h_{1} h_{2} h_{3}$ & $h_{1} h_{2} h_{3}$ & $h_{1} h_{2} h_{3}$
\\[1ex]
Shear & $\displaystyle p^2\cdot\frac{1}{q}\cdot\frac{1}{r}$ &$\displaystyle\frac{1}{p^2}\cdot q\cdot r$  & 1 \\[2ex]
Shear & 1 & $\displaystyle\frac{pq}{r^2}$ & $\displaystyle\frac{r^2}{pq}$\\[2ex] \hline
Products & $\displaystyle\frac{\upperPhant h_{1}h_{2}h_{3}p^2}{qr}$ & 
             $\displaystyle\frac{h_{1}h_{2}h_{3}q^2}{pr}$ &
               $\displaystyle\frac{h_{1}h_{2}h_{3}r^2}{pq}$\\[2ex] \hline            
\end{tabular}
\end{center}
\end{table}

\begin{table}[h]\begin{center}{Loads or Initial Stresses.\Blabel{becker:table3}}\\[1ex]
\begin{tabular}{l|c|c|c}\hline
Axes & \multicolumn{1}{|c}{x} & \multicolumn{1}{|c}{y} & \multicolumn{1}{|c}{z}\\ \hline
Dilation  &$\displaystyle\frac{P+Q+R\upperPhant}{3}$  &$\displaystyle\frac{P+Q+R}{3}$ 
        & $\displaystyle\frac{P+Q+R}{3}$\\[2ex]
Shear  &$\displaystyle-\,\frac{Q+R-2P}{3}$  &$\displaystyle\frac{Q+R-2P}{3}$ 
        & 0 \\[2ex] 
Shear & 0 & $\displaystyle\frac{P+Q-2R}{3}$ & $\displaystyle-\, \frac{P+Q-2R}{3}$\\ [1.5ex]\hline
Sums$\upperPhant$ & $P$ & $Q$ & $R$\\ \hline          
\end{tabular}
\end{center}
\end{table}
\newpage

In many cases it is convenient to abbreviate the strain products. Thus if one writes $h_{1}h_{2}h_{3}= h$, $qr/p^2=\alpha$ and $pq/r^2=\beta$, the products are $h/\alpha,h\alpha\beta$ and $h/\beta$.

{\it Inferences from the table}. -- It is at once evident that the load sums correspond to the products of the strain ratios\Blabel{becker:superposition}, and that zero force answers to unit strain ratios. There are also several reciprocal relations which are not unworthy of attention. If $R=0$ and $Q=-P$, the strain reduces to a pure shear. But the positive force, say $Q$, would by itself produce a dilation $h_{2}$, while the negative force, minus $P$, would produce cubical compression of ratio $h_{1} < 1$. Now a shear is by definition undilatational and therefore, in this case, $h_{1}h_{2}=1$. Hence equal initial stresses of opposite signs produce dilatations of reciprocal ratios. The same two forces acting singly would each produce two shears while their combination produces but one. $Q$ would contract lines parallel to $oz$ in the ratio $1/q$ while minus $P$ would elongate the same lines in the ratio $p/1$. Since the combination leaves these lines unaltered, $p/q=1$. Hence equal loads of opposite signs produce shears of reciprocal ratios. It is easy to show by similar reasoning that equal loads of opposite signs must produce pure distortions and extensions of reciprocal ratios.

{\it Strain as a function of load} -- One may at will regard strains as functions of load\Blabel{becker:invertibility}\Bfootnote{load: force per unit area of the reference configuration, corresponding to the Biot stress} or of final stress\Bfootnote{final stress: force per unit area of the deformed configuration, corresponding to the Cauchy stress}; but there seem to be sufficient reasons for selecting load rather than final stress as the variable. To obtain equations giving results applicable to different substances, the equations must contain constants characteristic of the material as well as forces measured in an arbitrary unit. In other words the forces must be measured in terms of the resistance which any particular substance presents. Now these resistances should be determined for some strain common to all substances for forces of a given intensity. The only such strain is zero strain corresponding to zero force. Hence initial stresses or loads are more conveniently taken as independent variables.\footnote{When the strains are infinitesimal, it is easy to see that load and final stress differ from one another by an infinitesimal fraction of either.}\vspace*{0.5cm}

\begin{center}{\it Argument based on small strains.}\end{center}

{\it Physical hypothesis.} --- In the foregoing no relation has been assumed connecting stress and strain. The stresses and strains corresponding to one another have been enumerated, but the manner of correspondence has not been touched upon. One may now at least imagine a homogeneous elastic substance of such a character as to offer equal resistance to distortion in \newpage\noindent every direction and equal resistance to dilation in every direction. The two resistances may also be supposed independent of one another -- for this is a more general case than that of dependence. The resistance finally may be supposed continuous and everywhere of the same order as the strains.

In such an ideal isotropic substance it appears that the number of independent moduluses cannot exceed two; for a pure shear irrespective of its amount is the simplest conceivable distortion and no strain can be simpler than dilation, while to assume that either strain involved more than one modulus would be equivalent to supposing still simplier strains, each dependent upon one of the units of resistance. It is undoubtedly true that, unless the load-strain curve is a straight line, finite strains involve constants of which infinitesimal strains are independent; but these constants are mere coefficients and not moduluses: for the function being continuous must be developable by Taylor's Theorem\Blabel{becker:analyticity}, and the first term must contain the same variable as the succeeding terms, this variable being the force measured in terms of the moduluses. In this statement it must be understood that the moduluses are to be determined for vanishing strain.\footnote{One sometimes sees the incompleteness of Hooke's law referred to in terms such as ``Young's modulus must in reality be variable.'' This is a perfectly legitimate statement provided that Young's modulus is defined in accordance with it; but the mode of statement does not seem to me an expedient one to indicate the failure of linearity. Let
$\varYoung$ represent Young's modulus regarded as variable and $F$ a force or a stress measured in arbitrary units. Then if $y$ is the length of a unit cube when extended by a force, the law of extension may be written in the form $y=1+F/\varYoung$. Now let $\Young$ be the value of Young's modulus for zero strain, and therefore an absolute constant. Then, assuming the continuity of the functions, one may write $\varYoung$ in terms of $\Young$ thus,
\[
\frac{1}{\varYoung}=\frac{1}{\Young}\ \phi\left(\frac{F}{\Young}\right) = \frac{1}{\Young}\left(1 + \frac{AF}{\Young}+ \frac{BF^2}{\Young^2}+\ldots\,.\right)\,.
\]
But this gives
\[
y = 1 + \frac{F}{\Young} + \frac{AF^2}{\Young^2} + \frac{BF^3}{\Young^3 } + \ldots\ .
\]
so that $1/\varYoung$ merely stands for a development in terms of $F/\Young$. If therefore one defines Young's modulus as the tangent of the curve for vanishing strain, the fact of curvature is expressed by saying that powers of the force (in terms of Young's modulus) higher than the first enter into the complete expression for extension.}

One can determine the general form of the variable in terms of the resistances or moduluses for the ideal isotropic solid defined above. The load effecting dilation\Blabel{becker:dilationalStrain} in simple traction, as was shown above, is exactly one third of the total load, or say $Q/3$; and if $a$ is the unit of resistance to linear dilation, $Q/3a$ is the quantity with which the linear dilation will vary. The components of the shearing stresses in the direction of the traction are each $Q/3$, and, if $c$ is the unit of resistance to this initial stress, the\newpage\noindent corresponding extension will vary with $2Q/3c$. In simple extension all faces of the unit cube remain parallel to their original positions, and the principle of superposition is applicable throughout the strain. Hence the total variable may be written $Q\left(\frac{1}{3a}+\frac{2}{3c}\right)$. The intensity of $Q$ will not affect the values of the constants $a$ and $c$ which indeed should be determined for vanishing strain as has been pointed out.

The quantities $a$ and $c$ have been intentionally denoted by unusual letters. In English treatises it is usual to indicate the modulus of cubical dilation by $\bulk$ and the modulus of distortion by $\shear$. With this nomenclature $a=3\bulk$ and $c=2\shear$. Using the abbreviation $\Young$ for Young's modulus the variable then becomes\Blabel{becker:modulusConversion}
\[
Q\left(\frac{1}{9\bulk} + \frac{1}{3\shear}\right) = Q/\Young\,.
\]
Since this is the form of the variable whether $Q$ is finite or infinitesimal, the length of the strained cube according to the postulate of continuity must be developable in terms of $Q/\Young$ and cannot consist, for example, solely of a series of terms in powers of $Q/9\bulk$ plus a series of powers of $Q/3\shear$; in other words the general term of the development must be of the form $A_m(Q/\Young)^m$ and not $A_{m}(Q/9\bulk)^m+B_{m}(Q/3\shear)^m$.

{\it Form of the functions}. -- If $\alpha$ is the ratio of shear produced by the traction $Q$ in the ideal isotropic solid under discussion, $\alpha$ must be some continuous function of $Q/3\shear$. So too if $h$ is the ratio of linear dilation, $h$ is some continuous function of $Q/9\bulk$. The length of the strained mass\Blabel{becker:strainedMass} is $\alpha^2 h$, and this must be a continuous function of $Q/\Young$. If then $f,\varphi$ and $\psi$ are three unknown continuous functions, one may certainly write
\begin{equation}
\label{1}
\alpha^2 = f\left(\frac{Q}{3\shear}\right);\qquad h = \varphi\left(\frac{Q}{9\bulk}\right);\qquad
   \alpha^2 h = \psi\left(\frac{Q}{\Young}\right)\,.
\end{equation}
It also follows from the definitions of $\alpha$ and $h$ that
\begin{equation}
\label{2}
1 = f(0);\qquad 1 = \varphi(0);\qquad 1 = \psi(0)\,.
\end{equation}

For the sake of brevity let $Q/3\shear=\Nu-xi$ and $Q/9\bulk=\Kappa-eta$. Then $\Nu-xi$ and $\Kappa-eta$ may be considered algebraically as independent of one another even if an invariable relation existed between $\shear$ and $\bulk$; for since in simple traction, the faces of the isotropic cube maintain their initial direction, the principle of superposition is applicable; and to put $\shear=\infty$ or $\bulk=\infty$ is merely equivalent to considering only {\it that part} of a strain due respectively to
\begin{center}
\small{AM. JOUR. SCI. -- THIRD SERIES, VOL. XLVI, No. 275. -- Nov., 1893.}
\end{center}
\newpage\noindent
\enlargethispage{1.5cm}\noindent
compressibility or to pure distortion.\footnote{\scriptsize Compare Thomson and Tait Nat. Phil. section 179.} Now the functions are related by the equation\Blabel{becker:UNUSEDsuperposition}  
\begin{equation}
\label{3}
f(\Nu-xi)\,\varphi(\Kappa-eta) = \psi(\Nu-xi + \Kappa-eta)
\end{equation}
and if $\Nu-xi$ and $\Kappa-eta$ are alternately equated to zero
\[
f(\Nu-xi)=\psi(\Nu-xi)\quad \mbox{and}\quad \varphi(\Kappa-eta) = \psi(\Kappa-eta)\,.
\]
Hence the three functions are identical in form\footnote{\scriptsize This proposition is vital to the whole demonstration. Another way of expressing it is as follows: -- If the functions are continuous,
\[
\alpha^2h = 1 + A\left(\frac{Q}{3\shear}+\frac{Q}{9\bulk}\right) + B\left(\frac{Q}{3\shear} + \frac{Q}{9\bulk}\right)^2 + \ldots
\]
where $ A, B$, etc. are constant coefficients. Then since $\shear$ and $\bulk$ are algebraically independent, or since the principle of superposition is applicable, the development of $\alpha^2$ is found by making $h=1$ and $\bulk=\infty$. Thus
\[
\alpha^2 = 1 + \frac{AQ}{3\shear} + B\left(\frac{Q}{3\shear}\right)^2 + \ldots
\]
$A, B$, etc., retaining the same values as before. Consequently $\alpha^2$ is the same function of $Q/3\shear$ that $\alpha^2 h$ is of $Q/\Young$. By equating $\alpha$ to unity and $\shear$ to infinity, it appears that $h$ also is of the same form as $\alpha^2h$.

There is the closest connection between this method of dealing with the three functions and the principle, that when an elastic mass is in equilibrium, any portion of it may be supposed to become infinitely rigid and incompressible without disturbing the equilibrium. For to suppose that in the development of $\alpha^2h,\bulk=\infty$ is equivalent to supposing a system of external forces equilibrating the forces $Q/9\bulk$. This again is simply equivalent to assorting the applicability of the principle of superposition to the case of traction.

In pure elongation, unaccompanied by lateral contraction, it is easy to see that $h=\alpha$ and that $\alpha$ varies as $Q/6\shear$. In this case also $6\shear=9\bulk$ because Poisson's ratio is zero. Hence without resorting to the extreme cases of infinite $\shear$ or $\bulk$, it appears that $h$ is the same function of $9\bulk$ that $\alpha$ is of $6\shear$. This accords with the result reached in (\ref{5}) without sufficing to prove that result.\\[14mm]} or (\ref{3}) becomes
\begin{equation}
\label{4}
f(\Nu-xi)\,f(\Kappa-eta) = f(\Kappa-eta + \Nu-xi)\,.
\end{equation}
Developing the second member by Taylor's theorem and dividing by $f(\Kappa-eta)$ gives a value for $f(\Nu-xi)$, viz:
\[
f(\Nu-xi) = 1 + \frac{df(\Kappa-eta)}{d\Kappa-eta}\cdot \frac{1}{f(\Kappa-eta)} \cdot \Nu-xi + \ldots\,.
\]
Since the two variables are algebraically independent, this equation must answer to McLaurin's Theorem, which implies that the expression containing $\Kappa-eta$ is constant, its value being say $b$. Then
\[
\frac{df(\Kappa-eta)}{f(\Kappa-eta)} = b\,d\Kappa-eta\,.
\]
Hence since $f(0)=1$
\[
f(\Kappa-eta) = \e^{b\,\Kappa-eta}
\]
and since all three functions have the same form
\[
f(\Nu-xi + \Kappa-eta) = \e^{b\,(\Nu-xi+\Kappa-eta)}= 1 + b\,Q/\Young + \ldots
\]
\newpage
\noindent
Here $b/\Young$ is the tangent of the load-strain curve for vanishing strain, and this by definition is $1/\Young$, so that $b = 1$.

It appears then that the equations sought for the load-strain functions are
\begin{equation}
\label{5}
\alpha^2 = \e^{Q/3\shear}\,;\qquad h = \e^{Q/9\bulk}\,;\qquad \alpha^2 h = \e^{Q/\Young}\,;
\end{equation}
a result which can also be reached from (\ref{4}) without the aid of Taylor's theorem.

{\it Tests of the equations.} -- These equations seem to satisfy all the kinematical conditions deduced on preceding pages. It is evident that opposite loads of equal intensity give shears, dilations and extensions of reciprocal ratios and that the products of the strain ratios vary with the sums of the loads. It is also evident that infinite forces and such only will give infinite strains. A very important point is that these equations represent a shear as held in equilibrium by the same force system whether this elementary strain is due to positive or negative forces. If any other quantity (not a mere power of $Q$ or the sum of such powers), such as the final stress were substituted for the load $Q$, a pure shear would be represented as due to different force systems in positive and negative strains which would be a violation of the conditions of isotropy.\footnote{Let a shearing strain be held in equilibrium by two loads, $Q/3$ and minus $Q/3$. If a second equal shear at right angles to the first is so combined with it that the tensile axes coincide, the entire tensile load is $2Q/3$. If on the other hand the two shears are combined by their contractile axes, the total pressure is $2Q/3$. In the first case the area of the deformed cube measured perpendicularly to the direction of the tension is $1/\alpha^2$, and if $Q'$ is the final stress, $Q'/\alpha^2 = 2Q/3$ or $Q' = 2Q\alpha^2/3$. In the second case the area on which the pressure acts is $\alpha^2$ and if the stress is $Q', Q'' = - 2Q/3\alpha^2.$ Thus $Q' = - Q'' \alpha^4$. Hence equal final stresses of opposite signs cannot produce shears of reciprocal ratios in an isotropic solid. The same conclusion is manifestly true of any quantity excepting $Q$ or an uneven power of $Q$ or the sum of such uneven powers.} One might suppose more than two independent moduluses to enter into the denominator of the exponent; but this again would violate the condition of isotropy by implying different resistances in different directions. Any change in the numerical coefficients of the moduluses would imply a different partition of the load between dilation and distortion, which is inadmissible. It would be consistent with isotropy to suppose the exponent of the form $(Q/\Young)^{1+2c}$; but then, if $c$ exceeds zero, the development of the function would contain no term in the first power of the variable and the postulate that strains and loads are to be of the same order would not be fulfilled. The reciprocal relations of load and strain would be satisfied and the loads would be of the same order as the strains, if one were to substitute a series of uneven powers of the variables for $\Nu-xi$ and $\Kappa-eta$. Such series are for example the developments of $\tan\Nu-xi$ and $\tan \Kappa-eta$.

\newpage
\noindent
In a case of this kind, however, $\alpha^2 h$ would not be a function of $Q/\Young = \Nu-xi + \Kappa-eta$ excepting for infinitesimal strain; the exponent then taking the form of a series of terms
$A_{m}\left(\Nu-xi^m + \Kappa-eta^m\right)$ instead of $A_{m}(\Nu-xi + \Kappa-eta)^m$. Finally it is conceivable that the expanded function should contain in the higher terms moduluses not appearing in the first variable term; but this would be inconsistent with continuity. In short I have been unable to devise any change in the functions which does not conflict with the postulate of isotropy as defined or with some kinematical condition.

{\it Abbreviation of proof.} -- In the foregoing the attempt has been made to take a broad view of the subject in hand lest some important relation might escape attention. Merely to reach the equations (\ref{5}) only the following steps seem to be essential. Exactly one-third of the external initial stress in a simple traction is employed in dilation, and of the remainder one half is employed in each of the two shears. An ideal isotropic homogeneous body is postulated as a material presenting equal resistance to strain in all directions, the two resistances to deformation and dilation being independent of one another; the strains moreover are to be of the same order as the loads, and continuous functions of them. In such a mass the simplest conceivable strains, shear and dilation, can each involve only a single unit of resistance or modulus. The principle of superposition is applicable to a simple traction applied axially to the unit cube however great the strain. It follows that the length of the strained unit cube is a function of $Q/\Young$.

Together these propositions and assumptions give (\ref{1}) and without further assumptions the final equations sought (\ref{5}) follow as a logical consequence.

{\it Data from experiment.} -- No molecular theory of matter is essential to the mechanical definition of an isotropic substance. An isotropic homogeneous body is one a sphere of which behaves to external forces of given intensity and direction in the same way however the sphere may be turned about its center. There may be no real absolutely isotropic substance, and if there were such a material we could not ascertain the fact, because observations are always to some extent erroneous. It is substantially certain, however, that there are bodies which approach complete symmetry so closely that the divergence is insensible or uncertain. Experience therefore justifies the assumption of an isotropic substance as an approximation closely representing real matter.

All the more recent careful experiments, such as those of Amagat and of Voigt, indicate that Cauchy's hypothesis, leading for isotropic substances to the relation $3\bulk = 5\shear$, is very far from being fulfilled by all substances\newpage\noindent of sensibly symmetrical properties. This is substantially a demonstration that the molecular constitution of matter is very complex,\footnote{Compare Lord Kelvin's construction of the system of eight molecules in a substance not fulfilling Poisson's hypothesis in his Lectures on Molecular Dynamics.} but provided that the mass considered is very large relatively to the distances between molecules this complexity does not interfere with the hypothesis that pure shear and simple dilation can each be characterized by one constant only.

The continuity of the load-strain function both for loads of the same sign and from positive to negative loads is regarded as established by experiment for many substances; and equally well established is the conclusion that for small loads, load and strain are of the same order.\footnote{Compare B.\,de Saint-Venant in his edition  of Navier's Le\c cons. 1864. p. 14, and Lord Kelvin, Encyc. Brit. 9th ed. Art. Elasticity, Section 37.} In other words Hooke's law is applicable to minute strains. Perfect elastic recovery is probably never realized, but it is generally granted that some substances approach this ideal under certain conditions so closely as to warrant speculation on the subject.

These results appear to justify the assumptions made in the paragraph headed ``physical hypothesis'' as representing the most important features of numerous real substances. On the other hand viscosity, plasticity and ductility have been entirely ignored; so that the results are applicable only to a part of the phenomena of real matter.

{\it Stress-strain function.} -- It is perfectly easy to pass from the load-strain function to the stress-strain function for the ideal solid under discussion. The area of the extended cube is its volume divided by its length of $h^3/\alpha^2 h$. Hence if $Q'$ is the stress, or force per unit area, $Q'h^2/\alpha^2 = Q$. Therefore the stress-strain function is
\[
(\alpha^2 h)^{\alpha^2/h^2} = \e^{Q'/\Young}
\]
an equation which though explicit in respect to stress and very compact is not very manageable. If one writes $\alpha^2 h = y$ and $h/\alpha = x$, the first member of this equation becomes $y^{1/x^2}$. Here $x$ and $y$ are the co\"ordinates of the corner of the strained cube.

{\it Verbal statement of law.} -- If one writes $\alpha^2 h-1 = f$, the last of equations (\ref{5}) gives
\[
df = (1 + f)\,d[Q/\Young]
\]
or the increment of strain is proportional to the increment of load and to the length of the strained mass. This is of course the ``compound interest law'' while Hooke's law answers to simple interest.

\newpage
\enlargethispage{0.2cm}
{\it Curves of absolute movement.} -- Let $\Poisson$ be Poisson's ratio

\begin{minipage}{0,1 \textwidth}
or
\end{minipage}
\begin{minipage}{0,75 \textwidth}
\[
\Poisson = \frac{3\bulk-2\shear}{2(3\bulk+\shear)}\,.
\]
\end{minipage}

Let $x_{0}\,y_{0}$ be the original positions of a particle in an unstrained bar, and let $xy$ be their positions after the bar has been extended by a load $Q$. Then $x = x_{0}h/\alpha$ and $y = y_{0}\,\alpha^2 h$. It also follows from (\ref{5}) that $\alpha^{6\shear}=h^{9\bulk}$, whence it may easily be shown that the path of the particle is represented by the extraordinarily simple equation\footnote{On Cauchy's hypothesis\Blabel{becker:cauchyHypothesis} $\Poisson = 1/4$, which, introduced into this equation, implies that the volume of the strained cube is the square root of its length.}
\begin{equation}
\label{6}
x\,y^\Poisson = x^{}_{0}\,y^\Poisson_{0}\,.
\end{equation}

If one defines Poisson's ratio as the ratio of lateral contraction to axial elongation, its expression is by definition
\[
\Poisson = - \frac{dx}{x}\big{/}\frac{dy}{y} = - \frac{y}{x}\,\frac{dx}{dy}\,;
\]
and this, when integrated on the hypothesis that $\Poisson$ is a constant, gives (\ref{6}). Thus for this ideal solid, the ratio of lateral contraction to linear elongation is independent of the previous strain.

The equation (\ref{6}) gives results which are undeniably correct in three special cases. For an incompressible solid $\Poisson = 1/2$, and (\ref{6}) becomes $x^2 y =$ constant, or the volume remains unchanged. For a compressible solid of infinite rigidity $\Poisson = -1$ and (\ref{6}) becomes $x/y =$ constant so that only radial motion is possible. For linear elongation unaccompanied by lateral extension $\Poisson = 0$, and (\ref{6}) gives $x =$ constant.\footnote{It seems possible to arrive at the conclusion that $\Poisson$ is constant by discussion of these three cases. Let $g$ and $-f$ be small axial increments of strain due to a small increment of traction applied to a mass already strained to any extent. Let it also be supposed that the moduluses are in general functions of the
co\"ordinates, so that $\shear$ and $\bulk$ are only limiting values for no strain. Then, by the ordinary analysis of a small strain (Thomson and Tait, section 682), one may at least write for an isotropic solid
\[
\begin{array}{rcl}
 g & = & \displaystyle P\left(\frac{1}{3\shear\,[1+f_{1}(x)]} + \frac{1}{9\bulk\,[1+f_{2}(x)]}\right)\,,\\ [4ex] 
-f & = & \displaystyle P\left(\frac{1}{6\shear\,[1+f_{3}(x)]} - \frac{1}{9\bulk\,[1+f_{4}(x)]}\right)\,,
\end{array}
\]
where $f(x)$ is supposed to disappear with the strain. These values represent each element of the axial extension and each element of the lateral contraction as wholly independent. The value of $\Poisson$ is $-f/g$. Now for an incompressible substance, as mentioned in the text, $\Poisson = 1/2$ and the formula gives
\[
\Poisson = \frac{1}{2} \cdot \frac{1+f_{1}(x)}{1+f_{3}(x)}\,,\quad\mbox {so that}\ f_{1}(x)=f_{3}(x)\,. 
\]
Again for $\shear = \infty$ only dilation is possible, or $\Poisson = -1$, while the formula gives
\vspace{1cm}
\[
\Poisson = -\,1 \cdot \frac{1+f_{2}(x)}{1+f_{4}(x)}\ ,\quad \mbox{so that}\ f_{2}(x)=f_{4}(x).
\]
For pure elongation the lateral contraction is by definition zero, or $\Poisson = 0$, and the formula is
\[
\Poisson = \frac{f_{3}(x)-f_{4}(x)}{2[1+f_{2}(x)]+[1+f_{1}(x)]}\ ,\quad \mbox{whence}\ f_{4}(x)=f_{3}(x)\,.
\]
Hence all four functions of $x$ are identical and $\Poisson$ reduces to its well known constant-form. --- With $\Poisson$ as a constant equation (\ref{6}) follows from the definition of $\Poisson$; and substituting $\alpha^2 h=y/y_{0}$ and $h/\alpha=x/x_{0}$ gives $\alpha^{6\shear}=h^{9\bulk}$. If $W=6\,\shear\log\,\alpha$ one may then write
\[
\alpha = \e^{W/6\shear}\,;\quad h = \e^{W/9\bulk}\,;\quad \alpha^2 h = \e^{W/\Young} = 1 + W/\Young + \ldots\,.
\]
Here experiment shows that $W$ may be regarded either as load or stress: and reasoning indicates that it must be considered as load if $\Young$ is determined for vanishing strain.}

\vspace*{0.5cm}

\begin{center}{\it Argument from finite vibrations}\end{center}

{\it Sonorous vibrations finite.} --- In the foregoing pages the attempt has been made to show, that a certain definition of an isotropic solid in combination with purely kinematical propositions leads to a definite functional expression for the load-strain curve. The definition of an isotropic solid is that usual except among elasticians who adhere to the rariconstant hypothesis, and it seems to be justified by experiments on extremely small strains. But the adoption of this definition for bodies under finite strain is, in a sense, extrapolation. It is therefore very desirable to consider the phenomena of such strains as cannot properly be considered infinitesimal.

It is usual to treat the strains of tuning forks and other sonorous bodies as so small that their squares may be neglected, and the constancy of pitch of a tuning fork executing vibrations of this amplitude has been employed by Sir George Stokes to extend the scope of Hooke's law to moving systems. It does not appear legitimate, however, to regard strongly excited sonorous bodies as only infinitesimally strained. Tuning forks sounding loud notes perform vibrations the amplitudes of which are sensible fractions of their length. Now it is certain that no elastician would undertake to give results for the strength of a bridge, or in other words he would deny that such flexures were so small as to justify neglect of their squares.\footnote{It is scarcely necessary to point out that many of the uses to which springs are put, in watches for example, afford excellent evidence of the continuity of the load-strain function for finite distortions.}

{\it Sonorous vibrations isochronous.} --- The vibrations of sonorous bodies seem to be perfectly isochronous, irrespective of the amplitude of vibration. Were this not the case, a tuning-fork strongly excited would of course sound a different note from that which it would give when feebly excited. Neither\newpage\noindent musicians nor physicists have detected any such variation of pitch which, if sensible, would render music impossible. The fact that the most delicate and accurate microchronometrical instruments yet devised divide time by vibrations of forks, is an additional evidence that these are isochronous. Lord Kelvin has even suggested the vibrations of a spring in a vacuum as a standard of time almost certainly superior to the rotation of the earth, which is supposed to lose a few seconds in the course of a century.\footnote{Nat.\,Phil., sections 406 and 830.}

It is therefore a reasonable hypothesis in the light of experiment that the load strain function is such as to permit of isochronous vibrations; but to justify this conclusion from an experimental point of view, it must also be shown that Hooke's law is incompatible with sensibly isochronous vibration. I shall therefore attempt to ascertain what load-strain function fulfills the condition of perfect isochronism (barring changes of temperature) and then to make a quantitative comparison between the results of the law deduced and those derived from Hooke's law.

{\it Application of moment of momenta.} --- If the cube circumscribed about the sphere of unit radius is stretched by opposing initial stresses and then set free, it will vibrate; and the plane through the center of inertia perpendicular to the direction of the stress will remain fixed. Each half of the mass will execute longitudinal vibrations like those of a rod of unit length fixed at one end, and it is known that the cross section of such a rod does not affect the period of vibration, because each fiber parallel to the direction of the external force will act like an independent rod. Hence attention may be confined to the unit cube whose edges coincide with the positive axes of
co\"ordinates, the origin of which is at the center of inertia of the entire mass.

The principle of the moment of momenta is applicable to one portion of the strain which this unit cube undergoes during vibration. The moment of a force in the $xy$ plane relatively to the axis of $oz$, being its intensity into its distance from this axis, is the moment of the tangential component of the force and is independent of the radial force component. Now dilation is due to radial forces and neither pure dilation nor any strain involving dilation can be determined by discussion of the moments of external forces. Hence the principle of the moment of momenta applies only to the distortion of the unit cube. This law as applied to the $xy$ plane consequently governs only the single shear in that plane.

The principle of the moment of momenta for the $xy$ plane may be represented by the formula

\newpage

\begin{equation}
\label{7}
\frac{d}{dt}\,\Sigma m\left(x_{1}\frac{dy_{1}}{dt} - y_{1}\frac{dx_{1}}{dt}\right) =
     \Sigma\,(x_{1} Y - y_{1}X)\,,
\end{equation}
where the second member expresses the moments of the external forces, which are as usual measured per unit area, and $x_{1}\,y_{1}$ are the co\"ordinates of any point the mass of which is $m$.

{\it Reduction of equation (\ref{7}).} --- Let $x$ and $y$ represent the position of the corner of the strained cube; then the abscissa of the center of inertia of the surface on which the stress $Y$ acts is $x/2$, and since $Y$ is uniform, $\Sigma x_{1}\,Y=xY/2$. Similarly $\Sigma y_{1}X=yX/2$. Now
$xY$ and $yX$ may also be regarded as the loads or initial stresses acting on the two surfaces of the mass parallel respectively to $ox$ and $oy$, and in a shear these two loads are equal and opposite. Hence the second member of (\ref{7}) reduces to $xY$. It has been shown above that, if $Q$ is an initial tractive load, $Q/3$ is the common value of the two equal and opposite loads producing one shear. But to obtain comparable results for shear dilation and extension, $Q/3$ must be measured in appropriate units of resistance. Since $\Young$ is the unit of resistance appropriate to extension, the separate parts of the force must be multiplied by $\Young$ and divided by resistances characteristic of the elementary strains. Now
\[
\frac{\Young}{2\shear} \cdot \frac{Q}{3} + \frac{\Young}{2\shear} \cdot \frac{Q}{3} + \frac{\Young}{3\bulk} \cdot \frac{Q}{3}
      = Q\,,
\]
and it is evident that $2\shear/\Young$ is the unit in which $Q/3$ should be measured for the single shear.\footnote{In this paper changes of temperature are expressly neglected. The changes of temperature produced by varying stress in a body performing vibrations of small amplitude can be allowed for by employing ``kinetic'' moduluses, which are a little greater than the ordinary ``static'' moduluses. Thomson and Tait, Nat.\,Phil., section 687.} Thus the second member of (\ref{7}) becomes $\Young Q/6\shear$.

This, then, is the value which the moment of the external forces assumes when these hold the strained unit cube in equilibrium. This unit cube forms an eighth part of the cube circumscribed about the sphere of unit radius. When the entire mass is considered, the sum of all the moments of the external forces is zero; since they are equal and opposite by pairs. If the entire mass thus strained is suddenly released and allowed to perform free vibrations, the sum of all the moments of momenta will of course remain zero. On the other hand the quantity $\Young Q/6\shear$ will remain constant. For this load determines the limiting value of the strain during vibration and is independent of the particular phase of vibration, or of the time counted from the instant of release. It may be considered as the moment of the forces which the other parts of the entire material system exert upon the unit cube. 

\newpage

Turning now to the first member of (\ref{7}), values of $x_{1}$ and $y_{1}$ appropriate to the case in hand must be substituted. Each point of the unit cube during shear moves on an equilateral hyperbola, so that if $x_{0}, y_{0}$ are the original co\"ordinates of a point, $x_{1}\,y_{1} = x_{0}\,y_{0}$. For the corner of the cube, whose co\"ordinates are $x$ and $y$, the path is $x\,y = 1$. Now $x_{1}/x_{0} = x$ and $y_{1}/y_{0}= y$ so that
\[
x_{1}\,dy_{1} - y_{1}\,dx_{1} = x_{0}\,y_{0}\,(xdy - ydx)
\]

If $\psi$ is the area which the radius vector of the point $x,y$ describes during strain, it is well known that $2d\psi = x\,dy - y\,dx$ and, since in this case $x\,y = 1$, it is easy to see that
\[
2d\psi = 2d\,[\log y]\,.
\]
Since the quantities $x$ and $y$ refer to a single point, the sign of summation does not affect them, and the first member of (\ref{7}) may be written
\[
\frac{d^2 \log y}{dt^2}\,\Sigma\,2m\,x_{0}\,y_{0}\,.
\]

Here one may write for $m, \rho dx_{0}\,dy_{0}$, where $\rho$ is the constant density of the body; and since the substance is uniform, summation may be performed by double integration between the limits unity and zero. This reduces the sum to $\rho/2$.

{\it Value of $\alpha$.} Equation (\ref{7}) thus becomes
\[
\frac{d^2 \log y}{dt^2}= \frac{2}{\rho}\cdot \frac{\Young Q}{6\shear}
\]
the second member being constant. Counting time from the instant of release, or from the greatest strain, and integrating $y$ between the limits $y = \alpha$ and $y = 1$ gives
\[
\log \alpha = \frac{\Young Q}{6\shear}\cdot \frac{t^2}{\rho}\ .
\]

It is now time to introduce the hypothesis that the vibrations are isochronous. It is a well known result of theory and experiment that a rod of unit length with one end fixed, executing its gravest longitudinal vibrations, performs one complete vibration of small amplitude in a time expressed by $4\sqrt{\rho/\Young}$. In the equation stated above $t$ expresses the time of one-quarter of a complete vibration or the interval between the periods at which $y = 1$ and $y = \alpha$. Hence for a small vibration, $t$ as here defined is $\sqrt{\rho/\Young}$. If the vibrations are to be isochronous irrespective of amplitude, this must also be the value of $t$ in a finite vibration. Hence at once
\[
\alpha = \e^{Q/6\shear} = \e^\psi\,,
\]
the same result reached in (\ref{5}).

\newpage

This result may also be expressed geometrically. The quantity $Q/6\shear$ is simply the area swept by the radius vector of the point $x_{0}=1, y_{0}=1$. This area is also the integral of $ydx$
from $x=1/\alpha$ to $x=1$, or the integral of $xdy$ from $y=\alpha$ to $y=1$. Thus $\psi$ represents any one of three distinct areas. In terms of hyperbolic functions, $\alpha=\sinh\psi+\cosh\psi$ and the amount of shear is $2\sinh \psi$.

It appears then that isochronous vibrations imply that in pure shear the area swept by the radius vector of the corner of the cube, or $\log \alpha$, is simply proportional to the load. The law proposed by Hooke implies that the length $\alpha-1$ is proportional to the same load. The law commonly accepted as Hooke's makes $\alpha-1$ proportional to the final stress, or $(\alpha-1)/\alpha$ proportional to the load.

{\it Value of $h$.} --- Knowing the value of $\alpha$, the value of $h$ can be found without resort to the extreme case $\shear=\infty$. In the case of pure elongation, unattended by lateral contraction, $h=\alpha$ and $9\bulk=6\shear$. If $\alpha_{1}$ and $h_{1}$ are the ratios for this case,
\[
\alpha_{1} = \e^{Q/9\bulk}\,;\qquad h_{1} = \e^{Q/9\bulk}\,;\qquad \alpha_{1}^2\,h_{1}= \e^{Q/3\bulk}\,.
\]
If three such elongations in the direction of the three axes are superimposed, the volume becomes
\[
(\alpha_{1}^2\,h_{1})^3 = \e^{Q/\bulk}\,,
\]
and this represents a case of pure dilation without distortion. Here however $\alpha_{1}=h_{1}$ and therefore the case of no distortion, irrespective of the value of $\shear$, is given by
\[
h^9 = \e^{Q/\bulk}\ .
\]

The values of $\alpha$ and $h$ derived from the hypothesis of isochronous vibrations when combined evidently give the same value of $\alpha^2\,h$ which was obtained from kinematical considerations and the definition of isotropy in equation (\ref{5}).

{\it Law of elastic force.} --- Let $s$ be the distance of a particle on the upper surface of a vibrating cube from its original position or
\[
s = \alpha^2\,h - 1 = \e^{Q/\Young} - 1\, .
\]
Then the elastic force per unit volume is minus $Q$, or
\[
\rho\,\frac{d^2s}{dt^2} = -\,Q = -\,\Young\log(s+1) = -\,\Young s + \frac{\Young s^2}{2} - \ldots\,.
\]
When the excursions of the particle from the position of no strain are very small, this becomes
\[
\rho\,\frac{d^2s}{dt^2} = -\,\Young s
\]
a familiar equation leading to simple harmonic motion.

\newpage

{\it Limitation of harmonic vibrations.} -- While the theory of harmonic vibrations is applicable to very small vibrations on any theory in which the load strain curve is represented as continuous and as making an angle with the axes whose tangent is finite, it appears to be inapplicable in all cases where the excursions are sufficient to display the curvature of the locus. If the attraction toward the position of no strain in the direction of $oy$ is proportional to $y-1$, then in an isotropic mass there will also be an attraction in the direction of $ox$ which will be proportional to $1-x$. The path of the particle at the corner of a vibrating cube will therefore be the resultant of two harmonic motions whose phases necessarily differ by exactly one-half of the period of vibration, however great and however different the amplitudes may be. This resultant is well known to be a straight line. Hence the theory precludes all displacements excepting those which are so small that the path of the corner of the cube may properly be regarded as rectilinear. It seems needless to insist that such cannot be the case for finite strains in general.

There is at least one elastic solid substance, vulcanized india rubber, which can be stretched to several times its normal length without taking a sensible permanent set. Now if the ideal elastic solid stretched to double its original length (or more) were allowed to vibrate, the hypothesis of simple harmonic vibration implies that this length would be reduced to zero (or less) in the opposite phase of the vibration, a manifest absurdity.

{\it Variation of pitch by Hooke's law.} -- It remains to be shown that if the commonly accepted law were applicable to finite strain, sonorous vibrations would be accompanied by changes of pitch which could scarcely have escaped detection by musicians and physicists. Experiments have shown that the elongation of steel piano wire may be pushed to $0\cdot 0115$ before the limit of elasticity is reached.\footnote{From experiments on English steel piano wire by Mr.\,D.\,McFarlane.} Since virtuosos not infrequently break strings in playing the piano, it is not unreasonable to assume that a one per cent elongation is not seldom attained. In simple longitudinal vibration the frequency of vibration is expressed by $1/4$ of $\sqrt{\Young/\rho}$, and if according to Hooke's law, $s = Q/\Young$, where $Q$ is the load, the number of vibrations, $v$, may be written
\[
v = \frac{1}{4}\,\sqrt{\frac{Q}{s\,\rho}}\,.
\]
If, on the other hand, according to the theory of this paper, $\log(1+s) = Q/\Young$ the number of vibrations, $u$, may be written

\newpage

\[
u = \frac{1}{4}\,\sqrt{\frac{Q}{\rho\log(1+s)}}\,,\quad \mbox{so that}\ \frac{v}{u} =
    \sqrt{\frac{\log(1+s)}{s}}\,.
\]
If $s = 0\cdot 01$, this expression gives $v/u=400/401$.

It would appear then that on the hypothesis of Hooke, a note due to longitudinal vibrations of about the pitch $G_{3}$ would give a lower note when sounding fortissimo than when sounding pianissimo, and that the difference would be one vibration per second, or one in four hundred. But according to Weber's experiments experienced violin players distinguish musical intervals in melodic progressions no greater than $1000/1001$, while simultaneous tones can be still more sharply discriminated.\footnote{Helmholtz, Tonempfindungen, page 491.} The value of $s$ corresponding to $v/u = 1000/1001$ is only $0\cdot 004$, and consequently strains reaching only about one-third of the elastic limit of piano wire should give sensible variations of tone during the subsidence of vibrations if Hooke's law were correct.

Longitudinal vibrations are not so frequently employed to produce notes as transverse vibrations. The quantity $\Young/\rho$ enters also into the expression for the frequency of transverse vibrations though in a more complex manner. In the case of rods not stretched by external tension, the ratio $v/u$ would take the same form as in the last paragraph. One theory of the tuning-fork represents it as a bar vibrating with two nodes, and therefore as comparable to a rod resting on two supports.

A pair of chronometrical tuning-forks could be adjusted to determine much smaller differences in the rate of vibration than $1000/1001$; for the relative rate of the forks having been determined on a chronographic cylinder for a certain small amplitude, one fork could be more strongly excited than the other and a fresh comparison made. The only influences tending to detract from the delicacy of this method of determining whether change of amplitude alters pitch, would seem to be the difficulty of sustaining a constant amplitude and the difference of temperature in the two forks arising from the dissipative action of viscosity.

{\it Conclusion.} -- The hypothesis that an elastic isotropic solid of constant temperature is such as to give absolutely isochronous longitudinal vibrations leads to the conclusion $\log(\alpha^2\,h)=Q/\Young$ without any apparent alternative. Comparison with the results of Hooke's law shows that, if this law were applicable to finite vibrations, easily sensible changes of pitch would occur during the subsidence of vibrations in strongly excited sonorous bodies. -- The logarithmic law is the same deduced in the earlier part of the paper from the ordinary definition of the ideal elastic isotropic solid, based \newpage \noindent upon experiments on very smalls strains, in combination with purely kinematical considerations. -- There can be no doubt that the law here proposed would simplify a great number of problems in the dynamics of the ether and of sound, as well as questions arising in engineering and in geology, because of the simple and plastic nature of the {\it logarithmic} function. In the present state of knowledge, the premises of the argument can scarcely be denied; whether the deductions have been logically made must be decided finally by better judges than myself.

\vspace*{4mm}
\small{Washington, D.\,C., July, 1893.}

\newpage
\thispagestyle{empty}
\noindent{\bf Incorporated changes to the original text}\\
\\
\begin{tabular}{lccl}
Young's modulus              &        M      & $\longrightarrow$ & $E$\\[2ex]
variable Young's modulus		&			$\mu$		&	$\longrightarrow$	&	$\varYoung$\\[2ex]
Poisson's number             &     $\sigma$   & $\longrightarrow$ & $\nu$\\[2ex]
modulus of distortion (shear modulus)       &        n      & $\longrightarrow$ & $G$\\[2ex] 
modulus of cubical dilation (bulk modulus) &        k      & $\longrightarrow$ & $K$\\[2ex]
Euler's number				& $\varepsilon$ & $\longrightarrow$ & $e = 2.718\ldots$\\[2ex]
Variables 					&      $\nu,\kappa$		& $\longrightarrow$ & $\xi,\eta$\\[2ex]	
natural logarithm			     &      $\ln$    & $\longrightarrow$	& $\log$\\[2ex]
the direction cosines of a plane	&		$\lambda,\mu,\nu$		& $\longrightarrow$	& $n_{1},n_{2},n_{3}$\\[2ex]
the amounts of traction			&	$\mathfrak{R},\mathfrak{N},\mathfrak{T}$	& $\longrightarrow$	& $\Bresultant,\Bnormal,\Btangent$\\[2ex]
the principle stresses		&		$N_{1},N_{2},N_{3}$		& $\longrightarrow$  & $\Cauchy_{1},\, \Cauchy_{2},\, \Cauchy_{3}$\\[2ex]
hyperbolic functions						 & Sin, Cos     & $\longrightarrow$   &$\sinh, \cosh$\\[2ex]
axial increment of strain						 & $e$     & $\longrightarrow$   &$g$\\[2ex]
formula for Poisson's number	&	$\nu = \frac{3K-2G}{2(3K+2G)}$&$ \xlongrightarrow{\text{corrected}}$&$ \frac{3K-2G}{2(3K+G)}$\\
on page 348
\end{tabular}\\[4mm]

\end{document}